\documentclass[a4paper]{amsart}

\usepackage{amsmath,amsthm,amssymb,enumerate}
\usepackage{epsfig}
\usepackage{amssymb}
\usepackage{amsmath}
\usepackage{amssymb}
\usepackage{amsmath,amsthm}
\usepackage[latin1]{inputenc}
\usepackage[T1]{fontenc}
\usepackage{path}
\usepackage{ae,aecompl}
\usepackage{amsfonts}
\usepackage{amsxtra}
\usepackage{euscript,mathrsfs}
\usepackage{color}
\usepackage[left=3.4cm,right=3.4cm,top=4cm,bottom=4cm]{geometry}
\usepackage[citecolor=blue,colorlinks=true]{hyperref}
\allowdisplaybreaks

\usepackage{enumitem}
\setenumerate{label={\rm (\alph{*})}}

\usepackage{amsfonts}
\usepackage{amsxtra}

\numberwithin{equation}{section}

\newtheorem{theorem}{Theorem}[section]

\newtheorem{lemma}{Lemma}[section]

\newtheorem{corollary}{Corollary}[section]

\newtheorem{remark}{Remark}[section]

\newtheorem{definition}{Definition}[section]

\newcommand{\bfu}{u}

\newcommand{\N}{\mathbb N}

\newcommand{\E}{\mathbb E}
\newcommand{\p}{\mathbb P}

\newcommand{\F}{\mathfrak F}

\newcommand{\dd}{\mathrm d}
\newcommand{\dx}{\, \mathrm{d}x}

\newcommand{\ds}{\, \mathrm{d}s}
\newcommand{\dt}{\, \mathrm{d}t}
\newcommand{\dxt}{\,\mathrm{d}x\, \mathrm{d}t}

\newcommand{\dif}{\mathrm{d}}

\newcommand{\mf}{\mathfrak{F}}
\newcommand{\mr}{\mathbb{R}}
\newcommand{\prst}{\mathbb{P}}

\newcommand{\mt}{\mathbb{T}^2}

\DeclareMathOperator{\diver}{div}

\usepackage{amsmath,amssymb,amsfonts,amsthm,mathtools}
\usepackage{bm}
\usepackage{enumitem}
\usepackage{hyperref}
\hypersetup{colorlinks=true,linkcolor=blue,citecolor=blue,urlcolor=blue}
\usepackage{xcolor}
\allowdisplaybreaks
\usepackage{float}          
\usepackage{algorithm}      
\usepackage{algpseudocode}  
\usepackage{xcolor}
\usepackage{mdframed}
\usepackage[most]{tcolorbox}

\newmdenv[
  linewidth=0.8pt,
  linecolor=black,
  backgroundcolor=gray!5,
  roundcorner=4pt,
  skipabove=10pt,
  skipbelow=10pt
]{scheme}

\tcbset{
  authorcomment/.style={
    enhanced,
    colback=blue!3,
    colframe=blue!50!black,
    coltitle=blue!20!black,
    boxrule=0.8pt,
    arc=4pt,
    left=6pt,
    right=6pt,
    top=6pt,
    bottom=6pt,
    fonttitle=\bfseries,
    title={Author's Comment},
  }
}

\newcommand{\Vspace}{\mathbb V}
\newcommand{\Qspace}{Q}

\begin{document}

\title[Higher Order Discretization for the Stochastic Navier--Stokes Equations]{A Higher Order Discretization for the Stochastic Navier--Stokes equations with additive Noise}

\author{\v{L}ubom\'{i}r Ba\v{n}as}
\address{Department of Mathematics, Bielefeld University, 33501 Bielefeld, Germany}
\email{banas@math.uni-bielefeld.de}

\author{Dominic Breit}
\address{Faculty of Mathematics, University of Duisburg-Essen, Thea-Leymann-Str. 9, 45127 Essen, Germany}
\email{dominic.breit@uni-due.de}

\author{Abhishek Chaudhary}
\address{Mathematisches Institut der Universit\"at T\"ubingen,
                Auf der Morgenstelle 10,
                D-72076 T\"ubingen, Germany.}
\email{chaudhary@na.uni-tuebingen.de}

\author{Andreas Prohl}
\address{Mathematisches Institut der Universit\"at T\"ubingen,
                Auf der Morgenstelle 10,
                D-72076 T\"ubingen, Germany.}
\email{prohl@na.uni-tuebingen.de}

%
%

\begin{abstract}
We propose a new higher-order time discretization scheme for the stochastic Navier--Stokes equations with additive noise, where its velocity and pressure approximates
converge at strong rate $1.5$ in probability. The construction rests on its reformulation as a random PDE for the transform $y = u-
\Phi W$, and different higher order numerical quadrature rules for the diffusion and the drift part.
The theoretical findings are supported by numerical simulations.
\end{abstract}


\date{\today}

\maketitle

%
%
%
%
%
%
%
%
%
%

\section{Introduction}
Let $\mt \subset {\mathbb R}^2$  be the two-dimensional torus, in which we consider the stochastic Navier--Stokes equations driven by additive noise. Its solution consists of the velocity field ${\bf u}$ and the pressure function $p$, which are both defined on a filtered probability space $(\Omega, {\mathfrak F},
({\mathfrak F})_t), {\mathbb P})$, solving
\begin{eqnarray}\nonumber 
{\rm d}{\bf u}(t) 
+ \big[({\bf u}(t)\cdot\nabla){\bf u}(t)
- \nu\,\Delta {\bf u}(t)
+ \nabla p(t)\big]\,{\rm d}t = \Phi\,{\rm d}W(t), \quad \mbox{in } {\mathcal Q}_T\,, \\ \label{eq:stochastic-NS}
{\rm div}\, {\bf u} = 0\quad \mbox{in } {\mathcal Q}_T\,, \\ \nonumber
{\bf u}(0)={\bf u}_0 \quad \mbox{in } {\mathcal Q}_T\,, 
\end{eqnarray}
${\mathbb P}$-a.s. in ${\mathcal Q}_T := (0,T) \times {\mt}$, with terminal time $T>0$, the viscosity $\nu>0$, and ${\bf u}_0$ a given initial datum. The process
$W$ is a cylindrical Wiener
defined on a filtered probability space
$(\Omega, {\mathfrak F},
({\mathfrak F})_t), {\mathbb P})$, and the additive
noise enters the model through the mapping
\[
\Phi W(t,x) := \sum_{k\geq 1} W_k(t)\,\boldsymbol{\phi}_k(x),\qquad
\boldsymbol{\phi}_k\in \mathbb H^{2}(\mt)^2
\ \text{deterministic and divergence-free;}
\]
see Section \ref{theo-1} for precise details.
Different stochastic forcing terms are used in engineering sciences to drive the Navier-Stokes equations, to {\em e.g.}~study the (in-)stability of fluid flow patterns from the deterministic equation; see {\em e.g.}~\cite{BF1,BPW} for recent surveys. These include general multiplicative noise, transport-type noise, or additive noise ---, and each related model requires different numerical strategies to guarantee its accurate numerical resolution as well. The main focus in this work is to construct a higher order discretisation for (\ref{eq:stochastic-NS}) where the noise is additive. This then allows a reliable, efficient
simulation  of  (\ref{eq:stochastic-NS}) to {\em e.g.}~quantitatively understand the effect of increased driving noise  onto 
deterministic fluid flow patterns; see Figure \ref{fig_lid_expval_nu40} for the lid-driven cavity problem.

\begin{figure}
\includegraphics[width=0.32\textwidth]{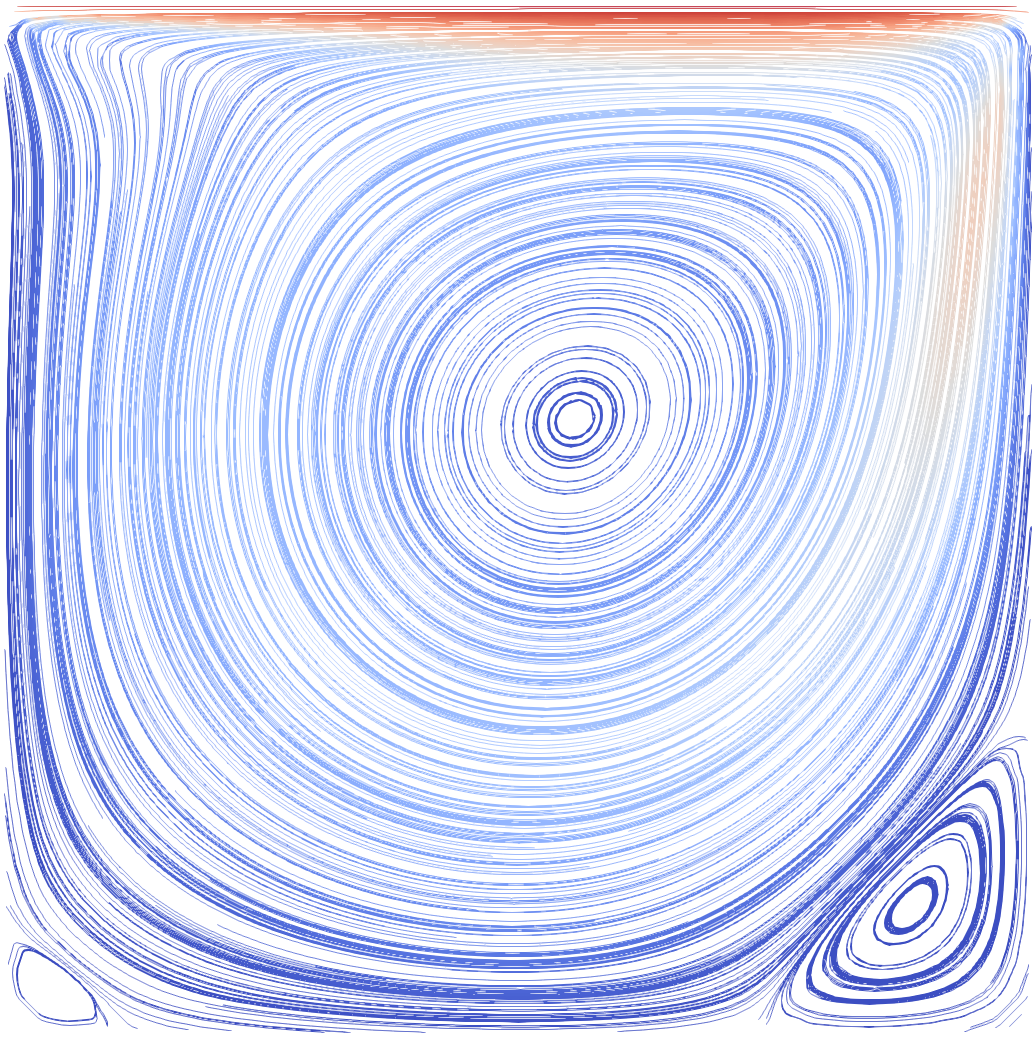}
\includegraphics[width=0.32\textwidth]{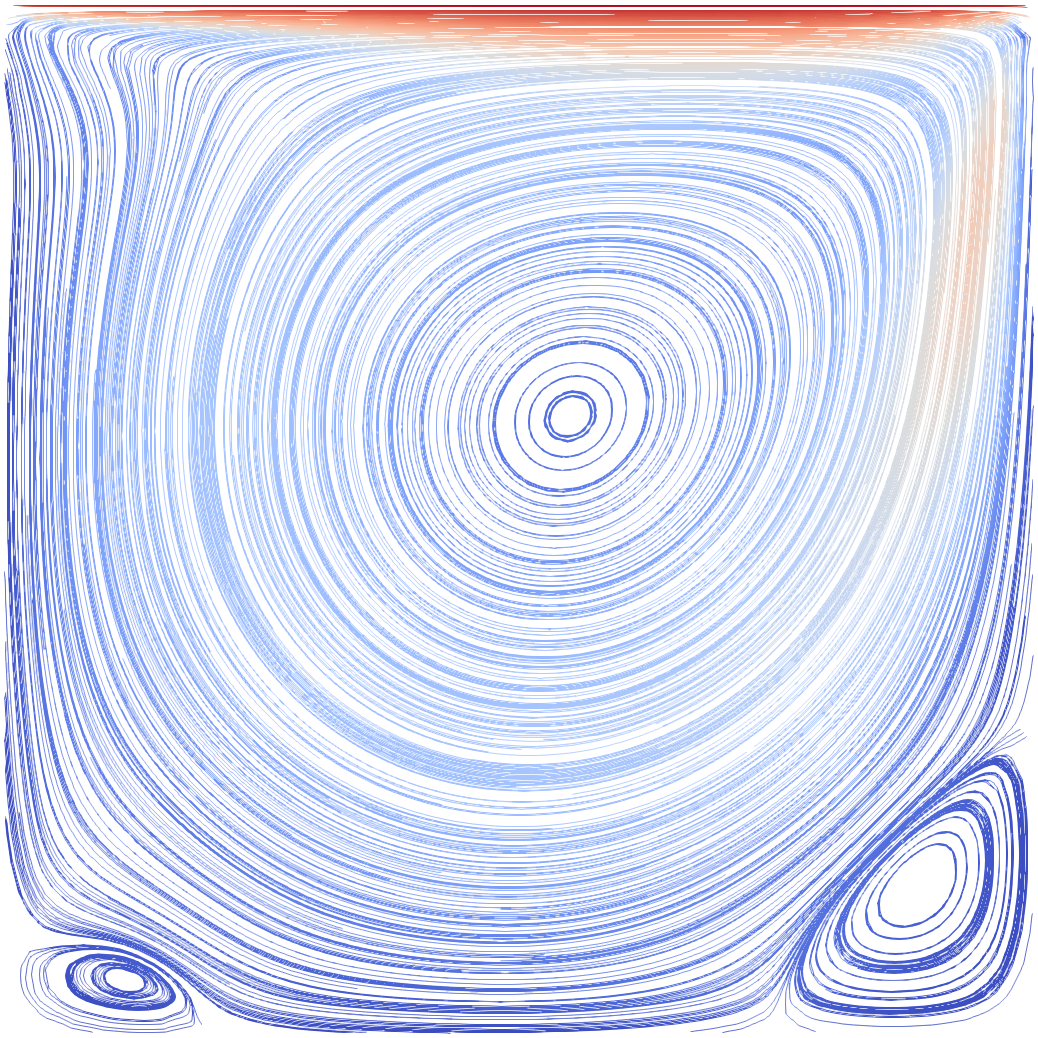}
\includegraphics[width=0.32\textwidth]{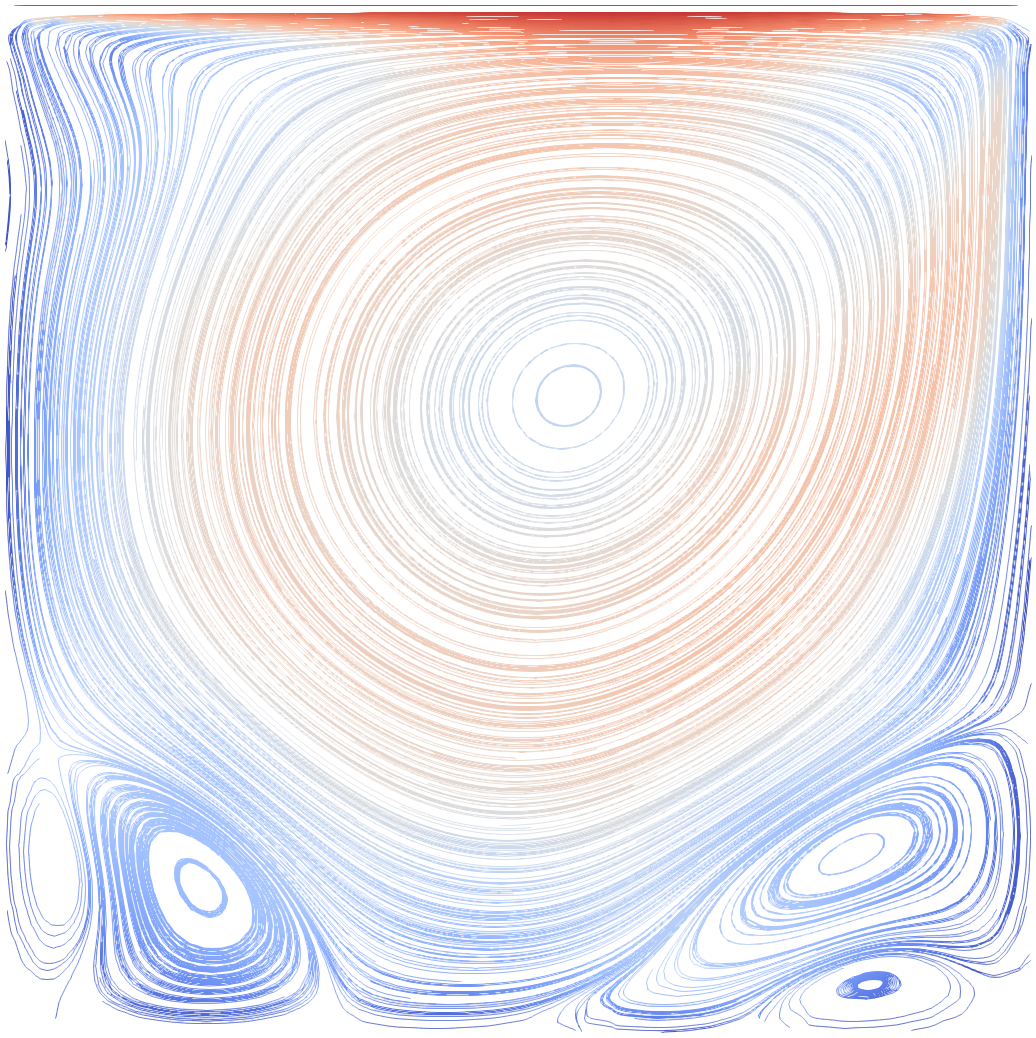}
\caption{Lid-driven cavity (see Section \ref{example_lid_driven}): Streamlines of the deterministic solution ($\mu = 0$; left),
streamlines of the expected value of the solution with  smaller  ($\mu=10$; middle), and  larger noise ($\mu=40$; right) at $T=20$.}
\label{fig_lid_expval_nu40}
\end{figure}

A main impedience towards faster simulation of (\ref{eq:stochastic-NS}) is the appearance of the Wiener process $W$, which is only H\"older continuous in time. As a result, a direct discretisation of it involves the increments of $W$ -- whose limited temporal regularity eventually leads to an order ${\mathcal O}(\sqrt{\tau})$ for the temporal discretisation error; we refer to \cite{BPW} for its detailed discussion. As a consequence, quite small time steps $\tau >0$ need be chosen to guarantee accurate simulations.

A special case is additive noise in (\ref{eq:stochastic-NS}), where more efficient discretisation strategies may be constructed; the idea for it is to apply  
the transform
\begin{equation}\label{zwei}
{\bf y}(t) = {\bf u}(t) - \int_0^t \Phi \, {\rm d}W(s) = {\bf u}(t) - \Phi W(t) \qquad (0 \leq t \leq T)\,,
\end{equation}
which solves the following {\em random} PDE
\begin{eqnarray}\nonumber
\partial_t {\bf y} = \nu \Delta{\bf y} - P_{\mathrm{HL}}\bigl[ ({\bf y}\cdot \nabla) {\bf y} \bigr]
+ \nu\Delta[\Phi W] - P_{\mathrm{HL}}[{\mathcal L}^W({\bf y})] \quad \mbox{in } {\mathcal Q}_T\,, \\ \label{req:stochastic-NS}
{\rm div}\, {\bf y} = 0\,, \\ \nonumber
{\bf y}(0) = {\bf u}_0.
\end{eqnarray}
Here $P_{\mathrm{HL}}$ is the Helmholtz-Leray projection onto divergence free functions, and ${\mathcal L}^W :=
\sum_{i=1}^3 {\mathcal L}_i^W$ couples the transform ${\bf y}$ with the Wiener process $W$, with
$${\mathcal L}_1^W({\bf y}) = ({\bf y}\cdot \nabla)[\Phi W] \,, \qquad
{\mathcal L}_2^W({\bf y}) = ([\Phi W]\cdot\nabla) {\bf y}\,, \qquad
{\mathcal L}_3^W  = ([\Phi W]\cdot\nabla )[\Phi W]\, .
 $$
Here we used that the driving noise is solenoidal, and sufficiently smooth in space.

The benefit now to construct a discretisation in time from  (\ref{req:stochastic-NS}) rather than
(\ref{eq:stochastic-NS}) is an improved convergence rate: in particular, the Euler discretisation for (\ref{req:stochastic-NS}) that was proposed and analyzed in \cite{BP1} converges at strong order ${\mathcal O}(\tau)$, which may be traced back to improved temporal regularity of the solution ${\bf y}$ of (\ref{req:stochastic-NS}), for which $\partial_t {\bf y}$ now exists as measurable function. Moreover, this special temporal discretisation of (\ref{req:stochastic-NS}) may be re-interpreted as implicit Euler method for (\ref{eq:stochastic-NS}), for which in fact an ${\mathcal O}(\tau)$-convergence rate now holds. We conclude by saying that driving an --- even  nonlinear --- stochastic partial differential equation (SPDE) with driving additive noise is special: an implicit Euler method for this SPDE may be expected to converge with order ${\mathcal O}(\tau)$ based on the above re-interpretation instead of only ${\mathcal O}(\sqrt{\tau})$ in the presence of driving multiplicative noise.

These considerations motivated \cite{CP1}, where even a discretisation of strong order ${\mathcal O}(\tau^{1.5})$ is constructed for the linear stochastic heat equation driven by additive noise. The starting point in \cite{CP1}  is again the random PDE for the transform ${\bf y}$ via (\ref{zwei}) --- which reads like (\ref{req:stochastic-NS}), but without the terms $P_{\mathrm{HL}}\bigl[ ({\bf y}\cdot\nabla) {\bf y} \bigr]$ and $P_{\mathrm{HL}}[{\mathcal L}^W({\bf y})]$. The  higher-order implicit scheme  uses 
\begin{itemize}
\item[(i)] a Crank-Nicholson discretisation for the time differentiable term, which appears in (\ref{req:stochastic-NS}) as $\nu \Delta{\bf y}$, and whose {\em time derivative}, in appropriate norms, is even H\"older continuous with exponent $\frac{1}{2}$, and
\item[(ii)]
a mesh of order ${\mathcal O}(\tau^2)$ and an Euler-type discretization
of the H\"older continuous term, which appears in (\ref{req:stochastic-NS}) as  $\Delta[\Phi W]$.
\end{itemize}
In the discussion below, and different from the implicit Euler method,
this higher-order method for the random PDE (see Algorithm \ref{alg:MCN-SPDE} in Section \ref{snse_scheme}) may now {\em not} any more be re-interpreted  as a straight-forward transformation of the discretisation for the original SPDE. We remark that the resulting  scheme of order ${\mathcal O}(\tau^{1.5})$ in \cite{CP1} is of comparable complexity as the implicit Euler method, {since the main computational effort is needed to solve related algebraic systems that result from discretization in space}; see also Remark \ref{remark_aufwand} in this regard. On a technical level for the convergence proof, the improved order is linked to step (ii) where the term $\int_{t_{m-1}}^{t_m}\Delta[\Phi W(s)]\, {\rm d}s$ is approximated on the finer micro mesh of order ${\mathcal O}(\tau^2)$. 
 
 The error analysis in \cite{CP1} for the linear stochastic heat equation is based on
 improved regularity properties for the solution ${\bf y}$ of the related random PDE,
 a truncation error bound  for the arising deterministic integral in the context of the Crank-Nicholson scheme for integrands with limited regularity in \cite{DM1}, and a perturbation argument. A crucial difference between the numerical analysis of
 low-order methods such as the implicit Euler method, and the present higher order method  is that the latter lacks an immediate  discrete energy bound --- which is a relevant property to prove convergence, in particular here since the SPDE (\ref{req:stochastic-NS}) is nonlinear.  This missing property for the higher order method will be
 compensated in its analysis below by assuming stronger regularities for the data; we refer to Section \ref{subsec:solution2} for further details. Hence, a relevant question could be whether the higher order method 
 still performs well for general data of only basic regularity; we do not pursue this direction here and leave this interesting direction to future research. Instead, we focus in this work on analyzing the interplay of the quadratic nonlinearity in (\ref{req:stochastic-NS}) with the noise term in the context of a higher-order temporal discretization, in the context of regular data.
%
%
%
 
 The detailed aims of this paper is to contribute to the following aspects:
 \begin{itemize}
 \item[(i)] {\bf Higher order scheme:} We use the reformulation (\ref{req:stochastic-NS}) to construct a scheme ({\em i.e.}, Algorithm \ref{alg:MCN-SPDE}) in Section \ref{snse_scheme}. Asymptotic strong order ${\mathcal O}(\tau^{1.5})$ in probability for velocity iterates is then verified in Section \ref{error_velocity}. The construction tools combine those already available from \cite{CP1} with a proper discretization of the nonlinear  coupling terms in (\ref{req:stochastic-NS}).
 \item[(ii)] {\bf Approximation of pressure:} {To approximate the pressure $p$ in (\ref{eq:stochastic-NS})
 is a delicate issue for its numerical analysis respectively simulation; see \cite{BPW} for a detailed discussion. We remark that 
 $p$ is {\em not} transformed as the velocity ${\bf u}$ is in (\ref{zwei}), and thus 
 the same pressure $p$ appears in a random PDE like  (\ref{zwei}) if no projection were applied. Consequently,  we may not expect $\partial_t p$ to exist in proper function spaces, in particular; see also Remark \ref{pressure_remark}. Nevertheless, we show strong order
 ${\mathcal O}(\tau^{1.5})$ in probability of pressure iterates $p_{n+1}$ from
 Algorithm \ref{alg:MCN-SPDE} towards {\em locally time averaged pressures $\frac{1}{\tau} \int_{t_n}^{t_{n+1}} p(s)\, ds$ from (\ref{zwei})} in Section \ref{strong_pressure}.}
 \item[(iii)] {\bf Computational studies:} Variational tools conceptionally spurred the construction (and analysis) of Algorithm \ref{alg:MCN-SPDE}:
 though its {\em theoretical backup} in this work is confined to the torus ${\mathbb T}^2$, and 
 periodic boundary conditions to profit from the regularity results in Section \ref{req:stochastic-NS}, {\em the scheme} may be used on general domains in 2D and 3D,
 or may easily be adjusted to mixtures in complex, or even non-Newtonian fluid flows --- which would not be possible if more specialized numerical tools ({\em e.g.}, exponential integrators) would constitute the method. 
 
 The simulations in Section \ref{comstudies2} use {\em discretely inf-sup stable} mixed finite elements, the two-dimensional domain $(0,1)^2$, and Dirichlet boundary conditions: first, we report on computational convergence orders for both, velocity and pressure iterates from  Algorithm \ref{alg:MCN-SPDE} for an academic example with prescribed Dirichlet boundary data. Then, we resume with the lid-driven cavity fluid flow from Figure \ref{fig_lid_expval_nu40}, evidencing the impact that stronger
 noise has on pathwise as well as averaged fluid flow dynamics. 
 Since this practically relevant problem studies dynamics for longer times --- and
 {the overall complexity of the higher order method is
 comparable to {\em e.g.}~the lower-order semi-implicit Euler method for (\ref{eq:stochastic-NS}) ---, its increased accuracy here allows larger time steps, and hence a lower computational effort to preserve a given error tolerance.}%
%
 \item[(iv)] {\bf Extensions:} The convergence theory in this work is for (\ref{eq:stochastic-NS}) with solenoidal noise. However, only slight modifications apply when general noise is considered (see Section \ref{section_general_noise}),
 and the quadratic nonlinearity is omitted, {\em i.e.}, the Stokes flow is considered (see Section
  \ref{sec:stokes}).
 \end{itemize}
 The remainder of the manuscript is organised as follows: Section \ref{theo-1} gathers together relevant notations, recalls the solution concept for (\ref{eq:stochastic-NS}), and
 improved time regularity results for (\ref{req:stochastic-NS}). The higher order scheme is proposed and analysed in Sections \ref{theo-2} and \ref{sec:stokes}.
Computational studies are reported in Section \ref{comstudies2}.%

\section{Theoretical Background}\label{theo-1}

\subsection{Function spaces}
By $\mathbb L^q(\mt)$ with $1\leq q\leq\infty$ we denote the standard Lebesgue spaces over $\mt$ of periodic functions.
In particular, we set
\[
\Qspace := \mathbb{L}_0^2({\mathbb{T}^2}):=\bigg\{q\in\mathbb{L}^2(\mathbb{T}^2) : \int_{\mathbb{T}^2}q(x)\,{\rm d}x=0\bigg\},
\]
for the pressure space
and we write $(\cdot,\cdot)$ for the $\mathbb L^2({\mathbb{T}^2})$ inner
product. Also, we write
\[
\mathbb H :=\bigg\{{\bf v}\in\mathbb{L}^2(\mathbb{T}^2)^2 : \int_{\mathbb{T}^2}\mathbf{v}\cdot\nabla q(x)\,{\rm d}x=0\,\,\,\forall q\in\mathbb H^1(\mt)\bigg\}.
\]
By $\mathbb{H}^k({\mathbb{T}^2})$ for $k\in\N$ we denote Sobolev spaces with differentiability $k$ (and integrability $q=2$). The dual of $\mathbb{H}^k({\mathbb{T}^2})$ is denoted by $\mathbb{H}^{-k}({\mathbb{T}^2})$.
We use the standard solenoidal velocity space
\begin{align*}
\Vspace := \{{\bf v}\in \mathbb{H}^1({\mathbb{T}^2})^2:\ \diver {\bf v}=0\}
\end{align*}
with dual space $\mathbb V^{-1}$.
The transport-form trilinear form is defined by
\[
\mathcal C({\bf a},{\bf b},{\bf c})
:
= \int_{\mathbb{T}^2}({\bf a}(x)\!\cdot\!\nabla){\bf b}(x)\cdot {\bf c}(x)\,\mathrm{d}x,
\qquad {\bf a}\in \Vspace,{\bf b},{\bf c}\in\mathbb{H}^1({\mathbb{T}^2})^2.
\]

We will use the following well-known cancellation property.
\begin{lemma}[Transport-form energy cancellation]\label{lem:transport}
Let ${\bf a}\in \mathbb{V}$ and
${\bf w}\in \mathbb{H}^1({\mathbb{T}^2})^2$. Then
\[
 \mathcal C({\bf a},{\bf w},{\bf w})
 = \int_{\mathbb{T}^2} ({\bf a}(x)\!\cdot\!\nabla){\bf w}(x)\cdot {\bf w}(x)\,\mathrm{d}x
 = \frac12\int_{\mathbb{T}^2} {\bf a}(x)\!\cdot\!\nabla\bigl(|{\bf w}({x})|^2\bigr)\,\mathrm{d}x = 0.
\]
\end{lemma}

We also recall the standard continuity estimate for the trilinear form,
which follows from H\"older's inequality, Ladyzhenskaya's inequality,
and the Poincar\'e inequality in two dimensions: there exists a constant
$C>0$ such that
\begin{equation}\label{eq:C-cont}
|\mathcal C({\bf a},{\bf b},{\bf c})|
\le C\,\|{\bf a}\|_{\mathbb{H}^1}\,
       \|{\bf b}\|_{\mathbb{H}^1}\,
       \|\nabla {\bf c}\|_{\mathbb L^2({\mathbb{T}^2})},
\qquad {\bf a},{\bf b},{\bf c}\in\mathbb H^1({\mathbb{T}^2})^2.
\end{equation}

For a separable Banach space $(X,\|\cdot\|_X)$, we denote by $\mathbb L^p(I;X)$, the set of (Bochner-) measurable functions $u:I\rightarrow X$ such that the mapping $t\mapsto \|u(t)\|_{X}$ belongs to $\mathbb L^p(I)$. 
The set $C(\overline{I};X)$ denotes the space of functions $u:\overline{I}\rightarrow X$ which are continuous with respect to the norm topology on $(X,\|\cdot\|_X)$. For $\alpha\in(0,1]$ we write
$C^{0,\alpha}(\overline{I};X)$ for the space of H\"older-continuous functions with values in $X$. 
Similarly, for a probability space $\mathbb F:=(\Omega,\mathfrak F,\p)$ and a separable Banach space $(X,\|\cdot\|_X)$ and $p\in[1,\infty]$
we write $\mathbb L^p(\Omega,\mathfrak F,\p;X)$ or short
$\mathbb L^p_{\mathbb F}(\Omega;X)$ for the set of (Bochner-) measurable functions $v:\Omega\rightarrow X$ such that the mapping $\omega\mapsto \|v(\omega)\|_{X}$ belongs to $\mathbb L^p(\Omega,\mathfrak F,\p)$.

\subsection{The concept of solutions for SPDE (\ref{eq:stochastic-NS})}
\label{subsec:solution1}
Let $(\Omega,\F,(\F_t)_{t\geq0},\prst)$ be a stochastic basis with a complete, right-continuous filtration. The process $W$ is a cylindrical Wiener process, that is, $W(t)=\sum_{k\geq1}W_j(t) \mathfrak e_j$ with $(W_j)_{j\geq1}$ being mutually independent real-valued standard Wiener processes relative to $(\F_t)_{t\geq0}$, and $(\mathfrak e_j)_{j\geq1}$ a complete orthonormal system in a separable Hilbert space $\mathfrak{U}$.
We assume that the diffusion coefficient $\varPhi$ belongs to the set of Hilbert-Schmidt operators $L_2(\mathfrak U;\mathbb X)$, where
$\mathbb X$ can take the role of various Hilbert spaces. Of particular importance are spaces of solenoidal vector fields, such as $\mathbb H$, $\mathbb V$ and $\mathbb H^{k}(\mt)$.

For a diffusion coefficient $\Phi\in L_2(\mathfrak U;\mathbb X)$, where $L_2$ denotes the space of Hilbert-Schmidt operators, the stochastic integral
 \begin{align*}
\int_0^t \varPhi\,\dd W=\sum_{k\geq 1} \int_0^t\Phi\mathfrak e_k\dd W_k
\end{align*}
is well-defined with values in $\mathbb X$.

In dimension two, pathwise uniqueness for weak solutions is known, we refer the reader for instance to Capi\'nski--Cutland \cite{CC} and Capi\'nski \cite{Ca}. Consequently, we may work with the definition of a weak pathwise solution to (\ref{eq:stochastic-NS}).

\begin{definition}\label{def:inc2d}
Let $(\Omega,\mf,(\mf_t)_{t\geq0},\prst)$ be a given stochastic basis with a complete right-conti\-nuous filtration and an $(\mf_t)$-cylindrical Wiener process $W$. Let ${\bf u}_0$ be an $\mf_0$-measurable random variable and $\Phi\in L_2(\mathfrak U;\mathbb L^2(\mt)^2)$. Then ${\bf u}$ is called a \emph{weak pathwise solution} \index{incompressible Navier--Stokes system!weak pathwise solution} to (\ref{eq:stochastic-NS}) with the initial condition ${\bf u}_0$ provided
\begin{enumerate}
\item the velocity field ${\bf u}$ is $(\mf_t)$-adapted and
$${\bf u} \in C([0,T];\mathbb H)\cap L^2(0,T; \mathbb V)\quad\text{$\p$-a.s.},$$
\item the momentum equation
\begin{align*}
&\int_{\mt}{\bf u}(t)\cdot\boldsymbol{\varphi}\dx-\int_{\mt}{\bf u}_0\cdot\boldsymbol{\varphi}\dx
\\&=-\int_0^t\int_{\mt}({\bf u}\cdot\nabla){\bf u}\cdot\boldsymbol{\varphi}\dx\,\dif t-\nu\int_0^t\int_{\mt}\nabla{\bf u}:\nabla\boldsymbol{\varphi}\dx\,\dif s\\
&\qquad+\int_0^t\int_{\mt}\Phi\cdot\boldsymbol{\varphi}\dx\,\dif W.
\end{align*}
holds $\p$-a.s. for all $\boldsymbol{\varphi}\in \mathbb V$ and all $t\in[0,T]$.
\end{enumerate}
\end{definition}

\begin{theorem}\label{thm:inc2d}
Let $N=2$ and $\Phi\in L_2(\mathfrak U;\mathbb L^2(\mt)^2)$. Let $(\Omega,\mf,(\mf_t)_{t\geq0},\prst)$ be a given stochastic basis with a complete right-continuous filtration and an $(\mf_t)$-cylindrical Wiener process $W$. Let ${\bf u}_0$ be an $\mf_0$-measurable random variable such that ${\bf u}_0\in L^r(\Omega;\mathbb H)$ for some $r>2$. Then there exists a unique weak pathwise solution to (\ref{eq:stochastic-NS}) in the sense of Definition \ref{def:inc2d} with the initial condition ${\bf u}_0$.
\end{theorem}
In the case of periodic boundary conditions spatial regularity can be shown by deterministic methods. A corresponding statement, suitable for our purposes, is given in the following lemma.
\begin{lemma}\label{lem:reg}
Let the assumptions of Theorem \ref{thm:inc2d} be satisfied. 
 \begin{enumerate}
\item[(a)] We have
\begin{align}\label{eq:W12}
\E\bigg[\sup_{0\leq t\leq T}\int_{\mt}|{\bf u}|^2\dx+\int_0^T\int_{\mt}|\nabla{\bf u}|^2\dxt\bigg]^{\frac{r}{2}}\leq\,c_r\,\E\Big[1+\|{\bf u}_0\|_{\mathbb L^2(\mt)^2}^2\Big]^{\frac{r}{2}}.
\end{align}
\item[(b)] Assume that ${\bf u}_0\in \mathbb L^r_{\mathbb F}(\Omega,\mathbb V)$ for some $r\geq2$ and $\Phi\in L_2(\mathfrak U;\mathbb V)$. Then we have
\begin{align}\label{eq:W22}
\E\bigg[\sup_{0\leq t\leq T}\int_{\mt}|\nabla{\bf u}|^2\dx+\int_0^T\int_{\mt}|\nabla^2{\bf u}|^2\dxt\bigg]^{\frac{r}{2}}\leq\,c_r\,\E\big[1+\|{\bf u}_0\|_{\mathbb H^1(\mt)^2}^2\Big]^{\frac{r}{2}}.
\end{align}
\item[(c)] Let $m\in\N$ with $m\geq 2$. Assume that ${\bf u}_0\in \mathbb L^r_{\mathbb F}(\Omega,\mathbb H^{m}(\mt))\cap \mathbb L_{\mathbb F}^{(2m+1)r}(\Omega,\mathbb V)$ for some $r\geq2$ and  and $\Phi\in L_2(\mathfrak U;\mathbb V\cap\mathbb H^{m}(\mt)^2)$. Then we have
\begin{align}\label{eq:W32}
\begin{aligned}
\E\bigg[\sup_{0\leq t\leq T}\|{\bf u}\|_{\mathbb H^{m}(\mt)^2}^2\dx&+\int_0^T\|{\bf u}\|_{\mathbb H^{m+1}(\mathbb T)^2}^2\dt\bigg]^{\frac{r}{2}}\\&\leq\,c_r\,\E\Big[1+\|{\bf u}_0\|_{\mathbb H^2(\mt)^2}^2+\|{\bf u}_0\|_{\mathbb V}^{2(2m+1)}\Big]^{\frac{r}{2}}.
\end{aligned}
\end{align}
\end{enumerate}
\end{lemma}
 \begin{proof}The proof of (a), (b) and (c) with $m=2$ is given in \cite[Lemma 2]{BrDo}. The proof for (c) with general $m$ follows the same strategy similar to \cite[Corollary 2.4.13]{KukShi}.
Formerly,\footnote{The proof can be made rigorous by working with a Galerkin-type approximation and show that the following estimates are uniform with respect to the dimension of the ansatz space.} one applies It\^{o}'s formula to the function $f^{\beta}(\bfu):=\tfrac{1}{2}\|\partial^\beta\bfu\|_{L^2_x}^2$ where $\beta\in \mathbb N_0^2$ is a multi-index of length $m$. 
By \cite[Lemma 2.1.20]{KukShi} and Young's inequality one can estimate the contribution of the convective term as
\begin{align*}
\bigg|\int_{\mt}\partial^\beta\mathrm{div}({\bf u}\otimes{\bf u})\cdot\partial^\beta{\bf u}\dx\bigg|&\leq\,c\|{\bf u}\|_{\mathbb H^{m+1}(\mt)^2}^{\frac{4m-1}{2m}}\|{\bf u}\|_{\mathbb V}^{\frac{m+1}{2m}}\|{\bf u}\|_{\mathbb H}^{\frac{1}{2}}\\
&\leq\delta\|{\bf u}\|_{\mathbb H^{m+1}(\mt)^2}^2+c(\delta)\|{\bf u}\|_{\mathbb V}^{2m+2}\|{\bf u}\|_{\mathbb H}^{2m}\\
&\leq\delta\|{\bf u}\|_{\mathbb H^{m+1}(\mt)^2}^2+c(\delta)\|{\bf u}\|_{\mathbb V}^{2(2m+1)},
\end{align*}
where $\delta>0$ is arbitrary. 
\end{proof}
The following result is a direct consequence, see \cite[Corollary 2]{BrDo}.
\begin{corollary}\label{cor:uholder}
\begin{enumerate}
\item Let the assumptions of Lemma \ref{lem:reg} (b) be satisfied for some $r>2$. Then we have
\begin{align}
\label{eq:holder}
C^{1/2}([0,T];\mathbb L^{r/2}_{\mathbb F}(\Omega;\mathbb H)).
\end{align}
\item Let $m\in\N$ with $m\geq 2$. Let the assumptions of Lemma \ref{lem:reg} (c) be satisfied for some $r>2$. Then we have
\begin{align}
\label{eq:holder}
C^{1/2}([0,T];\mathbb L^{r/2}_{\mathbb F}(\Omega;\mathbb H^{m-1}(\mathbb T^2)^2)).
\end{align}
\end{enumerate}
\end{corollary}
 \begin{proof}
It is proved in \cite[Lemma 2]{BrDo} 
\begin{align*}
&\E\Big[\|{\bf u}\|_{C^\alpha([0,T];\mathbb H)}\Big]^{\frac{r}{2}}<\infty,\quad
\E\Big[\|{\bf u}\|_{C^\alpha([0,T];\mathbb V)}\Big]^{\frac{r}{2}}<\infty,
\end{align*}
for all $\alpha<\frac{1}{2}$
under the assumptions of (a) and (b) with $m=2$, respectively (the second estimate is based on the regularity of
$\int_0^\cdot\Delta\bfu\ds$). 
The claim follows simply by swapping the role of the variables $t$ and $\omega$ in the argument.
 The general case follows exactly the same arguments.
\end{proof}

\subsection{Improved Regularity for pathwise solutions of the random PDE (\ref{req:stochastic-NS})}
\label{subsec:solution2} 
We may rewrite~\eqref{req:stochastic-NS} as  the following random PDE system
\begin{equation}\label{eq:random-NS}
\begin{aligned}
\begin{cases}
\partial_t{\bf y}
+ \big(({\bf y}+\Phi W)\cdot\nabla\big)\big({\bf y}+\Phi W)\big)
- \nu\,\Delta\big({\bf y}+\Phi W\big)
+ \nabla p(t) = 0,\\[0.5ex]
\diver {\bf y}(t) = 0,\\[0.5ex]
{\bf y}(0)={\bf y}_0:={\bf u}_0.
\end{cases}
\end{aligned}
\end{equation}
For each fixed $\omega\in\Omega$, this is a deterministic PDE with coefficients depending on the realization $W(\cdot,\omega)$.
Note that for $\mathbb P$-a.a. $\omega\in\Omega$
the function ${\bf y}(\omega,\cdot)$ is a solution to the Navier--Stokes equations with right-hand side
\begin{align*}
{\bf f}:=\nu \Delta [\Phi W]
-{\mathcal P} \bigl[\mathcal L^W({\bf y})\bigr].
\end{align*} 
Standard regularity results apply provided $\Phi$ is sufficiently regular.
In particular, we have $\partial_t {\bf y} \in \mathbb L^2\bigl( 0,T; \mathbb H \bigr)$ ${\mathbb P}$-a.s.
\begin{lemma}\label{lem:regadditive}
Suppose ${\bf u}_0\in \mathbb L^r_{\mathbb F}(\Omega;\mathbb H)$ for some $r>2$ and $\Phi\in L_2(\mathfrak U;\mathbb V)$. Let ${\bf u}$ be the unique weak pathwise solution to \eqref{eq:stochastic-NS} in the sense of Definition \ref{def:inc2d}.
\begin{enumerate} 
\item[(a)] Assume additionally ${\bf u}_0\in \mathbb V$ $\mathbb P$-a.s.~and $\Phi\in L_2(\mathfrak U;\mathbb H^{2}(\mt)^2)$. Then we have $\partial_t {\bf y} \in L^2\bigl( 0,T; \mathbb H\bigr)$  $\mathbb P$-a.s.~and for a.a. $t\in(0,T)$
\begin{align}\label{eq:W22y}
\int_{\mt}|\partial_t{\bf y}|^2\dx\leq\,c\,\Big[\|\Phi W\|_{\mathbb H^2(\mt)^2}^2+ \|{\bf u}\|_{\mathbb H^2(\mt)^2}^2+\|\Phi W\|_{\mathbb V}^4+\|{\bf u}\|_{\mathbb V}^4\Big].
\end{align}
\item[(b)] Assume additionally ${\bf u}_0\in \mathbb H^{2}(\mt)^2\cap \mathbb V$ $\mathbb P$-a.s.~and $\Phi\in L_2(\mathfrak U;\mathbb H^3(\mt)^2\cap\mathbb V)$. Then $\partial_t {\bf y} \in L^2\bigl( 0,T; \mathbb V) \bigr)$  $\mathbb P$-a.s.~and
\begin{align}\label{eq:W32y}
\begin{aligned}
\int_{\mathcal Q_T}|\partial_t\nabla{\bf y}|^2\dxt&\leq\,c\,\int_0^T\Big[\|\Phi W\|_{\mathbb H^3(\mt)^2}^2+\|{\bf u}\|_{\mathbb H^3(\mt)^2}^2\Big]\dt\\&+c\,\int_0^T\Big[\|\Phi W\|_{\mathbb H^2(\mt)^2}^4+\|{\bf u}\|_{\mathbb H^2(\mt)^2}^4\Big]\dt.
\end{aligned}
\end{align}
\item[(c)] Assume additionally ${\bf u}_0\in \mathbb L^r_{\mathbb F}(\Omega,\mathbb H^4(\mt)^2)\cap \mathbb L_{\mathbb F}^{9r}(\Omega,\mathbb V)$ for some $r\geq2$~and $\Phi\in L_2(\mathfrak U;\mathbb H^{4}(\mt)^2)$. Then we have
\begin{align}\label{eq:W32y}
\partial_t{\bf y}\in C^{1/2}([0,T];\mathbb L^{r/2}_{\mathbb F}(\Omega);\mathbb V).
\end{align}
\end{enumerate}
\end{lemma}
\begin{proof}
(a) and (b) are proved in \cite[Corollary 3.1]{BP1}. As far as (c) is concerned, we simply use \eqref{eq:random-NS} together with the temporal regularity of the Wiener process. Indeed it holds for $t,s\in[0,T]$ with $s\neq t$
\begin{align*}
\|\partial_t{\bf y}(t)-\partial_t{\bf y}(s)\|^2_{\mathbb V}&\leq\,\|\Delta \Phi W(t)-\Delta \Phi W(s)\|_{\mathbb V}^2+\|({\bf u}(t)\nabla) {\bf u}(t)-({\bf u}(s)\nabla) {\bf u}(s)\|_{\mathbb V}^2\\&+\|\Delta{\bf u}(t)-\Delta{\bf u}(s)\|^2_{\mathbb V}\\
&\leq\,\|\Phi (W(t)-W(s))\|_{\mathbb H^3(\mt)^2}^2+\|({\bf u}(t)\nabla) {\bf u}(t)-({\bf u}(s)\nabla) {\bf u}(s)\|_{\mathbb V}^2\\&+\|{\bf u}(t)-{\bf u}(s)\|^2_{\mathbb H^3(\mt)^2}.
\end{align*}
The last term can bounded by means of Corollary \ref{cor:uholder} (b), the first one by the assumptions on $\Phi$. As for the second one we have
\begin{align*}
\|({\bf u}(t)\nabla) {\bf u}(t)-({\bf u}(s)\nabla) {\bf u}(s)\|_{\mathbb V}^2&\leq \|\nabla {\bf u}(t)-\nabla{\bf u}(s)\|_{\mathbb L^{4}(\mt)^4}^2\|\mathbf u(t)\|_{\mathbb L^4(\mt)^2}\\&+c\|{\bf u}(t)-{\bf u}(s)\|_{\mathbb L^{4}(\mt)^2}^2\|\nabla{\mathbf u}\|^2_{\mathbb L^4(\mt)^{4}}\\&\leq c\|{\bf u}(t)-{\bf u}(s)\|_{\mathbb H^{2}(\mt)^2}^2\sup_{t\in[0,T]}\|{\bf u}(t)\|_{\mathbb H^{2}(\mt)^2}^2.
\end{align*}
Now taking the power $r/2$, building expectations and taking the supremum with respect to $s$ and $t$ yields the claim.
\end{proof}

\section{Scheme and Strong Rates in Probability}\label{theo-2}



\subsection{Time--discretisation scheme}\label{sec: time-discretization}

Let $\{t_n=n\tau\}_{n=0}^N$ be a uniform partition of $[0,T]$ with time
step $\tau>0$. For $n\ge0$, we define the discrete midpoint and the BDF2
extrapolated values by
\[
{{\bf y}}^{n+\frac12} := \tfrac12({\bf y}_{n+1}+{\bf y}_n), 
\qquad 
{\bf y}_\star^{n+\frac12}:=\tfrac32\,{\bf y}_n-\tfrac12\,{\bf y}_{n-1},
\quad n\ge 0,\qquad {\bf y}_{-1}:={\bf y}_0.
\]
On each sub-interval $I_n=[t_n,t_{n+1}]$ we introduce the
fine mesh
\[
t_{n,\ell}=t_n+\ell\,\tau^2,\qquad \ell=0,\dots,M,\qquad M=\tau^{-1}.
\]

We now describe the various quadrature and approximation ingredients
used in the construction of the scheme.

\begin{itemize}[leftmargin=2em]
  \item[(i)] \textbf{Brownian quadrature on a fine mesh.}
  We approximate the time average of the Brownian motion on $I_n$,
  \[
    \mathcal Q_n^W:=\frac1{\tau}\int_{t_n}^{t_{n+1}} W(s)\,\mathrm{d}s,
  \]
  by the Riemann sum
  \[
    \mathcal I_n^W:=\sum_{\ell=1}^{M}\tau\, W(t_{n,\ell}).
  \]
  The mean-square quadrature error satisfies (see the proof of \cite[Eq.~(3.10)]{CP1})
  \begin{equation}\label{eq:QWvsIW}
  \mathbb E\!\left[|\mathcal Q_n^W-\mathcal I_n^W|^2\right]
  \ \le\ C\,\tau^{3},
  \end{equation}
  for some constant $C>0$ independent of $n$ and $\tau$.

  \item[(ii)] \textbf{Matrix-valued Brownian triple integral.}
  We define the (matrix-valued) triple integral
  \begin{equation*}
  \mathcal{Q}^{W^2}_n
  :=\frac{1}{\tau^3}\int_{t_n}^{t_{n+1}}\!\int_{t_n}^{t_{n+1}}\!\int_{t_n}^{t_{n+1}}
      \bigl(\Phi W(t)-\Phi W(s)\bigr)\otimes\bigl(\Phi W(t)-\Phi W(r)\bigr)\,
      \mathrm{d}s\,\mathrm{d}t\,\mathrm{d}r.
  \end{equation*}
  Its discrete approximation on the fine mesh is
  \begin{align*}
  \mathcal{I}^{W^2}_{n}
  &:= \tau^3\sum_{\ell=1}^{M}\sum_{k=1}^{M}\sum_{j=1}^{M}
      \bigl(\Phi W(t_{n,\ell})-\Phi W(t_{n,k})\bigr)\otimes
      \bigl(\Phi W(t_{n,\ell})-\Phi W(t_{n,j})\bigr)\\
      &=\tau\sum_{\ell=1}^{M}
      \bigl(\Phi W(t_{n,\ell})-\Phi \mathcal{I}_n^W\bigr)\otimes
      \bigl(\Phi W(t_{n,\ell})-\Phi \mathcal{I}_n^W\bigr),
  \end{align*}
  which satisfies the mean-square bound (see Lemma~\ref{lem:triple-BM})
  \begin{equation}\label{eq:QW2vsIW2}
  \mathbb{E}\big[\|\mathcal{Q}_n^{W^2}-\mathcal{I}_n^{W^2}\|_{{\mathbb{L}^2({\mathbb{T}^2})^{4}}}^2\big]
  \le C_{\Phi}\,\tau^3.
  \end{equation}
  \item[(iii)] \textbf{Integral involving the convective term.}
  We denote the time average of the convective term on $I_n$ by
  \[
  \mathcal{Q}_n^{{\bf y}^2}:=\frac{1}{\tau}\int_{t_n}^{t_{n+1}}
  \diver\big(({\bf y}(t)+\Phi W(t))\otimes ({\bf y}(t)+\Phi W(t))\big)\,\mathrm{d}t.
  \]
  Its discrete approximation is chosen as
  \[
  \mathcal{I}_{n}^{{\bf y}^2}
  :=\diver\Big(\big(\bar{\bf u}_\star^{n+\frac{1}{2}}+\Phi\mathcal{I}_n^W\big)
               \otimes\big(\bar{\bf u}^{n+\frac{1}{2}}+\Phi \mathcal{I}_n^W\big)\Big)
    + \diver \mathcal{I}_{n}^{W^2},
  \]
  where
  \[
  \bar{{\bf y}}^{n+\frac12}:=\frac{{\bf y}(t_n)+{\bf y}(t_{n+1})}{2},
  \qquad 
  \bar{\bf y}_*^{n+\frac12}:=\frac{3}{2}{\bf y}(t_n)-\frac{1}{2}{\bf y}(t_{n-1}),
  \qquad {\bf y}(t_{-1}):={\bf y}_0.
  \]

  \item[(i{\bf v })] \textbf{Integral involving the diffusive term.}
  We denote the averaged diffusive term on $I_n$ by
  \[
  \mathcal{Q}_n^{\Delta {\bf y}}
  :=\frac{1}{\tau}\int_{t_n}^{t_{n+1}}\Delta \big[{\bf y}(t)+\Phi W(t)\big]\,\mathrm{d}t,
  \]
  and approximate it by
  \[
  \mathcal{I}_{n}^{\Delta {\bf y}}
  :=\Delta\big[\bar{{\bf y}}^{n+\frac12}+\Phi\mathcal{I}_n^W\big].
  \]
\end{itemize}

The above ingredients motivate the following time--semi--discretisation
of the random PDE~\eqref{eq:random-NS}.

\newpage

\begin{scheme}
\medskip\noindent
\textbf{Linear Crank--Nicolson scheme for the random PDE~\eqref{eq:random-NS}.} For all $n\ge 0$, given ${\bf y}_n,{\bf y}_{n-1}\in\mathbb V$ (with ${\bf y}_{-1}={\bf y}_0$), find
$({\bf y}_{n+1},p_{n+1})\in\mathbb H^1(\mt)^2\times\Qspace$ such that, for all
$({\boldsymbol{\varphi}},q)\in\mathbb H^1(\mt)^2\times\Qspace$,
\begin{align}\label{scheme:main0}
\begin{cases}
\displaystyle
\Big(\frac{{\bf y}_{n+1}-{\bf y}_n}{\tau},{\boldsymbol{\varphi}}\Big)
+ \mathcal C\!\big({\bf y}_\star^{n+\frac12}+\Phi\mathcal I_n^W,\;
                   {{\bf y}}^{n+\frac12}+\Phi\mathcal I_n^W,\;
                   {\boldsymbol{\varphi}}\big)
 -\big(\mathcal{I}_n^{W^2},\nabla{\boldsymbol{\varphi}}\big)
\\[1ex]\qquad\qquad\qquad\qquad\qquad\qquad
+ \nu\,\big(\nabla({{\bf y}}^{n+\frac12}+\Phi\mathcal I_n^W),\nabla{\boldsymbol{\varphi}}\big)
 - (p_{n+1},\mathrm{div}\,{\boldsymbol{\varphi}})
 =0,\\[1ex]
(\mathrm{div}\, {\bf y}_{n+1},q)=0.
\end{cases}
\end{align}
\end{scheme}

Note that the convective term in \eqref{scheme:main0} is linear in
${\bf y}_{n+1}$; the nonlinearity appears only through the extrapolated,
divergence-free advecting field ${\bf y}_\star^{n+\frac12}+\Phi\mathcal I_n^W$.

\subsection{A numerical scheme for the SPDE~\eqref{eq:stochastic-NS}}
\label{snse_scheme}

We now rewrite the time-discrete random PDE scheme in terms of the
original stochastic variable $X$. For all $n\ge 0$ we define the
Brownian quadrature correction terms
\[
\mathcal{J}_{\star}^{n+\frac{1}{2}}
:=\mathcal{I}_n^W-\frac{3}{2}W(t_n)+\frac{1}{2}W(t_{n-1}),
\qquad W(t_{-1})=W(0)=0,
\]
\[
\mathcal{J}^{n+\frac12}
:=\mathcal{I}_n^W-\frac{1}{2}\big(W(t_{n+1})+W(t_{n})\big),
\qquad
\Delta_{n+1}W:=W(t_{n+1})-W(t_{n}).
\]
We then define the discrete transformation
\begin{equation}\label{discrete transformation}
{\bf u}_{n}:={\bf y}_n+\Phi W(t_n),
\qquad n\ge 0.
\end{equation}
By construction, the sequence $\{({\bf u}_n,p_n)\}_{n\ge 1}$ satisfies the
following time--semi--discrete scheme for the original SPDE
\eqref{eq:stochastic-NS}.

\medskip

\begin{scheme}
\medskip\noindent
\textbf{Modified Crank--Nicolson scheme for the SPDE~\eqref{eq:stochastic-NS}.}  For all $n\ge 0$, given ${\bf u}_n,{\bf u}_{n-1}\in\Vspace$ (with ${\bf u}_{-1}={\bf u}_0$),
find $({\bf u}_{n+1},p_{n+1})\in\mathbb H^1(\mt)^2\times\Qspace$ such that, for all
$({\boldsymbol{\varphi}},q)\in\mathbb H^1(\mt)\times\Qspace$,
\begin{align}\label{scheme:main}
\begin{cases}
\displaystyle
\Big(\frac{{\bf u}_{n+1}-{\bf u}_n}{\tau},{\boldsymbol{\varphi}}\Big)
+ \mathcal C\!\big({\bf u}_\star^{n+\frac12}+\Phi\mathcal J_\star^{n+\frac{1}{2}},\;
                   {\bf u}^{n+\frac12}+\Phi\mathcal J^{n+\frac{1}{2}},\;
                   {\boldsymbol{\varphi}}\big)
 -\big(\mathcal{I}_n^{W^2},\nabla{\boldsymbol{\varphi}}\big)
\\[1ex]\qquad\qquad
+ \nu\,\big(\nabla({\bf u}^{n+\frac12}+\Phi\mathcal J^{n+\frac{1}{2}}),\nabla{\boldsymbol{\varphi}}\big)
 - (p_{n+1},\mathrm{div}\,{\boldsymbol{\varphi}})
 =\bigl(\Phi,{\boldsymbol{\varphi}} \bigr)\frac{\Delta_{n+1}W}{\tau},\\[1ex]
(\mathrm{div}\, {\bf u}_{n+1},q)=0,
\end{cases}
\end{align}
where, analogously to ${{\bf y}}^{n+\frac12}$ and ${\bf y}_\star^{n+\frac12}$, we set
\[
{\bf u}^{n+\frac12}:=\tfrac12({\bf u}_{n+1}+{\bf u}_n),
\qquad
{\bf u}_\star^{n+\frac12}:=\tfrac32 {\bf u}_n-\tfrac12 {\bf u}_{n-1},
\qquad {\bf u}_{-1}:={\bf u}_0.
\]
\end{scheme}

\medskip

For practical implementation, it is convenient to summarize the scheme
\eqref{scheme:main} in algorithmic form.

\begin{algorithm}[H]
\caption{Modified Crank--Nicolson scheme for the SPDE~\eqref{eq:stochastic-NS}}
\label{alg:MCN-SPDE}
\begin{algorithmic}[1]
\State \textbf{Input:} Final time $T>0$, time step $\tau>0$, $N:=T/\tau$,
       initial velocity ${\bf u}_0\in\Vspace$.
\State \textbf{Initialization:} Set ${\bf u}_{-1}:={\bf u}_0$, $W(t_0):=0$.
\For{$n=0,\dots,N-1$}
  \State Generate the Brownian increment
         $\Delta_{n+1}W:=W(t_{n+1})-W(t_n)$.
  \State Construct the fine mesh
         $t_{n,\ell}=t_n+\ell\,\tau^2$, $\ell=0,\dots,M$, $M=\tau^{-1}$.
  \State Compute the Brownian quadratures
  \[
    \mathcal I_n^W:=\sum_{\ell=1}^{M}\tau\,W(t_{n,\ell}),\qquad
    \mathcal{I}_n^{W^2}
    := \tau\sum_{\ell=1}^{M}
       \bigl(\Phi W(t_{n,\ell})-\Phi\mathcal{I}_n^W\bigr)\otimes
       \bigl(\Phi W(t_{n,\ell})-\Phi \mathcal{I}_n^W\bigr).
  \]
  \State Define the correction terms
  \[
    \mathcal{J}_{\star}^{n+\frac{1}{2}}
    :=\mathcal{I}_n^W-\frac{3}{2}W(t_n)+\frac{1}{2}W(t_{n-1}),\qquad
    \mathcal{J}^{n+\frac12}
    :=\mathcal{I}_n^W-\frac{1}{2}\big(W(t_{n+1})+W(t_{n})\big).
  \]
  \State Form the midpoint and extrapolated velocities
  \[
    {\bf u}^{n+\frac12}:=\tfrac12({\bf u}_{n+1}+{\bf u}_n),\qquad
    {\bf u}_\star^{n+\frac12}:=\tfrac32 {\bf u}_n-\tfrac12 {\bf u}_{n-1}.
  \]
  \State Find $({\bf u}_{n+1},p_{n+1})\in\Vspace\times\Qspace$ such that, for all
         $(\boldsymbol{\varphi},q)\in\Vspace\times\Qspace$,
  \[
  \begin{aligned}
    \Big(\tfrac{{\bf u}_{n+1}-{\bf u}_n}{\tau},{\boldsymbol{\varphi}}\Big)
    &+ \mathcal C\!\big({\bf u}_\star^{n+\frac12}+\Phi\mathcal J_\star^{n+\frac{1}{2}},\;
                        {\bf u}^{n+\frac12}+\Phi\mathcal J^{n+\frac{1}{2}},\;
                        {\boldsymbol{\varphi}}\big)
     -\big(\mathcal{I}_n^{W^2},\nabla{\boldsymbol{\varphi}}\big)\\
    &\quad+ \nu\,\big(\nabla({\bf u}^{n+\frac12}+\Phi\mathcal J^{n+\frac{1}{2}}),\nabla{\boldsymbol{\varphi}}\big)
     - (p_{n+1},\mathrm{div}\,{\boldsymbol{\varphi}})
     =\bigl(\Phi,{\boldsymbol{\varphi}} \bigl)\frac{\Delta_{n+1}W}{\tau},\\[0.5ex]
    (\mathrm{div}\, {\bf u}_{n+1},q)&=0.
  \end{aligned}
  \]
\EndFor
\State \textbf{Output:} Approximations $\{({\bf u}_n,p_n)\}_{n=1}^N$ to the
       velocity and pressure of \eqref{eq:stochastic-NS}.
\end{algorithmic}
\end{algorithm}

\subsection{Strong rate of convergence for the velocity field}\label{error_velocity}
For any fixed large enough $R>0$ we introduce the high--probability subset
\[
  \Omega_R
  :=\Big\{\omega\in\Omega:\ \sup_{t\in[0,T]}\|{\bf {Y}}(t,\omega)\|_{\mathbb H^2({\mathbb{T}^2})^2}\le R\Big\},
  \qquad {\bf {Y}}(t):={\bf y}(t)+\Phi W(t),
\]
and denote its indicator function by
\[
  \mathcal A_R:=\mathbf{1}_{\Omega_R}.
\]
\textbf{Solution assumptions:} We collect here the regularity assumptions that will be
used later in the error analysis.

\begin{itemize}[leftmargin=2em]
\item \textbf{Regularity of the velocity.}  
We assume that the (pathwise) solution $u$ of \eqref{eq:random-NS} satisfies
\begin{equation}\label{regularity}
\begin{aligned}
\mathcal{A}_R {\bf y}&\in \mathbb{L}^2_{\mathbb{F}}\big(\Omega;C([0,T];\mathbb{H}^2({\mathbb{T}^2})^2)\big)
\cap C^{1}\big([0,T]; \mathbb{L}^4(\Omega;\mathbb{V})\big),\\&\qquad\qquad\mathcal{A}_R \partial_t{\bf y}\in C^{1/2}\big([0,T]; \mathbb{L}^2(\Omega;\mathbb{V})\big).
\end{aligned}
\end{equation}
As a consequence of Lemma \ref{lem:regadditive} the velocity field belongs to these spaces
(even without restricting to $\Omega_R$) provided we have
$${\bf u}_0\in \mathbb L^4_{\mathbb F}(\Omega,\mathbb H^4(\mt)^2)\cap \mathbb L_{\mathbb F}^{36}(\Omega,\mathbb V),\quad\Phi\in L_2(\mathfrak U;\mathbb H^{4}(\mt)^2).$$
\end{itemize}

On the event $\Omega_R$ we have the uniform bound
\begin{equation}\label{eq:UR-H1-bound}
  \|{\bf {Y}}(t)\|_{\mathbb H^2({\mathbb{T}^2})^2}\le R,
  \qquad t\in[0,T].
\end{equation}

Our first main result is the following mean-square error estimate for
the velocity.

\begin{theorem}[First main result]\label{thm:main-local}
Let the assumptions~\eqref{regularity} hold. Then, for every
$R>0$ there exists a constant $C>0$, independent of $\tau$ and $N$, such
that
\begin{equation}\label{eq:main-local-L2}
  \max_{0\le n\le N}
  \mathbb E\big[\mathcal{A}_R\|{\bf y}(t_n)-{\bf y}_n\|_{\mathbb L^2(\mt)^2}\big]
  \ \le\ {\bf e }^{C\,R^2}\,\tau^{3},
\end{equation}
and
\begin{equation}\label{eq:main-local-H1}
  \nu\,\tau\sum_{n=0}^{N-1}
  \mathbb E\big[\mathcal{A}_R\|\nabla(\bar{{\bf y}}^{n+\frac12}-{{\bf y}}^{n+\frac12})\|_{\mathbb L^2(\mt)^4}^2\big]
  \ \le\ {\bf e }^{C\,R^2}\,\tau^{3}.
\end{equation}
By using the transformation ${\bf u}(t)={\bf y}(t)+\Phi W(t)$ and
${\bf u}_n={\bf y}_n+\Phi W(t_n)$, these bounds are equivalent to
\[
  \max_{0\le n\le N}
  \mathbb E\big[\mathcal{A}_R\|{\bf u}(t_n)-{\bf u}_n\|_{\mathbb L^2(\mt)^2}^2\big]
  \ \le\ {\bf e }^{C\,R^2}\,\tau^{3},
\]
and
\[
  \nu\,\tau\sum_{n=0}^{N-1}
  \mathbb E\big[\mathcal{A}_R\|\nabla(\bar{{\bf X}}^{n+\frac12}-{\bf u}^{n+\frac12})\|_{\mathbb L^2(\mt)^4}^2\big]
  \ \le\ {\bf e }^{C\,R^2}\,\tau^{3},
\]
where $\bar{{\bf X}}^{n+\frac12}:=\tfrac12\big({\bf u}(t_n)+{\bf u}(t_{n+1})\big)$ and
${\bf u}^{n+\frac12}:=\tfrac12({\bf u}_{n+1}+{\bf u}_n)$.
\end{theorem}
Making a suitable choice of $R$, for instance
$R=\sqrt{C^{-1}\log \tau^{-\varepsilon}}$ with $\varepsilon>0$ arbitrary, we obtain the following result concerning the convergence in probability arguing as in 
\cite{BP1,BP2}.
\begin{corollary}
Under the assumptions of Theorem \ref{thm:main-local}
we have for any $\xi>0$, $\alpha<3/2$ 
\begin{align*}
&\max_{1\leq n\leq N}\mathbb P\bigg[\|{\bf y}(t_n)-{\bf y}_{n}\|_{\mathbb L^2(\mt)^2}^2+\sum_{m=1}^{n-1} \tau\|\nabla\bar{\bf y}^{m+\frac{1}{2}}-\nabla{\bf y}^{m+\frac{1}{2}}\|_{\mathbb L^2(\mt)^4}^2>\xi\,\tau^{2\alpha}\bigg]\rightarrow 0,\\
&\max_{1\leq n\leq N}\mathbb P\bigg[\|{\bf u}(t_n)-{\bf u}_{n}\|_{\mathbb L^2(\mt)^2}^2+\sum_{m=1}^{n-1} \tau\|\nabla\bar{\bf X}^{m+\frac{1}{2}}-\nabla{\bf u}^{m+\frac{1}{2}}\|_{\mathbb L^2(\mt)^4}^2>\xi\,\tau^{2\alpha}\bigg]\rightarrow 0,
\end{align*}
as $\tau\rightarrow0$
\end{corollary}
The remainder of this subsection is devoted to the proof of
Theorem~\ref{thm:main-local}. The argument consists of three main
steps:
\begin{enumerate}[label=\roman*),leftmargin=2em]
  \item a detailed bound for the consistency residual of the time
        discretisation;
  \item a localized error identity in conservative (transport) form,
        exploiting the cancellation in Lemma~\ref{lem:transport};
  \item an application of a discrete Gronwall inequality to close the
        estimates and obtain \eqref{eq:main-local-L2}--\eqref{eq:main-local-H1}.
\end{enumerate}
\begin{proof}
We prove this theorem in the following subsections.

\subsubsection{Consistency residual}

We introduce the shorthand notation for all $n\ge 0$
\[
{\bf {Y}}(t):={\bf y}(t)+\Phi W(t),\qquad 
\bar{{\bf Y}}^{n+\frac12}:=\frac{1}{\tau}\int_{t_n}^{t_{n+1}}{\bf y}(t)\,\mathrm{d}t+\Phi\,\mathcal Q_n^W,
\]
and
\[
{{\bf Y}}^{n+\frac12}:=\frac{{\bf y}(t_n)+{\bf y}(t_{n+1})}{2}+ \Phi \mathcal{I}_{n}^W,
\qquad 
\bar{\bf y}_*^{n+\frac12}=\frac{3}{2}{\bf y}(t_n)-\frac{1}{2}{\bf y}(t_{n-1})\qquad\text{with}\qquad {\bf y}(t_{-1})={\bf y}_0.
\]

For each time level \(n=0,\dots,N-1\), we define the \emph{consistency
residual} \(\mathcal{R}^{n+\frac12}\in \mathbb{V}^{-1}\) by: for any \(\boldsymbol{\varphi}\in \mathbb{V}\),
\begin{align}\label{eq:R-def}
\langle \mathcal{R}^{n+\frac12},\boldsymbol{\varphi}\rangle
&:=\frac{1}{\tau}\!\int_{t_{n}}^{t_{n+1}}\!
\Big(\mathcal C({\bf {Y}}(t),{\bf {Y}}(t),\boldsymbol{\varphi})-\mathcal C(\bar{{\bf Y}}^{n+\frac12},{{\bf Y}}^{n+\frac12},\boldsymbol{\varphi})\Big)\,\mathrm{d}t\notag\\
&\quad- \mathcal C\!\bigl((\bar{\bf y}_\star^{n+\frac12}+\Phi \mathcal I_n^W)-\bar{{\bf Y}}^{n+\frac12},\, {{\bf Y}}^{n+\frac12},\, \boldsymbol{\varphi}\bigr)\notag\\
&\quad+ \nu\,\frac{1}{\tau}\!\int_{t_{n}}^{t_{n+1}}\!
\big(\nabla {\bf {Y}}(t)-\nabla {{\bf Y}}^{n+\frac12},\nabla\boldsymbol{\varphi}\big)\,\mathrm{d}t.
\end{align}
By using the bilinearity of $\mathcal{C}$, we obtain
\begin{align}\nonumber
\langle \mathcal{R}^{n+\frac12},\boldsymbol{\varphi}\rangle
&:=\frac{1}{\tau}\!\int_{t_{n}}^{t_{n+1}}\!
\Big(\mathcal C({\bf {Y}}(t)-\bar{{\bf Y}}^{n+\frac12},{\bf {Y}}(t)-\bar{{\bf Y}}^{n+\frac12},\boldsymbol{\varphi})\,{\rm d}t\\&\quad+\underbrace{\frac{1}{\tau}\!\int_{t_{n}}^{t_{n+1}}\!
\Big(\mathcal C( {\bf Y}(t)-\bar{{\bf Y}}^{n+\frac12},\bar{{\bf Y}}^{n+\frac12},\boldsymbol{\varphi})\Big)\,\mathrm{d}t}_{=0}\notag\\
&\quad +\frac{1}{\tau}\!\int_{t_{n}}^{t_{n+1}}\!
\mathcal C\big(\bar{{\bf Y}}^{n+\frac12},\,{\bf {Y}}(t)-{{\bf Y}}^{n+\frac12},\,\boldsymbol{\varphi}\big)\,\mathrm{d}t-\mathcal C\!\bigl((\bar{\bf y}_\star^{n+\frac12}+\Phi \mathcal I_n^W)-\bar{{\bf Y}}^{n+\frac12},\, {{\bf Y}}^{n+\frac12},\, \boldsymbol{\varphi}\bigr)\notag\\
&\quad+ \nu\,\frac{1}{\tau}\!\int_{t_{n}}^{t_{n+1}}\!
\big(\nabla {\bf {Y}}(t)-\nabla {{\bf Y}}^{n+\frac12},\nabla\boldsymbol{\varphi}\big)\,\mathrm{d}t.
\end{align}
By using the bilinearity of \(\mathcal C\) and the decomposition
\[
{\bf {Y}}(t)-\bar {{\bf Y}}^{n+\frac12}
=\big({\bf y}(t)-\bar {{\bf y}}^{n+\frac12}\big)+\Phi\big(W(t)-\mathcal Q_n^W\big),
\]
a straightforward expansion yields
\[
\langle\mathcal R^{n+\frac12},\boldsymbol{\varphi}\rangle
=\sum_{i=1}^{7}\langle\mathcal{T}_{i,R,n},\boldsymbol{\varphi}\rangle,
\]
where
\begin{align*}
\big\langle\mathcal{T}_{1,R,n},\boldsymbol{\varphi}\big\rangle
&:=\frac{1}{\tau}\!\int_{t_{n}}^{t_{n+1}}\!
\mathcal C\big({\bf y}(t)-\bar{{\bf y}}^{n+\frac12},\,{\bf y}(t)-\bar{{\bf y}}^{n+\frac12},\,\boldsymbol{\varphi}\big)\,\mathrm{d}t,\\[1mm]
\big\langle\mathcal{T}_{2,R,n},\boldsymbol{\varphi}\big\rangle
&:=\frac{1}{\tau}\!\int_{t_{n}}^{t_{n+1}}\!
\mathcal C\big(\Phi W(t)-\Phi\mathcal{Q}^{W}_n,\,
               \Phi W(t)-\Phi\mathcal{Q}^{W}_n,\,
               \boldsymbol{\varphi}\big)\,\mathrm{d}t,\\[1mm]
\big\langle\mathcal{T}_{3,R,n},\boldsymbol{\varphi}\big\rangle
&:=\frac{1}{\tau}\!\int_{t_{n}}^{t_{n+1}}\!
\mathcal C\big({\bf y}(t)-\bar{{\bf y}}^{n+\frac12},\,
               \Phi W(t)-\Phi\mathcal{Q}^{W}_n,\,
               \boldsymbol{\varphi}\big)\,\mathrm{d}t,\\[1mm]
\big\langle\mathcal{T}_{4,R,n},\boldsymbol{\varphi}\big\rangle
&:=\frac{1}{\tau}\!\int_{t_{n}}^{t_{n+1}}\!
\mathcal C\big(\Phi W(t)-\Phi\mathcal{Q}^{W}_n,\,
               {\bf y}(t)-\bar{{\bf y}}^{n+\frac12},\,
               \boldsymbol{\varphi}\big)\,\mathrm{d}t,\\[1mm]
\big\langle\mathcal{T}_{5,R,n},\boldsymbol{\varphi}\big\rangle
&:=\frac{1}{\tau}\!\int_{t_{n}}^{t_{n+1}}\!
\mathcal C\big(\bar{{\bf Y}}^{n+\frac12},\,{\bf {Y}}(t)-{{\bf Y}}^{n+\frac12},\,\boldsymbol{\varphi}\big)\,\mathrm{d}t,\\[1mm]
\big\langle\mathcal{T}_{6,R,n},\boldsymbol{\varphi}\big\rangle
&:=- \mathcal C\!\bigl((\bar{\bf y}_\star^{n+\frac12}+\Phi \mathcal I_n^W)-\bar{{\bf Y}}^{n+\frac12},\,
                       {{\bf Y}}^{n+\frac12},\, \boldsymbol{\varphi}\bigr),\\[1mm]
\big\langle\mathcal{T}_{7,R,n},\boldsymbol{\varphi}\big\rangle
&:=\nu\,\frac{1}{\tau}\!\int_{t_{n}}^{t_{n+1}}\!
\big(\nabla {\bf {Y}}(t)-\nabla {{\bf Y}}^{n+\frac12},\nabla\boldsymbol{\varphi}\big)\,\mathrm{d}t.
\end{align*}

We recall \eqref{eq:C-cont}, and note that on $\Omega_R$ we have
$\|\bar {{\bf Y}}^{n+\frac12}\|_{\mathbb{H}^2(\mt)^2}\le R$ and
$\|{{\bf Y}}^{n+\frac12}\|_{\mathbb{V}}\le R$.

\begin{lemma}[Residual bound]\label{lem:R-bound}
Assume the regularity hypotheses \eqref{regularity}. Then, for every
\(\delta>0\) there exists a constant \(C(\delta)>0\), independent of
\(n\) and \(\tau\) but depending on \(\nu\) and \(T\), such that for all
\(n=0,\dots,N-1\) and all \( {\boldsymbol{\psi}}\in \mathbb{H}^1({\mathbb{T}^2})^2\),
\begin{align}\label{eq:R-bound-final}
 \begin{aligned}
  \mathbb{E}\big[\,\mathcal{A}_R\,\langle \mathcal{R}^{n+\frac12}, {\boldsymbol{\psi}}\rangle\,\big]
 & \;\le\; C(\delta)\,R^2\big(\tau^{3}+\tau^2\mathbf{1}_{\{n=0\}}\big)
  \;+\;\delta\,\mathbb{E}\big[\mathcal{A}_R\|\nabla  {\boldsymbol{\psi}}\|_{\mathbb{L}^2(\mt)^4}^2\big]
  \\&-\;\mathbb{E}\big[\mathcal{A}_R\,(\mathcal{Q}_n^{W^2},\nabla {\boldsymbol{\psi}})\big].
\end{aligned}
\end{align}
\end{lemma}

\begin{proof}
We only summarize the main steps, since detailed arguments follow
standard lines and use the regularity assumptions
\eqref{regularity}--\eqref{eq:QWvsIW}.
Let $ {\boldsymbol{\psi}}\in\mathbb{V}$ and set $\boldsymbol{\varphi}_R:=\mathcal A_R {\boldsymbol{\psi}}$. For each
$i\in\{1,\dots,7\}$ define
\[
I_i(\boldsymbol{\varphi}_R):=\mathcal A_R\,\langle\mathcal{T}_{i,R,n}, {\boldsymbol{\psi}}\rangle
=\langle\mathcal{T}_{i,R,n},\boldsymbol{\varphi}_R\rangle.
\]
The term $\mathcal T_{2,R,n}$ is left unestimated and appears
explicitly on the right-hand side of \eqref{eq:R-bound-final} through
the identity
\[
\big\langle\mathcal T_{2,R,n}, {\boldsymbol{\psi}}\big\rangle
= -\big(\mathcal Q_n^{W^2},\nabla {\boldsymbol{\psi}}\big).
\]

\smallskip\noindent
\emph{(i) Term $\mathcal T_{1,R,n}$.}  
By using \eqref{eq:C-cont}, we obtain
\begin{align*}
|I_1(\boldsymbol{\varphi}_R)|
&=\Big|\frac{1}{\tau}\int_{t_n}^{t_{n+1}}
  \mathcal C\big({\bf y}(t)-\bar {{\bf y}}^{n+\frac12},
                 {\bf y}(t)-\bar {{\bf y}}^{n+\frac12},
                 \boldsymbol{\varphi}_R\big)\,\mathrm{d}t\Big|\\
&\le \frac{C}{\tau}\int_{t_n}^{t_{n+1}}
 \|{\bf y}(t)-\bar {{\bf y}}^{n+\frac12}\|_{\mathbb{V}}^2\,\mathrm{d}t\,
 \|\nabla\boldsymbol{\varphi}_R\|_{\mathbb{L}^2(\mt)^4}.
\end{align*}
By using the $C^{1}$--regularity of $u$ in
$\mathbb{L}^4(\Omega;\mathbb H^1)$, one shows that
\[
\mathbb E\Big[\frac{1}{\tau}\int_{t_n}^{t_{n+1}}
 \mathcal{A}_R\|{\bf y}(t)-\bar {{\bf y}}^{n+\frac12}\|_{\mathbb{V}}^4\,\mathrm{d}t\Big]
\le C\,R^2\,\|\mathcal{A}_R {\bf u}\|_{C^1([0,T];\mathbb{L}^4(\Omega; \mathbb{V}))}^4\tau^{4}.
\]
Hence it gives
\begin{align*}
\mathbb E[|I_1(\boldsymbol{\varphi}_R)|]
&\le C\,\bigg(\mathbb E\Big[\frac{1}{\tau}\int_{t_n}^{t_{n+1}}
 \mathcal{A}_R\|{\bf y}(t)-\bar {{\bf y}}^{n+\frac12}\|_{\mathbb{V}}^4\,\mathrm{d}t\Big]\bigg)^{1/2}\,\mathbb E[\|\nabla\boldsymbol{\varphi}_R\|_{\mathbb{L}^2(\mt)^4}^2]^{1/2}\\&\le C\,R\,\tau^2\,\mathbb E[\|\nabla\boldsymbol{\varphi}_R\|_{\mathbb{L}^2(\mt)^4}^2]^{1/2}.
\end{align*}
By Young's inequality, for any \(\delta>0\), we obtain
\begin{equation}\label{eq:I1-final}
\mathbb E[|I_1(\boldsymbol{\varphi}_R)|]
\le \delta\,\mathbb E[\|\nabla\boldsymbol{\varphi}_R\|_{\mathbb{L}^2(\mt)^4}^2]
+ C(\delta)\,R^2\,\tau^{4}.
\end{equation}

\smallskip\noindent
\emph{(ii) Terms \(\mathcal T_{3,R,n}\) and \(\mathcal T_{4,R,n}\).}
By \eqref{eq:C-cont}, we get
\begin{align*}
|I_3(\boldsymbol{\varphi}_R)|
&\le \frac{C}{\tau}\int_{t_n}^{t_{n+1}}
 \mathcal{A}_R\|{\bf y}(t)-\bar {{\bf y}}^{n+\frac12}\|_{\mathbb{V}}\,
 \|\Phi(W(t)-\mathcal Q_n^W)\|_{\mathbb{V}}\,
 \|\nabla\boldsymbol{\varphi}_R\|_{\mathbb{L}^2(\mt)^4}\,\mathrm{d}t,\\
|I_4(\boldsymbol{\varphi}_R)|
&\le \frac{C}{\tau}\int_{t_n}^{t_{n+1}}
 \|\Phi(W(t)-\mathcal Q_n^W)\|_{\mathbb{V}}\,
 \|{\bf y}(t)-\bar {{\bf y}}^{n+\frac12}\|_{\mathbb{V}}\,
 \|\nabla\boldsymbol{\varphi}_R\|_{\mathbb{L}^2(\mt)^4}\,\mathrm{d}t.
\end{align*}
Thus both are bounded by the same mixed quantity. By using Cauchy--Schwarz,
\[
\mathbb{E}\bigg[\bigg(\frac{1}{\tau}\int_{t_n}^{t_{n+1}}
 \mathcal{A}_R\|{\bf y}(t)-\bar {{\bf y}}^{n+\frac12}\|_{\mathbb{V}}\,
 \|\Phi(W(t)-\mathcal Q_n^W)\|_{\mathbb{V}}\,\mathrm{d}t\bigg)^2\bigg]^{1/2}
\le \big(\mathbb{E}[A_{u,n}^2]\big)^{1/2}\,\big(\mathbb{E}\big[A_{W,n}^2\big]\big)^{1/2},
\]
where
\[
A_{u,n}:=\frac{1}{\tau}\int_{t_n}^{t_{n+1}}
 \mathcal{A}_R\|{\bf y}(t)-\bar {{\bf y}}^{n+\frac12}\|_{\mathbb{V}}^2\,\mathrm{d}t,\qquad
A_{W,n}:=\frac{1}{\tau}\int_{t_n}^{t_{n+1}}
 \|\Phi(W(t)-\mathcal Q_n^W)\|_{\mathbb{V}}^2\,\mathrm{d}t.
\]
From the previous step,
\(\mathbb E[A_{u,n}^2]\le C\,R\tau^{2}\). 
we use the identity
\begin{align*}
    \frac{1}{\tau}\int_{t_{n}}^{t_{n+1}}\|\Phi (W(t)-\mathcal{Q}_n^W)\|^2_{\mathbb{V}}\,\mathrm{d}t
    =\frac{1}{\tau^2}\int_{t_{n}}^{t_{n+1}}\int_{t_{n}}^{t_{n+1}}\|\Phi(W(t)-W(s))\|^2_{\mathbb V}\,\mathrm{d}s\,\mathrm{d}t.
\end{align*}
Since \(\Phi\in L_2(\mathfrak U;\mathbb H^2(\mt)^2\)
one checks that
\(\mathbb E[A_{W,n}^2]\le C\,\tau\). Hence we obtain
\[
\mathbb E[A_{u,n}A_{W,n}]
\le \mathbb E[A_{u,n}^2]^{1/2}\,\mathbb E[A_{W,n}^2]^{1/2}
\le C\,R\,\tau^{3/2}.
\]
It follows that
\[
\mathbb E[|I_3(\boldsymbol{\varphi}_R)|]+\mathbb E[|I_4(\boldsymbol{\varphi}_R)|]
\le C\,\tau^{3/2}\,\mathbb E[\|\nabla\boldsymbol{\varphi}_R\|_{\mathbb{L}^2(\mt)^4}^2]^{1/2}
\le \delta\,\mathbb E[\|\nabla\boldsymbol{\varphi}_R\|_{\mathbb{L}^2(\mt)^4}^2]
+ C(\delta)\,R^2\,\tau^{3},
\]
that is,
\begin{equation}\label{eq:I3-I4-final}
\mathbb E[|I_3(\boldsymbol{\varphi}_R)|]
+\mathbb E[|I_4(\boldsymbol{\varphi}_R)|]
\le \delta\,\mathbb E[\|\nabla\boldsymbol{\varphi}_R\|_{\mathbb{L}^2(\mt)^4}^2]
+ C(\delta)\,R^2\,\tau^{3}.
\end{equation}

\medskip\noindent
\emph{(iii) Term \(\mathcal T_{5,R,n}\).}
We have
\[
I_5(\boldsymbol{\varphi}_R)
=\frac{1}{\tau}\int_{t_n}^{t_{n+1}}
 \mathcal C(\bar {{\bf Y}}^{n+\frac12},{\bf {Y}}(t)-{{\bf Y}}^{n+\frac12},\boldsymbol{\varphi}_R)\,\mathrm{d}t.
\]
By using \eqref{eq:C-cont} and the bound \eqref{eq:UR-H1-bound} for
\({\bf u}\) on \(\Omega_R\),
\[
|I_5(\boldsymbol{\varphi}_R)|
\le C\,R\,\mathcal{A}_R
\Big\|\frac{1}{\tau}\int_{t_n}^{t_{n+1}}
 \big({\bf {Y}}(t)-{{\bf Y}}^{n+\frac12}\big)\,\mathrm{d}t\Big\|_{\mathbb{V}}\,
\|\nabla\boldsymbol{\varphi}_R\|_{\mathbb{L}^2(\mt)^4}.
\]
We introduce
\[
A_{5,n}
:=\mathcal{A}_R\Big\|\frac{1}{\tau}\int_{t_n}^{t_{n+1}}{\bf {Y}}(t)\,\mathrm{d}t
      -{{\bf Y}}^{n+\frac12}\Big\|_{\mathbb{V}}.
\]
Then
\(
|I_5(\boldsymbol{\varphi}_R)|
\le C\,R\,A_{5,n}\,\|\nabla\boldsymbol{\varphi}_R\|_{\mathbb{L}^2}.
\)
We now estimate \(A_{5,n}\). We split \({\bf{Y}}={\bf u}+\Phi W\) as
\[
A_{5,n}\le A_{5,n}^{(y)}+A_{5,n}^{(W)},
\]
with the obvious definitions. For the deterministic part, by using
\(\mathcal{A}_R \partial_t{\bf y}\in C^{1/2}([0,T];\mathbb \mathbb{L}^2(\Omega;\mathbb V))\) and a
trapezoidal-rule error estimate from Lemma \ref{lem:deltastar-average}, one obtains
\[
\mathbb E\big[\big(A_{5,n}^{(y)}\big)^2\big]\le C\,R^2\,\tau^{3}.
\]
For the stochastic part, we use the Brownian quadrature estimate
\eqref{eq:QWvsIW} together with the spatial regularity of \(\Phi\),
which gives
\[
\mathbb E\big[\big(A_{5,n}^{(W)}\big)^2\big]\le C\,\tau^{3}.
\]
Thus
\[
\mathbb E[A_{5,n}^2]\le C\,R^2\,\tau^{3},
\]
and therefore, we get
\[
\mathbb E[|I_5(\boldsymbol{\varphi}_R)|]
\le C\,R\,\mathbb E[A_{5,n}^2]^{1/2}\,
                 \mathbb E[\|\nabla\boldsymbol{\varphi}_R\|_{\mathbb{L}^2(\mt)^2}^2]^{1/2}
\le \delta\,\mathbb E[\|\nabla\boldsymbol{\varphi}_R\|_{\mathbb{L}^2(\mt)^2}^2]
+ C(\delta)\,R^2\,\tau^{3}.
\]

\medskip\noindent
\emph{(i{\bf v }) Term \(\mathcal T_{6,R,n}\).}
We recall
\[
I_6(\boldsymbol{\varphi}_R)
= -\mathcal C\!\bigl((\bar{\bf y}_\star^{n+\frac12}+\Phi \mathcal I_n^W)
                     -\bar{{\bf Y}}^{n+\frac12},\,
                     {{\bf Y}}^{n+\frac12},\, \boldsymbol{\varphi}_R\bigr).
\]
We set
\[
B_n
:=\mathcal{A}_R\,(\bar{\bf y}_\star^{n+\frac12}-\bar {{\bf y}}^{n+\frac12})
  +\Phi(\mathcal I_n^W-\mathcal Q_n^W),
\]
so that
\((\bar{\bf y}_\star^{n+\frac12}+\Phi \mathcal I_n^W)-\bar{{\bf Y}}^{n+\frac12}=B_n\),
and hence
\[
|I_6(\boldsymbol{\varphi}_R)|
\le C\,\|B_n\|_{\mathbb{V}}\,\|{{\bf Y}}^{n+\frac12}\|_{\mathbb{V}}\,
          \|\nabla\boldsymbol{\varphi}_R\|_{\mathbb{L}^2(\mt)^4}
\le C\,R\,\|B_n\|_{\mathbb{V}}\,\|\nabla\boldsymbol{\varphi}_R\|_{\mathbb{L}^2(\mt)^4},
\]
by using again \eqref{eq:UR-H1-bound}.

\smallskip\noindent
\textbf{Case 1.} If $n\ge 1$, then by Lemma~\ref{lem:deltastar-average}
(applied in \(\mathrm{K}=\mathbb{L}^2(\Omega; \mathbb{H}^1({\mathbb{T}^2})^2)\)) and the Brownian quadrature estimate
\eqref{eq:QWvsIW},
\[
\mathbb E[\|B_n\|_{\mathbb{H}^1(\mt)^2}^2]\le C\,R^2\,\tau^{3}.
\]

\smallskip\noindent
\textbf{Case 2.} If $n=0$, then by using the regularity
${\bf u}\in C^1([0,T];\mathbb{L}^2(\Omega;\mathbb{H}^1({\mathbb{T}^2})^2))$, we obtain
\[
\mathbb E[\|B_n\|_{\mathbb{H}^1(\mt)^2}^2]\le C\,\tau^{2}.
\]
Therefore,
\[
\mathbb E[|I_6(\boldsymbol{\varphi}_R)|]
\le \delta\,\mathbb E[\|\nabla\boldsymbol{\varphi}_R\|_{\mathbb{L}^2(\mt)^4}^2]
+ C(\delta)\,R^2\,(\tau^{3}+\tau^2\mathbf{1}_{\{n=0\}}).
\]

\medskip\noindent
\emph{({\bf v }) Term \(\mathcal T_{7,R,n}\).}
Finally,
\[
I_7(\boldsymbol{\varphi}_R)
=\nu\,\frac{1}{\tau}\int_{t_n}^{t_{n+1}}
\big(\nabla {\bf {Y}}(t)-\nabla {{\bf Y}}^{n+\frac12},\nabla\boldsymbol{\varphi}_R\big)\,\mathrm{d}t.
\]
Since \(\boldsymbol{\varphi}_R\) is time--independent, we can average inside:
\[
I_7(\boldsymbol{\varphi}_R)
=\nu\,\big(A_{n},\nabla\boldsymbol{\varphi}_R\big),
\qquad
A_{n}
:=\mathcal{A}_R\frac{1}{\tau}\int_{t_n}^{t_{n+1}}
\big(\nabla {\bf {Y}}(t)-\nabla {{\bf Y}}^{n+\frac12}\big)\,\mathrm{d}t.
\]
Thus we have
\[
|I_7(\boldsymbol{\varphi}_R)|
\le \nu\,\|A_n\|_{\mathbb{L}^2(\mt)^4}\,\|\nabla\boldsymbol{\varphi}_R\|_{\mathbb{L}^2(\mt)^4}.
\]
By arguing as for \(A_{5,n}\), but at gradient level (again combining the
\(C^{1/2}\) temporal regularity of \(\partial_t{\bf y}\) and the
Brownian quadrature estimate \eqref{eq:QWvsIW}), one obtains
\[
\mathbb E[\|A_n\|_{\mathbb{L}^2(\mt)^4}^2]\le C\,R^2\,\tau^{3}.
\]
Hence
\[
\mathbb E[|I_7(\boldsymbol{\varphi}_R)|]
\le \delta\,\mathbb E[\|\nabla\boldsymbol{\varphi}_R\|_{\mathbb{L}^2(\mt)^2}^2]
+ C(\delta)\,R^2\,\tau^{3}.
\]

\medskip\noindent
\emph{(vi) Conclusion.}
By collecting the estimates for \(I_1,I_3,I_4,I_5,I_6,I_7\) and noting that
\(\mathcal A_R\le1\), we find that for any \(\delta>0\),
\begin{align*}
\mathbb E\big[\mathcal{A}_R\,|\langle\mathcal R^{n+\frac12},\boldsymbol{\varphi}_R\rangle|\big]
&\le \delta\,\mathbb E[\mathcal{A}_R\|\nabla {\boldsymbol{\psi}}\|_{\mathbb{L}^2(\mt)^4}^2]
+ C(\delta)\,R^2\,(\tau^{3}+\tau^2 \mathbf{1}_{\{n=0\}})\\&
-\mathbb E\big[\mathcal{A}_R\,\langle\mathcal Q_{n}^{W^2},\nabla {\boldsymbol{\psi}}\rangle\big],
\end{align*}
which is exactly \eqref{eq:R-bound-final}. This completes the proof of Lemma~\ref{lem:R-bound}.
\end{proof}

\begin{remark}
We define a modified residual for all $\boldsymbol{\varphi}\in\mathbb{V}$ by
\begin{align}\label{modified residual}
\left<\widetilde{\mathcal{R}}^{n+\frac12},\boldsymbol{\varphi} \right>
:=\left<\mathcal{R}^{n+\frac12},\boldsymbol{\varphi}\right>
+\left<\mathcal{Q}_n^{W^2},\nabla \boldsymbol{\varphi}\right>.
\end{align}
The combined noise term is
\begin{align}\label{combined noise}
\big\langle\mathcal N^{n+\frac12},\boldsymbol{\varphi}\big\rangle
:= -\big\langle\mathcal Q_n^{W^2},\nabla\boldsymbol{\varphi}\big\rangle
   +\big(\mathcal I_n^{W^2},\nabla\boldsymbol{\varphi}\big)
= -\big(\mathcal Q_n^{W^2}-\mathcal I_n^{W^2},\nabla\boldsymbol{\varphi}\big).
\end{align}
We introduce the second modified residual
\begin{align}
\big\langle\widetilde{\mathcal R}_1^{n+\frac12},\boldsymbol{\varphi}\big\rangle
:= \big\langle\widetilde{\mathcal R}^{n+\frac12},\boldsymbol{\varphi}\big\rangle
 -\left<\mathcal{N}^{n+\frac12},\boldsymbol{\varphi}\right>.
\end{align}
By arguing along similar lines as in the proof of Lemma~\ref{lem:R-bound}, one can show that
\begin{align*}
    \mathbb{E}\big[\mathcal{A}_R\|\widetilde{\mathcal{R}}_1^{n+\frac12}\|_{\mathbb{V}^{-1}}^2\big]
    \le C\,R^2\,\Big(\mathbf{1}_{\{n=0\}}\tau^2+\tau^3+\mathbb{E}\big[\|\mathcal{N}^{n+\frac12}\|_{{\mathbb{L}^2({\mathbb{T}^2})^2}}^2\big]\Big).
\end{align*}
By using the quadrature error estimate \eqref{eq:QW2vsIW2}, we obtain
\begin{align}\label{20}
    \mathbb{E}\big[\mathcal{A}_R\|\widetilde{\mathcal{R}}_1^{n+\frac12}\|_{\mathbb{V}^{-1}}^2\big]\le C\,R^2\,\big(\mathbf{1}_{\{n=0\}}\tau^2+\tau^3\big).
\end{align}
This estimate will be helpful in the error analysis for the pressure in the next section.
\end{remark}

\subsubsection{Error inequality}

We now derive the localized error inequality in conservative form and
combine it with Lemma~\ref{lem:R-bound}. Let
\[
{\bf e }_n:={\bf y}(t_n)-{\bf y}_n,\qquad 
{\bf e }^{n+\frac12}:=\tfrac12\big({\bf e }_{n+1}+{\bf e }_n\big),\qquad
{\bf e }_\star^{\,n+\frac12}:=\tfrac32 {\bf e }_n-\tfrac12 {\bf e }_{n-1}\quad (n\ge1).
\]

\begin{lemma}[Error inequality]\label{lem:balanced-local}
For every \(\delta\in(0,\nu)\) there exists a constant
\(C(\delta)>0\), independent of \(n\) and \(\tau\) but possibly
depending on \(R,\nu,T\), such that
\begin{align}\label{eq:balanced-local}
\begin{aligned}
\frac{1}{2\tau}\,\mathbb E\big[\mathcal{A}_R\big(\|{\bf e }_{n+1}\|_{\mathbb{L}^2(\mt)^2}^2&-\|{\bf e }_n\|_{\mathbb{L}^2(\mt)^2}^2\big)\big]
+ \frac{\nu}{2}\,\mathbb E\big[\mathcal{A}_R\|\nabla {\bf e }^{n+\frac12}\|_{\mathbb{L}^2(\mt)^4}^2\big]\\
&\;\le\; C(\delta) R^2\,\big(\tau^3+\mathbf{1}_{\{n=0\}}\tau^2\big)+ \delta\,\mathbb E\big[\mathcal{A}_R\|{\bf e }_\star^{n+\frac12}\|_{\mathbb L^2(\mt)^2}^2\big].
\end{aligned}
\end{align}
\end{lemma}

\begin{proof}
Fix \(n\ge0\). Integrating the continuous equation on
\(I_n=[t_n,t_{n+1}]\), writing it in terms of \({\bf {Y}}(t)={\bf y}(t)+\Phi W(t)\),
and testing with \(\boldsymbol{\varphi}\in \mathbb{V}\), we obtain
\begin{align}\label{eq:cont-avg}
\Big(\frac{{\bf y}(t_{n+1})-{\bf y}(t_n)}{\tau},\boldsymbol{\varphi}\Big)
+ &\frac{1}{\tau}\int_{t_n}^{t_{n+1}}
\mathcal C({\bf {Y}}(t),{\bf {Y}}(t),\boldsymbol{\varphi})\,\mathrm{d}t
+ \nu\,\frac{1}{\tau}\int_{t_n}^{t_{n+1}}
(\nabla {\bf {Y}}(t),\nabla\boldsymbol{\varphi})\,\mathrm{d}t=0.
\end{align}
The discrete step \eqref{scheme:main} reads
\begin{align}\label{eq:disc-step2}
\begin{aligned}
\Big(\frac{{\bf y}_{n+1}-{\bf y}_n}{\tau},\boldsymbol{\varphi}\Big)
&+ \mathcal C\big({\bf y}_\star^{n+\frac12}+\Phi\mathcal I_n^W,\;
                   {{\bf y}}^{n+\frac12}+\Phi\mathcal I_n^W,\;
                   \boldsymbol{\varphi}\big)
\\&- \big(\mathcal I_n^{W^2},\nabla\boldsymbol{\varphi}\big)
+ \nu\,\big(\nabla({{\bf y}}^{n+\frac12}+\Phi\mathcal I_n^W),\nabla\boldsymbol{\varphi}\big)=0.
\end{aligned}
\end{align}
By subtracting \eqref{eq:disc-step2} from \eqref{eq:cont-avg} and using
\({\bf e }_{n+1}={\bf y}(t_{n+1})-{\bf y}_{n+1}\), \({\bf e }_n={\bf y}(t_n)-{\bf y}_n\), we obtain
\begin{align*}
\Big(\frac{{\bf e }_{n+1}-{\bf e }_n}{\tau},\boldsymbol{\varphi}\Big)
&+ \nu\,(\nabla {\bf e }^{n+\frac12},\nabla\boldsymbol{\varphi})
+ \mathcal C\big({\bf y}_\star^{n+\frac12}+\Phi\mathcal I_n^W,\;
                 {\bf e }^{n+\frac12},\boldsymbol{\varphi}\big)\notag\\
&= \big\langle\widetilde{\mathcal R}^{n+\frac12},\boldsymbol{\varphi}\big\rangle
-\big(\mathcal Q_n^{W^2},\nabla\boldsymbol{\varphi}\big)
+\big(\mathcal I_n^{W^2},\nabla\boldsymbol{\varphi}\big)
+ \mathcal {E }_{\text{conv}}^{\,n+\frac12}(\boldsymbol{\varphi}),
\label{eq:err-abstract}
\end{align*}
where the \emph{convection defect} is defined by
\[
\mathcal {E }_{\text{conv}}^{\,n+\frac12}(\boldsymbol{\varphi})
:=\mathcal C({\bf e }_\star^{n+\frac12}, {{\bf Y}}^{n+\frac12},\boldsymbol{\varphi}),
\]
and the modified residual is given by
\begin{align}\label{modified-residual-2}
\left<\widetilde{\mathcal{R}}^{n+\frac12},\boldsymbol{\varphi} \right>
=\left<\mathcal{R}^{n+\frac12},\boldsymbol{\varphi}\right>
+\left(\mathcal{Q}_n^{W^2},\nabla \boldsymbol{\varphi}\right).
\end{align}
We now choose \(\boldsymbol{\varphi}={\bf e }^{n+\frac12}\in \mathbb{V}\) and multiply the whole
identity by \(\mathcal A_R\). By Lemma~\ref{lem:transport}, the
conservative transport term vanishes,
\[
\mathcal C\big({\bf y}_\star^{n+\frac12}+\Phi\mathcal I_n^W,\;
               {\bf e }^{n+\frac12},{\bf e }^{n+\frac12}\big)=0,
\]
since the advecting field ${\bf y}_\star^{n+\frac12}+\Phi\mathcal I_n^W$ is divergence-free. Thus we obtain
\begin{align}
\mathcal A_R\Big(\frac{{\bf e }_{n+1}-{\bf e }_n}{\tau},{\bf e }^{n+\frac12}\Big)
+ \nu\,\mathcal A_R\|\nabla {\bf e }^{n+\frac12}\|_{\mathbb{L}^2(\mt)^4}^2
&= \mathcal A_R\big\langle\widetilde{\mathcal R}^{n+\frac12},{\bf e }^{n+\frac12}\big\rangle
 - \mathcal A_R\big(\mathcal Q_n^{W^2},\nabla {\bf e }^{n+\frac12}\big)\notag\\
&\quad +\mathcal A_R\big(\mathcal I_n^{W^2},\nabla {\bf e }^{n+\frac12}\big)
 + \mathcal A_R\,\mathcal {E }_{\text{conv}}^{\,n+\frac12}({\bf e }^{n+\frac12}).
\label{eq:err-local-raw}
\end{align}
By using the identity
\[
\Big(\frac{{\bf e }_{n+1}-{\bf e }_n}{\tau},{\bf e }^{n+\frac12}\Big)
= \frac{1}{2\tau}\big(\|{\bf e }_{n+1}\|_{\mathbb{L}^2(\mt)^2}^2-\|{\bf e }_n\|_{\mathbb{L}^2(\mt)^2}^2\big),
\]
we rewrite \eqref{eq:err-local-raw} as
\begin{align}
\frac{\mathcal A_R}{2\tau}\big(\|{\bf e }_{n+1}\|_{\mathbb{L}^2(\mt)^2}^2-\|{\bf e }_n\|_{\mathbb{L}^2(\mt)^2}^2\big)
&+ \nu\,\mathcal A_R\|\nabla {\bf e }^{n+\frac12}\|_{\mathbb{L}^2(\mt)^4}^2
= \mathcal A_R\big\langle\widetilde{\mathcal R}^{n+\frac12},{\bf e }^{n+\frac12}\big\rangle
\notag\\
&\quad\qquad + \mathcal A_R\big\langle\mathcal N^{n+\frac12},{\bf e }^{n+\frac12}\big\rangle + \mathcal A_R\,\mathcal {E }_{\text{conv}}^{\,n+\frac12}({\bf e }^{n+\frac12}),
\label{eq:err-local}
\end{align}
where the combined noise term is
\begin{align*}
\big\langle\mathcal N^{n+\frac12},\boldsymbol{\varphi}\big\rangle
:= -\big\langle\mathcal Q_n^{W^2},\nabla\boldsymbol{\varphi}\big\rangle
   +\big(\mathcal I_n^{W^2},\nabla\boldsymbol{\varphi}\big)
= -\big(\mathcal Q_n^{W^2}-\mathcal I_n^{W^2},\nabla\boldsymbol{\varphi}\big).
\end{align*}

\medskip\noindent
\emph{Noise-correction term.}  
By \eqref{eq:QW2vsIW2} and Cauchy--Schwarz, we obtain
\[
\mathbb E\Big[
\mathcal A_R\big|\big\langle\mathcal N^{n+\frac12},\boldsymbol{\varphi}\big\rangle\big|
\Big]
\le C\,\tau^{3/2}\,
\mathbb E\big[\mathcal{A}_R\|\nabla\boldsymbol{\varphi}\|_{\mathbb{L}^2(\mt)^4}^2\big]^{1/2},
\]
and hence, for any $\delta>0$,
\begin{equation}\label{eq:noise-corr-bound2}
\mathbb E\Big[
\mathcal A_R\big|\big\langle\mathcal N^{n+\frac12},\boldsymbol{\varphi}\big\rangle\big|
\Big]
\le \delta\,\mathbb E\big[\mathcal{A}_R\|\nabla\boldsymbol{\varphi}\|_{\mathbb{L}^2(\mt)^4}^2\big]
+ C(\delta)\,\tau^3.
\end{equation}

\medskip\noindent
\emph{Convection defect.}  
By \eqref{eq:C-cont} and the localization
$\|\bar {{\bf Y}}^{n+\frac12}\|_{\mathbb{H}^2(\mt)^22}\le R$,
\begin{align}
\big|\mathcal {E }_{\text{conv}}^{\,n+\frac12}({\bf e }^{n+\frac12})\big|
&= \big|\mathcal C({\bf e }_\star^{n+\frac12},\bar {{\bf Y}}^{n+\frac12},{\bf e }^{n+\frac12})\big|\notag\\
&\le C\,\|{\bf e }_\star^{n+\frac12}\|_{\mathbb{L}^2(\mt)^2}\,
          \|\bar {{\bf Y}}^{n+\frac12}\|_{\mathbb{H}^2(\mt)^2}\,
          \|\nabla {\bf e }^{n+\frac12}\|_{\mathbb{L}^2(\mt)^4}\notag\\
&\le C\,R\,\|{\bf e }_\star^{n+\frac12}\|_{\mathbb{L}^2(\mt)^2}\,
          \|\nabla {\bf e }^{n+\frac12}\|_{\mathbb{L}^2(\mt)^4}.
\label{eq:Econv-bound-local}
\end{align}
Hence, for any \(\delta\in(0,\nu)\), we have
\begin{equation}\label{eq:Econv-Young}
\mathcal A_R\big|\mathcal {E }_{\text{conv}}^{\,n+\frac12}({\bf e }^{n+\frac12})\big|
\le \frac{\delta}{2}\,\mathcal A_R\|\nabla {\bf e }^{n+\frac12}\|_{\mathbb{L}^2(\mt)^4}^2
+ C(\delta)\,R^2\,\mathcal A_R\|{\bf e }_\star^{n+\frac12}\|_{\mathbb{L}^2(\mt)^2}^2.
\end{equation}
By applying Lemma~\ref{lem:R-bound} with $ {\bf \psi}={\bf e }^{n+\frac12}$, we obtain
\begin{equation}\label{eq:residual-bound}
\mathbb{E}\big[\,\mathcal{A}_R\,\langle \widetilde{\mathcal{R}}^{n+\frac12},{\bf e }^{n+\frac12}\rangle\,\big]
  \;\le\; C(\delta)\,R^2\,\big(\tau^{3}+\tau^2\mathbf{1}_{\{n=0\}}\big)
  \;+\;\delta\,\mathbb{E}\big[\mathcal{A}_R\|\nabla {\bf e }^{n+\frac12}\|_{\mathbb{L}^2(\mt)^4}^2\big].
\end{equation}
We now apply \eqref{eq:residual-bound} and
\eqref{eq:noise-corr-bound2} together with
\eqref{eq:Econv-Young} to the right-hand side of \eqref{eq:err-local} to obtain for any $\delta\in(0,\nu)$
\begin{align*}
\frac{1}{2\tau}\mathbb E\big[\mathcal{A}_R(\|{\bf e }_{n+1}\|_{\mathbb{L}^2(\mt)^2}^2-\|{\bf e }_n&\|_{\mathbb{L}^2(\mt)^2}^2)\big]
+ \nu\,\mathbb E\big[\mathcal{A}_R\|\nabla {\bf e }^{n+\frac12}\|_{\mathbb{L}^2(\mt)^4}^2\big]
\le 2\delta\,\mathbb E\big[\mathcal{A}_R\|\nabla {\bf e }^{n+\frac12}\|_{\mathbb{L}^2(\mt)^4}^2\big]\\
&\quad + C(\delta)\,R^2\,\big(\tau^3+\mathbf{1}_{\{n=0\}}\tau^2\big)
 + C(\delta)\,R^2\,\mathbb E\big[\mathcal{A}_R\|{\bf e }_\star^{n+\frac12}\|_{\mathbb{L}^2(\mt)^2}^2\big].
\end{align*}
By choosing, for instance, $\delta=\nu/4$, we absorb the gradient term on
the right-hand side into the left-hand side and obtain for all $n\ge 0$,
\begin{align*}
\frac{1}{2\tau}\mathbb E\big[\mathcal{A}_R(\|{\bf e }_{n+1}\|_{\mathbb{L}^2(\mt)^2}^2-\|{\bf e }_n&\|_{\mathbb{L}^2(\mt)^2}^2)\big]
+ \frac{\nu}{2}\,\mathbb E\big[\mathcal{A}_R\|\nabla {\bf e }^{n+\frac12}\|_{\mathbb{L}^2(\mt)^4}^2\big]\notag
\\&\le C(\delta)R^2\,\big(\tau^3+\mathbf{1}_{\{n=0\}}\tau^2\big)
 + C(\delta)R^2\,\mathbb E\big[\mathcal{A}_R\|{\bf e }_\star^{n+\frac12}\|_{\mathbb{L}^2(\mt)^2}^2\big]
 \end{align*}
which is \eqref{eq:balanced-local}.
\end{proof}

\subsubsection{Summation in time}

By Lemma~\ref{lem:balanced-local}, for every \(n\ge1\),
\begin{align*}
\frac{1}{2\tau}\,\mathbb E\big[\mathcal{A}_R(\|{\bf e }_{n+1}\|_{\mathbb{L}^2(\mt)^2}^2-\|{\bf e }_n\|_{\mathbb{L}^2(\mt)^2}^2)\big]
&+ \frac{\nu}{2}\,\mathbb E\big[\mathcal{A}_R\|\nabla {\bf e }^{n+\frac12}\|_{\mathbb{L}^2(\mt)^4}^2\big]\notag
\\&\le C\,R^2\,\big(\tau^3+\mathbf{1}_{\{n=0\}}\tau^2\big)
+ C\,R^2\,\mathbb E\big[\mathcal{A}_R\|{\bf e }_\star^{\,n+\frac12}\|_{\mathbb{L}^2(\mt)^2}^2\big].
\end{align*}
By summing \eqref{eq:balanced-local} over \(n=1,\dots,m-1\) and using the
definition of ${\bf e }^{n+\frac12}$, we obtain
\begin{align*}
\frac{1}{2\tau}\,\mathbb E\big[\mathcal{A}_R(\|{\bf e }_m\|_{\mathbb{L}^2(\mt)^2}^2-\|{\bf e }_1\|_{\mathbb{L}^2(\mt)^2}^2)\big]
&+ \frac{\nu}{3}\sum_{n=1}^{m-1}
\mathbb E\big[\mathcal{A}_R\|\nabla {\bf e }^{n+\frac12}\|_{\mathbb{L}^2(\mt)^4}^2\big]
\\&\qquad\le C\,R^2\,\tau^3+ C\,R^2\sum_{n=1}^{m-1}\mathbb E\big[\mathcal{A}_R\|{\bf e }_n\|_{\mathbb{L}^2(\mt)^2}^2\big].
\end{align*}
By using a discrete Gronwall inequality, we obtain for any $m>1$
\begin{align}\label{today1}
\mathbb E\big[\mathcal{A}_R\|{\bf e }_m\|_{\mathbb{L}^2(\mt)^2}^2\big]+\nu\,\tau\sum_{n=0}^{m-1}
\mathbb E\big[\mathcal{A}_R\|\nabla {\bf e }^{n+\frac12}\|_{\mathbb{L}^2(\mt)^4}^2\big]
\le {\bf e }^{CR^2}\big(\mathbb E\big[\mathcal{A}_R\|{\bf e }_1\|_{\mathbb{L}^2(\mt)^2}^2\big]+ \tau^{3}\big).
\end{align}
For $m=1$, Lemma~\ref{lem:balanced-local} yields
\begin{align}\label{today2}
\mathbb E\big[\mathcal{A}_R \|{\bf e }_1\|_{\mathbb{L}^2(\mt)^2}^2\big]+\nu\,\tau
\mathbb E\big[\mathcal{A}_R\|\nabla {\bf e }^{\frac12}\|_{\mathbb{L}^2(\mt)^4}^2\big]
\le C\,R^2\tau^{3}.
\end{align}
By combining \eqref{today1}--\eqref{today2}, it completes
the proof of Theorem~\ref{thm:main-local}.
\end{proof}

\subsection{Strong rate of convergence for the pressure $p$}\label{strong_pressure}
To derive an error bound for the pressure, we (temporarily) work with the
full velocity space $\mathbb H^1({\mathbb{T}^2})^2$ and the mixed formulation
of the stochastic Navier--Stokes system. We assume the standard continuous
inf--sup (Ladyzhenskaya--Babu\v{s}ka--Brezzi) condition: there exists
$\beta>0$ such that
\begin{equation}\label{eq:inf-sup-cont}
  \beta \,\|q\|_{\mathbb L^2({\mathbb{T}^2})}
  \;\le\;
  \sup_{0\neq {\bf v}\in\mathbb H^1({\mathbb{T}^2})^2}
  \frac{(q,\mathrm{div}\, {\bf v })}{\|\nabla {\bf v}\|_{\mathbb L^2({\mathbb{T}^2})^4}},
  \qquad \forall\,q\in\mathbb L_0^2({\mathbb{T}^2}).
\end{equation}

Let $p(t)$ denote the (continuous) pressure associated with the exact
solution $({\bf y}(t),p(t))$ of the {\em random} Navier--Stokes system~\eqref{eq:random-NS}, and let
$p_{n+1}\in\Qspace$ be the discrete pressure at time level $t_{n+1}$ from
the time--semi-discrete scheme. We define the pressure error
\[
\pi_{n+1}:=\frac{1}{\tau}\int_{t_{n}}^{t_{n+1}}p(t)\,\mathrm{d}t-p_{n+1}\in\Qspace.
\]
We also define the large probability event based on the discrete velocity
component as follows:
\begin{align*}
    \widetilde{\Omega}_R
    := 
\Big\{\omega\in \Omega_R:\,
    \frac{1}{\tau^3e^{cR^2}}\sup_{0\le m\le N-1}\bigg(\tau\sum_{n=1}^m\|{\bf e}^{n+\frac{1}{2}}\|^2_{\mathbb{V}}\bigg)+\sup_{0\le n\le N-1}\|\Phi\mathcal{I}_\ell^W\|^2_{\mathbb{V}}\le R^2\Big\},
\end{align*}
and denote its indicator by
\begin{align*}
    \mathcal{B}_R:=\mathbf{1}_{\widetilde{\Omega}_R}.
\end{align*}
On $\widetilde{\Omega}_R,$ we obtain
\begin{align*}
	\sup_{0\le\,n\le\,N-1}\|{\bf e}^{n+\frac{1}{2}}\|_{\mathbb{V}}^2\le C\tau^2\,R^2\,e^{cR^2}\le C\tau^2\,e^{cR^2}.
\end{align*}
It gives that
\begin{align*}
	\sup_{0\le\,n\le\,N-1}\|{\bf e}_{n+1}\|_{\mathbb{V}}^2\le\,\sum_{m=0}^{n+1}\|e^{m+\frac{1}{2}}\|_{\mathbb{V}}^2\le \frac{1}{\tau} C\tau^2\,e^{cR^2}\le\,C\tau\,e^{cR^2}.
	\end{align*}
It also provide the following bound,
\begin{align}\label{new}
\sup_{0\le\,n\le N-1}\mathcal{B}_R	\|{\bf y}_\star^{n+\frac12}+\Phi\mathcal I_n^W\|_{\mathbb V}^2\le e^{cR^2}
\end{align}
Note that Theorem \ref{thm:main-local} yields
\begin{align*}
\mathbb P\bigg(\bigg\{  \frac{1}{\tau^3e^{cR^2}}\mathcal A_R\sup_{0\le n\le N-1}\bigg(\tau\sum_{\ell=1}^n\|{\bf e}_{\ell}\|^2_{\mathbb{V}}\bigg)\bigg\}> R^2\bigg)&\leq\frac{1}{R^2}\mathbb E\bigg[  \frac{1}{\tau^3e^{cR^2}}\mathcal{A}_R\bigg(\tau\sum_{\ell=1}^N\|{\bf e}_{\ell}\|^2_{\mathbb{V}}\bigg)\bigg]\\
&\leq \frac{1}{R^2}\rightarrow0
\end{align*}
as $R\rightarrow\infty$.

\begin{theorem}[Second main result]\label{lem:pressure-local}
Under the assumptions \eqref{regularity}, the inf--sup
condition~\eqref{eq:inf-sup-cont}, there exists a constant $C\,>0$,
independent of $\tau$ and $N$, such that
\begin{equation}\label{eq:pressure-rate}
  \tau\sum_{n=0}^{N-1}
     \mathbb E\big[\mathcal{B}_R\|\pi_{n+1}\|_{\mathbb L^2(\mt)}^2\big]
  \;\le\; {\bf e }^{C\,R^2}\,\tau^{3}.
\end{equation}
\end{theorem}
\begin{remark}[Time average of the analytical pressure] \label{pressure_remark}
		In Theorem~\ref{lem:pressure-local} we compare the discrete pressure $p^{n+1}$ with the
	time average of the analytical pressure over the interval $(t_n,t_{n+1}]$, i.e.,
	\[
	\frac{1}{\tau}\int_{t_n}^{t_{n+1}} p(t)\,dt,
	\qquad \tau:=t_{n+1}-t_n.
	\]
	This choice is natural because, after applying the random transformation
	(cf.~\eqref{zwei}), one cannot, in general, expect the pressure to possess a meaningful
	time derivative. Moreover, the pressure itself is unchanged by the transformation:
	it is the same pressure that appears in SPDE~\eqref{eq:stochastic-NS} and in the
	corresponding random PDE~\eqref{eq:random-NS}. Consequently, it is not appropriate to
	compare $p^{n+1}$ with the pointwise values $p(t_{n+1})$ or $p(t_n)$ in order to obtain
	a higher convergence rate. The time-averaged pressure is the correct quantity that
	can be estimated with the available temporal regularity.

\end{remark}

The remainder of this subsection is devoted to the proof of
Theorem~\ref{lem:pressure-local}. 
\subsubsection{Discrete time derivative estimate}

\begin{lemma}[Discrete time derivative in $\mathbb V^{-1}$]\label{lem:dtd-H-1}
Let the assumptions~\eqref{regularity} hold. Then there exists a
constant $C>0$, independent of $\tau$ and $N$, such that
\begin{equation}\label{eq:dtd-H-1}
  \tau\sum_{n=0}^{N-1}
  \mathbb E\Big[\mathcal{B}_R\,\frac{\|{\bf e }_{n+1}-{\bf e }_n\|_{\mathbb V^{-1}}^2}{\tau^2}\Big]
  \;\le\; {\bf e }^{C\,R^2}\,\tau^3.
\end{equation}
\end{lemma}

\begin{proof}
We prove this result in several steps.

\medskip\noindent
\emph{Step 1: error equation with general test function.}
For $n\ge1$, integrating the continuous equation over $I_n=[t_n,t_{n+1}]$,
testing with an arbitrary ${\bf v}\in \mathbb V$, and subtracting the discrete step
\eqref{scheme:main}, we obtain the error identity
\begin{align}\label{eq:err-dtd}
\Big(\frac{{\bf e }_{n+1}-{\bf e }_n}{\tau},{\bf v }\Big)
&+ \nu\,(\nabla {\bf e }^{n+\frac12},\nabla {\bf v })
+ \mathcal C\big({\bf y}_\star^{n+\frac12}+\Phi\mathcal I_n^W,\,
                 {\bf e }^{n+\frac12},{\bf v }\Big)\notag\\
&\quad + \mathcal C({\bf e }_\star^{\,n+\frac12},\bar {{\bf Y}}^{n+\frac12},{\bf v })\notag\\
&= \big\langle\widetilde{\mathcal R}^{n+\frac12},{\bf v }\Big\rangle-\left<\mathcal{N}^{n+\frac12}, {\bf v}\right>,
\end{align}
where $\widetilde{\mathcal R}^{n+\frac12}$ is the modified residual
\eqref{modified-residual-2} and $\mathcal{N}^{n+\frac12}$ is the combined noise \eqref{combined noise}.
As in the velocity error analysis, we introduce the second modified
residual
\[
\big\langle\widetilde{\mathcal R}_1^{n+\frac12},{\bf v }\Big\rangle
:= \big\langle\widetilde{\mathcal R}^{n+\frac12},{\bf v }\Big\rangle
 -\left<\mathcal{N}^{n+\frac12},{\bf v}\right>,
\]
so that \eqref{eq:err-dtd} can be written as
\begin{align}\label{eq:err-dtd-tildeR}
\Big(\frac{{\bf e }_{n+1}-{\bf e }_n}{\tau},{\bf v }\Big)
&= - \nu\,(\nabla {\bf e }^{n+\frac12},\nabla {\bf v })
    - \mathcal C\big({\bf y}_\star^{n+\frac12}+\Phi\mathcal I_n^W,\,
                     {\bf e }^{n+\frac12},{\bf v }\Big)\notag\\
&\quad - \mathcal C({\bf e }_\star^{\,n+\frac12},{{\bf Y}}^{n+\frac12},{\bf v })
    + \big\langle\widetilde{\mathcal R}_1^{n+\frac12},{\bf v }\Big\rangle,
\qquad \forall {\bf v}\in \mathbb V.
\end{align}

\medskip\noindent
\emph{Step 2: $\mathbb V^{-1}$--bound for the discrete time derivative.}
By definition of the $\mathbb V^{-1}$--norm,
\[
\Big\|\frac{{\bf e }_{n+1}-{\bf e }_n}{\tau}\Big\|_{\mathbb V^{-1}}
= \sup_{0\neq {\bf v}\in \mathbb V}\frac{\big|\big(\frac{{\bf e }_{n+1}-{\bf e }_n}{\tau},{\bf v }\Big)\big|}
                            {\|\nabla {\bf v}\|_{\mathbb{L}^2(\mt)^4}}.
\]
By using \eqref{eq:err-dtd-tildeR} and estimating each term on the
right--hand side, we obtain on $\widetilde{\Omega}_R$:
\begin{itemize}[leftmargin=2em]
  \item Viscous term: we have
  \[
  \nu\,|(\nabla {\bf e }^{n+\frac12},\nabla {\bf v })|
  \le \nu\,\|\nabla {\bf e }^{n+\frac12}\|_{\mathbb{L}^2(\mt)^4}\,\|\nabla {\bf v}\|_{\mathbb{L}^2(\mt)^4}.
  \]
  \item Transport term with ${\bf y}_\star^{n+\frac12}+\Phi\mathcal I_n^W$:
 by using the continuity of $\mathcal C$ in transport form~\eqref{eq:C-cont} and bound~\eqref{new}
  on $\widetilde{\Omega}_R$,
  \begin{align*}
  \big|\mathcal C({\bf y}_\star^{n+\frac12}+\Phi\mathcal I_n^W,{\bf e }^{n+\frac12},{\bf v })\big|
&\le C\,\|{\bf y}_\star^{n+\frac12}+\Phi\mathcal I_n^W\|_{\mathbb{V}}\|{\bf e }^{n+\frac12}\|_{\mathbb{V}}\,\|\nabla {\bf v}\|_{\mathbb{L}^2} \\& \le C\sqrt{e^{c{R}^2}}\|{\bf e }^{n+\frac12}\|_{\mathbb{V}}\,\|\nabla {\bf v}\|_{\mathbb{L}^2(\mt)^4}.
  \end{align*}
  \item Convection defect term:
  again by \eqref{eq:C-cont} and \eqref{eq:UR-H1-bound}, we obtain
  \[
  |\mathcal C({\bf e }_\star^{\,n+\frac12},\bar {{\bf Y}}^{n+\frac12},{\bf v })|
  \le C\,R\,\|{\bf e }_\star^{\,n+\frac12}\|_{\mathbb{L}^2(\mt)^2}\,\|\nabla {\bf v}\|_{\mathbb{L}^2(\mt)^4}.
  \]
  \item Modified residual term: for
  any $v\in V$, we have
  \[
  \big|\big\langle\widetilde{\mathcal R}_1^{n+\frac12},{\bf v }\Big\rangle\big|
  \le \|\widetilde{\mathcal R}_1^{n+\frac12}\|_{\mathbb{V}^{-1}}\,
  \|\nabla {\bf v}\|_{\mathbb L ^2}.
  \]
\end{itemize}

By combining these estimates in \eqref{eq:err-dtd-tildeR} and dividing by
$\|\nabla v\|_{\mathbb{L}^2}$, we obtain on $\widetilde{\Omega}_R$:
\begin{align}\label{eq:dtd-local-square}
\begin{aligned}
\Big\|\frac{{\bf e }_{n+1}-{\bf e }_n}{\tau}\Big\|_{\mathbb V^{-1}}^2
&\le C\,\Big(
 \|\nabla {\bf e }^{n+\frac12}\|_{\mathbb{L}^2(\mt)^4}^2
 +e^{cR^2}\,\|{\bf e }^{n+\frac12}\|_{\mathbb{V}}^2\\
 & +R^2\|{\bf e }^{n+\frac12}\|_{\mathbb{V}}^2+\|{\bf e }_\star^{\,n+\frac12}\|_{\mathbb{L}^2(\mt)^2}^2
 +\|\widetilde{\mathcal R}_1^{n+\frac12}\|_{\mathbb{V}^{-1}}^2\Big),
\end{aligned}
\end{align}
where inequality~\eqref{new} is used.

\medskip\noindent
\emph{Step 3: summation in time and use of the velocity error bounds.}
We now take expectations in \eqref{eq:dtd-local-square}, multiply by
$\tau$, and sum over $n=1,\dots,N-1$:
\begin{align}\label{38}
\tau\sum_{n=1}^{N-1}
\mathbb E\Big[\mathcal{B}_R\Big\|\frac{{\bf e }_{n+1}-{\bf e }_n}{\tau}\Big\|_{\mathbb V^{-1}}^2\Big]
&\le C\,R^2\,\tau\sum_{n=0}^{N-1}
\mathbb E\big[\mathcal{B}_R\|\nabla {\bf e }^{n+\frac12}\|_{\mathbb{L}^2(\mt)^4}^2\big]\notag\\
&\quad + C\,e^{cR^2}\,R^2\,\tau\sum_{n=0}^{N-1}
\mathbb E\big[\mathcal{B}_R\|{\bf e }^{n+\frac12}\|_{\mathbb{V}}^2\big]\notag\\
&\quad + C\,R^2\,\tau\sum_{n=0}^{N-1}
\mathbb E\big[\mathcal{B}_R\|{\bf e }_\star^{\,n+\frac12}\|_{\mathbb{L}^2(\mt)^2}^2\big]\notag\\
&\quad + \tau\sum_{n=0}^{N-1}\mathbb{E}\big[\mathcal{B}_R \|\widetilde{\mathcal R}_1^{n+\frac12}\|_{\mathbb{V}^{-1}}^2\big].
\end{align}
By Theorem~\ref{thm:main-local},
\[
\max_{0\le n\le N}\mathbb E\big[\mathcal{B}_R\|{\bf e }_n\|_{\mathbb{L}^2(\mt)^2}^2\big]
\le {\bf e }^{C\,R^2}\tau^3,
\qquad
\tau\sum_{n=0}^{N-1}
\mathbb E\big[\mathcal{B}_R\|\nabla {\bf e }^{n+\frac12}\|_{\mathbb{L}^2(\mt)^4}^2\big]
\le {\bf e }^{C\,R^2}\tau^3.
\]
Moreover, for the extrapolated error
${\bf e }_\star^{\,n+\frac12}=\tfrac32 {\bf e }_n-\tfrac12 {\bf e }_{n-1}$ one has
\[
\|{\bf e }_\star^{\,n+\frac12}\|_{\mathbb{L}^2(\mt)^2}^2
\le C\big(\|{\bf e }_n\|_{\mathbb{L}^2(\mt)^2}^2+\|{\bf e }_{n-1}\|_{\mathbb{L}^2(\mt)^2}^2\big),
\]
so that
\[
\tau\sum_{n=0}^{N-1}
\mathbb E\big[\mathcal{B}_R\|{\bf e }_\star^{\,n+\frac12}\|_{\mathbb{L}^2(\mt)^2}^2\big]
\le {\bf e }^{C\,R^2}\,\tau^3.
\]
By putting everything together with the estimate \eqref{20} for the
second modified residual $\widetilde{\mathcal{R}}_1^{n+\frac12}$ in the
error inequality \eqref{38}, we deduce
\[
\tau\sum_{n=0}^{N-1}
\mathbb E\Big[\mathcal{B}_R\Big\|\frac{{\bf e }_{n+1}-{\bf e }_n}{\tau}\Big\|_{\mathbb V^{-1}}^2\Big]
\le {\bf e }^{C\,R^2}\,\tau^3,
\]
which is exactly \eqref{eq:dtd-H-1}. This completes the proof.
\end{proof}

\subsubsection{Proof of Theorem~\ref{lem:pressure-local}}

\begin{proof}
Fix $n\in\{0,\dots,N-1\}$ and work on the set $\widetilde{\Omega}_R$, where
$\mathcal B_R=1$ and where we have the uniform $\mathbb H^2({\mathbb{T}^2})$--bound for
${\bf {Y}}(t)={\bf y}(t)+\Phi W(t)$, see \eqref{eq:UR-H1-bound}. All bounds below hold
pathwise on $\widetilde{\Omega}_R$, with constants depending on $T, \nu$ but
not on $n$ or~$\tau$.

\medskip\noindent
\emph{Step 1: inf--sup applied to the instantaneous pressure error.}
For each $\omega\in\widetilde{\Omega}_R$, the continuous inf--sup condition
\eqref{eq:inf-sup-cont} gives
\begin{equation}\label{eq:pi-infsup-pathwise}
  \beta\,\|\pi_{n+1}(\omega)\|_{\mathbb L^2({\mathbb{T}^2})}
  \;\le\;
  \sup_{0\neq {\bf v}\in\mathbb H^1({\mathbb{T}^2})^2}
  \frac{(\pi_{n+1}(\omega),\mathrm{div}\,{\bf v })}{\|\nabla {\bf v}\|_{\mathbb L^2({\mathbb{T}^2})^2}}.
\end{equation}
Hence, to bound $\|\pi_{n+1}(\omega)\|$ it suffices to estimate
\[
(\pi_{n+1},\mathrm{div}\, {\bf v })
=\Big(\frac{1}{\tau}\int_{t_n}^{t_{n+1}}p(t)\,\mathrm{d}t-p_{n+1},\mathrm{div}\, {\bf v }\Big)
\]
for arbitrary ${\bf v}\in\mathbb H^1({\mathbb{T}^2})^2$.

\medskip\noindent
\emph{Step 2: error identity with a general test function.}
We briefly sketch the standard derivation of a pressure error identity;
the structure is the same as in the velocity error analysis,
but we do not restrict the test function to be divergence-free.

On the one hand, integrating the continuous momentum balance over
$I_n=[t_n,t_{n+1}]$, writing it in terms of ${\bf {Y}}(t)={\bf y}(t)+\Phi W(t)$, and
testing by an arbitrary ${\bf v}\in\mathbb H^1({\mathbb{T}^2})^2$ yields
\begin{align}\label{eq:cont-average-mom}
\Big(\frac{{\bf y}(t_{n+1})-{\bf y}(t_n)}{\tau},{\bf v }\Big)
&+ \frac{1}{\tau}\int_{t_n}^{t_{n+1}}\mathcal C({\bf {Y}}(t),{\bf {Y}}(t),{\bf v })\,\mathrm{d}t
 + \nu\,\frac{1}{\tau}\int_{t_n}^{t_{n+1}}(\nabla {\bf {Y}}(t),\nabla {\bf v })\,\mathrm{d}t\notag\\
&\quad - \frac{1}{\tau}\int_{t_n}^{t_{n+1}}(p(t),\mathrm{div}\, {\bf v })\,\mathrm{d}t
= 0.
\end{align}

On the other hand, the discrete scheme at step $n$ (for $n\ge1$) gives,
for all ${\bf v}\in\mathbb H^1({\mathbb{T}^2})^2$,
\begin{align}\label{eq:disc-mom-general}
\Big(\frac{{\bf y}_{n+1}-{\bf y}_n}{\tau}, {\bf v }\Big)
&+ \mathcal C\!\big({\bf y}_\star^{n+\frac12}+\Phi\mathcal I_n^W,\;
                     {{\bf y}}^{n+\frac12}+\Phi\mathcal I_n^W,\;
                     {\bf v }\Big)
 -\big(\mathcal I_n^{W^2},\nabla {\bf v }\big)\notag\\
&\quad
 + \nu\,(\nabla({{\bf y}}^{n+\frac12}+\Phi\mathcal I_n^W),\nabla {\bf v })
 - (p_{n+1},\mathrm{div}\, {\bf v })
 = 0.
\end{align}

By subtracting \eqref{eq:cont-average-mom} from \eqref{eq:disc-mom-general}
and using ${\bf e }_n={\bf y}(t_n)-{\bf y}_n$, ${\bf e }^{n+\frac12}=\tfrac12({\bf e }_{n+1}+{\bf e }_n)$, we
obtain the error identity
\begin{align}\label{eq:err-press-raw}
\Big(\frac{{\bf e }_{n+1}-{\bf e }_n}{\tau},{\bf v }\Big)
&+ \nu\,(\nabla {\bf e }^{n+\frac12},\nabla {\bf v })
+ \mathcal C\big({\bf y}_\star^{n+\frac12}+\Phi\mathcal I_n^W,\,
                 {\bf e }^{n+\frac12},{\bf v }\Big)\notag\\
&\quad + \mathcal C({\bf e }_\star^{\,n+\frac12},\bar {{\bf Y}}^{n+\frac12},{\bf v })\notag\\
&\quad - \Big(\frac{1}{\tau}\int_{t_n}^{t_{n+1}}p(t)\,\mathrm{d}t-p_{n+1},\mathrm{div}\, {\bf v }\Big)\notag\\
&= \big\langle\widetilde{\mathcal R}_1^{n+\frac12},{\bf v }\Big\rangle,
\end{align}
where $\widetilde{\mathcal R}_1^{n+\frac12}$ is the second modified residual
from \eqref{20}, and
\[
{\bf e }_\star^{\,n+\frac12}
:=\tfrac32 {\bf e }_n-\tfrac12 {\bf e }_{n-1},\qquad
\bar {{\bf Y}}^{n+\frac12}=\frac{1}{\tau}\int_{t_n}^{t_{n+1}}{\bf {Y}}(t)\,\mathrm{d}t.
\]
Rearranging \eqref{eq:err-press-raw} gives, for all
${\bf v}\in\mathbb H^1({\mathbb{T}^2})^2$,
\begin{align}\label{eq:pi-v-div}
(\pi_{n+1},\mathrm{div}\, {\bf v })
&= \Big(\frac{{\bf e }_{n+1}-{\bf e }_n}{\tau},{\bf v }\Big)
 + \nu\,(\nabla {\bf e }^{n+\frac12},\nabla {\bf v })\notag\\
&\quad + \mathcal C\big({\bf y}_\star^{n+\frac12}+\Phi\mathcal I_n^W,\,
                        {\bf e }^{n+\frac12},{\bf v }\Big)
       + \mathcal C({\bf e }_\star^{\,n+\frac12},\bar {{\bf Y}}^{n+\frac12},{\bf v })\notag\\
&\quad - \big\langle\widetilde{\mathcal R}_1^{n+\frac12},{\bf v }\Big\rangle.
\end{align}

\medskip\noindent
\emph{Step 3: pathwise bound of the right-hand side.}
We now bound each term on the right-hand side of \eqref{eq:pi-v-div} on
$\widetilde{\Omega}_R$, keeping track of $\|\nabla {\bf v}\|_{\mathbb{L}^2(\mt)^4}$.

\smallskip\noindent
$\bullet$ Time derivative term:
by Cauchy--Schwarz and the definition of the $\mathbb V^{-1}$ norm, since $\mathrm{div}({\bf e }_{n+1}-{\bf e }_n)=0$ and by continuity of $\nabla\Delta^{-1}\mathrm{div}$
\begin{align*}
\Big|\Big(\frac{{\bf e }_{n+1}-{\bf e }_n}{\tau},{\bf v }\Big)\Big|&=\Big|\Big(\frac{{\bf e }_{n+1}-{\bf e }_n}{\tau},{\bf v }-\nabla\Delta^{-1}\mathrm{div}{\bf v}\Big)\Big|\\
&\le \frac{\|{\bf e }_{n+1}-{\bf e }_n\|_{\mathbb V^{-1}}}{\tau}\,
     \|\nabla ({\bf v}-\nabla\Delta^{-1}\mathrm{div}\,{\bf v})\|_{\mathbb{L}^2(\mt)^4}\\
&\le\,C \frac{\|{\bf e }_{n+1}-{\bf e }_n\|_{\mathbb V^{-1}}}{\tau}\,
     \|\nabla {\bf v}\|_{\mathbb{L}^2(\mt)^4}.
\end{align*}

\smallskip\noindent
$\bullet$ Viscous term: we have
\[
|\nu(\nabla {\bf e }^{n+\frac12},\nabla {\bf v })|
\le \nu\,\|\nabla {\bf e }^{n+\frac12}\|_{\mathbb{L}^2(\mt)^4}\,\|\nabla {\bf v}\|_{\mathbb{L}^2(\mt)^2}.
\]

\smallskip\noindent
$\bullet$ Transport terms:
by using the continuity estimate \eqref{eq:C-cont} and the localization
bound \eqref{eq:UR-H1-bound} for $U$ on $\widetilde{\Omega}_R$,
\begin{align*}
\big|\mathcal C({\bf y}_\star^{n+\frac12}+\Phi\mathcal I_n^W,{\bf e }^{n+\frac12},{\bf v })\big|
&\le C\,\|{\bf y}_\star^{n+\frac12}+\Phi\mathcal I_n^W\|_{\mathbb{V}}\,
       \|{\bf e }^{n+\frac12}\|_{\mathbb{V}}\,
       \|\nabla {\bf v}\|_{\mathbb{L}^2(\mt)^4}
\end{align*}
and similarly we get
\[
|\mathcal C({\bf e }_\star^{\,n+\frac12},\bar {{\bf Y}}^{n+\frac12},{\bf v })|
\le C\,R\,\|{\bf e }_\star^{\,n+\frac12}\|_{\mathbb{L}^2(\mt)^2}\,\|\nabla {\bf v}\|_{\mathbb{L}^2(\mt)^4}.
\]

\smallskip\noindent
$\bullet$ Residual term: we have
\[
\big|\big\langle\widetilde{\mathcal R}_1^{n+\frac12},{\bf v }\Big\rangle\big|
\le \|\widetilde{\mathcal R}_1^{n+\frac12}\|_{\mathbb H^{-1}({\mathbb{T}^2})}\,
    \|\nabla {\bf v}\|_{\mathbb{L}^2(\mt)^4}.
\]

\smallskip\noindent
By collecting all contributions, we obtain the pathwise inequality
\begin{align}\label{eq:pi-v-bound}
\frac{|(\pi_{n+1},\mathrm{div}\,{\bf v })|}{\|\nabla {\bf v}\|_{\mathbb{L}^2(\mt)^4}}
&\le C\,R\bigg(
   \frac{\|{\bf e }_{n+1}-{\bf e }_n\|_{\mathbb H^{-1}}}{\tau}
 + \|\nabla {\bf e }^{n+\frac12}\|_{\mathbb{L}^2(\mt)^4}
 +\|{\bf y}_\star^{n+\frac12}+\Phi\mathcal I_n^W\|_{\mathbb V} \|{\bf e }^{n+\frac12}\|_{\mathbb{V}}\notag
 \\&\qquad\qquad+ \|{\bf e }_\star^{\,n+\frac12}\|_{\mathbb{L}^2(\mt)^2}
 + \|\widetilde{\mathcal R}_1^{n+\frac12}\|_{\mathbb H^{-1}}\bigg),
\end{align}
for all ${\bf v}\in\mathbb H^1({\mathbb{T}^2})^2$ and all $\omega\in\widetilde{\Omega}_R$.

\medskip\noindent
\emph{Step 4: inf--sup and expected localized bound.}
By taking the supremum in \eqref{eq:pi-v-bound} over all nonzero
$v\in\mathbb H^1({\mathbb{T}^2})^2$ and using
\eqref{eq:pi-infsup-pathwise}, we obtain on $\widetilde{\Omega}_R$,
\begin{align*}
\|\pi_{n+1}\|_{\mathbb L^2({\mathbb{T}^2})}
&\le C\,R\Big(
   \frac{\|{\bf e }_{n+1}-{\bf e }_n\|_{\mathbb H^{-1}}}{\tau}
 + \|\nabla {\bf e }^{n+\frac12}\|_{\mathbb{L}^2(\mt)^4}
 +\|{\bf y}_\star^{n+\frac12}+\Phi\mathcal I_n^W\|_{\mathbb V} \|{\bf e }^{n+\frac12}\|_{\mathbb{V}}\\&\qquad
 + \|{\bf e }_\star^{\,n+\frac12}\|_{\mathbb{L}^2}
 + \|\widetilde{\mathcal R}_1^{n+\frac12}\|_{\mathbb H^{-1}}\Big).
\end{align*}
By multiplying by $\mathcal B_R$, squaring, and using
$(x_1+\dots+x_k)^2\le k(x_1^2+\dots+x_k^2)$, we find
\begin{align}\label{eq:pi-local-square}
\mathcal B_R\|\pi_{n+1}\|_{\mathbb{L}^2(\mt)}^2
&\le C\,R^2\mathcal B_R\Big(
   \frac{\|{\bf e }_{n+1}-{\bf e }_n\|_{\mathbb H^{-1}}^2}{\tau^2}
 + \|\nabla {\bf e }^{n+\frac12}\|_{\mathbb{L}^2(\mt)^4}^2
 + \|{\bf y}_\star^{n+\frac12}+\Phi\mathcal I_n^W\|_{\mathbb V}^2\|{\bf e }^{n+\frac12}\|_{\mathbb{V}}^2\notag\\
&\hspace{3cm}
 + \|{\bf e }_\star^{\,n+\frac12}\|_{\mathbb{L}^2(\mt)^2}^2
 + \|\widetilde{\mathcal R}_1^{n+\frac12}\|_{\mathbb H^{-1}}^2\Big).
\end{align}

We now take expectations and sum over $n$. The velocity error
estimates from Theorem~\ref{thm:main-local} give
\[
\max_{0\le n\le N}
\mathbb E\big[\mathcal{B}_R\|{\bf e }_n\|_{\mathbb{L}^2(\mt)^2}^2\big]\le {\bf e }^{C\,R^2}\tau^3,
\qquad
\tau\sum_{n=0}^{N-1}
\mathbb E\big[\mathcal{B}_R\|\nabla {\bf e }^{n+\frac12}\|_{\mathbb{L}^2(\mt)^4}^2\big]
\le {\bf e }^{C\,R^2}\tau^3.
\]
From these and standard estimates for the discrete time derivative and
BDF2 extrapolation (by using that
${\bf e }_\star^{\,n+\frac12}=\tfrac32 {\bf e }_n-\tfrac12 {\bf e }_{n-1}$), together with
Lemma~\ref{lem:dtd-H-1} and the residual estimate \eqref{20}, we obtain
\[
\tau\sum_{n=0}^{N-1}
\mathbb E\Big[\mathcal{B}_R\frac{\|{\bf e }_{n+1}-{\bf e }_n\|_{\mathbb H^{-1}}^2}{\tau^2}\Big]
\le {\bf e }^{C\,R^2}\tau^3,
\qquad
\tau\sum_{n=0}^{N-1}
\mathbb E\big[\mathcal{B}_R\|{\bf e }_\star^{\,n+\frac12}\|_{\mathbb{L}^2(\mt)^2}^2\big]
\le {\bf e }^{C\,R^2}\,\tau^3.
\]
By combining these bounds with \eqref{new} and summing \eqref{eq:pi-local-square} over
$n=0,\dots,N-1$, we obtain
\[
\tau\sum_{n=0}^{N-1}
\mathbb E\big[\mathcal{B}_R\|\pi_{n+1}\|_{\mathbb L^2(\mt)}^2\big]
\le {\bf e }^{C\,R^2}\tau^3,
\]
which is exactly \eqref{eq:pressure-rate}. This completes the proof.
\end{proof}
\subsection{In case of general noise coefficient}\label{section_general_noise}
For any vector field
$\mathbf v\in \mathbb{L}^{2}(\mathbb{T}^2;\mathbb{R}^2)$, we define the mean-zero scalar potential operator
\[
Q_{\mathrm{HL}}(\mathbf v):=\Delta^{-1}\operatorname{div}\mathbf v \in \mathbb{L}^{2}_{0}(\mathbb{T}^2; \mathbb{R}^2),
\]
where $\Delta^{-1}$ is the inverse Laplacian on mean-zero functions, i.e.,
$\phi=\Delta^{-1}f$ is the unique solution of
\[
-\Delta \phi=f \quad \text{in }\mathbb{T}^2,
\qquad \int_{\mathbb{T}^2}\phi\dx=0.
\]
The Helmholtz--Leray projection is then given by
\[
P_{\mathrm{HL}}\mathbf v := \mathbf v - \nabla Q_{\mathrm{HL}}(\mathbf v).
\]
Hence the Helmholtz decomposition reads
\begin{equation}\label{eq:HD-torus}
	\mathbf v = P_{\mathrm{HL}}\mathbf v + \nabla Q_{\mathrm{HL}}(\mathbf v),
	\tag{HD}
\end{equation}
where $\operatorname{div}(P_{\mathrm{HL}}\mathbf v)=0$ and
$\int_{\mathbb{T}^2} Q_{\mathrm{HL}}(\mathbf v)\dx=0$.

\medskip
\noindent
Let $\Phi$ be the (possibly non-divergence-free) noise coefficient and define
\[
\Psi := P_{\mathrm{HL}}\Phi, 
\qquad 
\Theta := Q_{\mathrm{HL}}\Phi,
\]
so that
\begin{equation}\label{eq:Phi-split}
	\Phi = \Psi + \nabla \Theta .
\end{equation}
Introduce the random transformation
\[
\hat{\mathbf y}(t):=\mathbf u(t)-P_{\mathrm{HL}}\Phi\,W(t)
=\mathbf u(t)-\Psi\,W(t).
\]
Then \eqref{req:stochastic-NS} can be rewritten as the following random PDE system: for each time step
$(t_n,t_{n+1}]$,
\begin{equation}\label{eq:random-NS1}
	\begin{cases}
		\hat{\mathbf y}(t_{n+1})-\hat{\mathbf y}(t_n)
		+\displaystyle\int_{t_n}^{t_{n+1}}
		\Big[
		\Big((\hat{\mathbf y}+\Psi W)\cdot\nabla\Big)\Big(\hat{\mathbf y}+\Psi W)\Big)
		-\nu\,\Delta(\hat{\mathbf y}+\Psi W)
		\Big]\dt
		+\displaystyle\int_{t_n}^{t_{n+1}}\nabla q(t)\,dt
		=0,\\[1.0ex]
		\operatorname{div}\hat{\mathbf y}(t)=0,\\[0.5ex]
		\hat{\mathbf y}(0)=\mathbf y_0:=\mathbf u_0,
	\end{cases}
\end{equation}
where the modified pressure is defined by
\[
q(t):=p(t)+\frac{1}{\tau}\,\Theta\,\Delta_{n+1}W,
\qquad t\in(t_n,t_{n+1}],
\]
with $\tau:=t_{n+1}-t_n$.

\medskip
\noindent
Since $\hat{\mathbf y}$ is divergence-free, it satisfies the same type of regularity estimates
. In particular, these estimates are sufficient to carry out
the convergence analysis in the proofs of Theorems~\ref{thm:main-local} and~\ref{lem:pressure-local}.

\medskip
\noindent
For the time discretisation of \eqref{eq:random-NS1}, we again introduce a modified discrete pressure
$q^{n+1}$ on $(t_n,t_{n+1}]$ defined by
\[
q^{n+1}:=p^{n+1}+\frac{1}{\tau}\,\Theta\,\Delta_{n+1}W.
\]
Then Theorem~\ref{lem:pressure-local} yields a rate for the quantity
\[
\frac{1}{\tau}\int_{t_n}^{t_{n+1}} q(t)\dt - q^{n+1}.
\]
By the definitions of $q(t)$ and $q^{n+1}$, the noise-induced terms cancel, and we obtain
\[
\frac{1}{\tau}\int_{t_n}^{t_{n+1}} q(t)\dt - q^{n+1}
=
\frac{1}{\tau}\int_{t_n}^{t_{n+1}} p(t)\dt - p^{n+1}.
\]
Therefore, this term is precisely the difference between the mean value of the analytical pressure
over $(t_n,t_{n+1}]$ and the discrete pressure at time $t_{n+1}$.

\medskip
\noindent
Consequently, by proceeding in this way for a non-divergence-free noise coefficient, one can prove
the same type of results as stated in Theorems~\ref{thm:main-local} and~\ref{lem:pressure-local} for
the divergence-free noise case.

\section{Linear stochastic Stokes system with additive noise}
\label{sec:stokes}

For comparison we briefly discuss the linear incompressible Stokes system
with additive noise. In this case there is no convective nonlinearity,
and the error analysis becomes significantly simpler: in particular, no
localization in probability is needed and all constants in the error
bounds are free from large constant $R$.

\subsection{Model problem} We consider the following Linear Stokes system with additive noise
\begin{equation}\label{eq:stokes-spde}
\begin{aligned}
  \begin{cases}
{\rm d}{\bf u}(t) 
+ \big[\nabla p(t)- \nu\,\Delta {\bf u}(t)\big]\,{\rm d}t = \Phi\,{\rm d}W(t),
& (x,t)\in \mt\times (0,T],\\[0.5ex]
\diver {\bf u}(t) = 0,
& (x,t)\in\mt\times(0,T],\\[0.5ex]
{\bf u}(0)={\bf u}_0,
& x\in\mt,
\end{cases}
\end{aligned}
\end{equation}
with the same coefficient field $\Phi$ and Wiener process $W$ as in
the previous section. 
The time--discretisation for \eqref{eq:stokes-spde} is obtained from
our Crank--Nicolson scheme \eqref{scheme:main} by simply dropping all convective and
noise--correction terms.

\subsection{Main error bounds}

The higher--order Brownian quadrature on the fine mesh (Section~\ref{sec: time-discretization})
still improves the consistency error of the random diffusive term and
leads to a strong convergence rate of order $3/2$ in time, now \emph{without}
localization and without any dependence on a truncation parameter $R$.

\begin{theorem}[Velocity error for the linear Stokes system]
\label{thm:stokes-vel}
Let $({\bf u}(t),p(t))$ solve \eqref{eq:stokes-spde} and let
$\{({\bf u}_n,p_n)\}_{n=0}^N$ be the corresponding Crank--Nicolson
approximations obtained from \eqref{scheme:main} by removing all
convective and noise--correction terms. Then there exists a constant
$C>0$, independent of $\tau$ and $N$, such that
\begin{equation}\label{eq:stokes-vel-L2}
  \max_{0\le n\le N}
  \mathbb{E}\bigl[\|{\bf u}(t_n)-{\bf u}_n\|_{\mathbb{L}^2(\mt)^2}^2\bigr]
  \;\le\; C\,\tau^{3},
\end{equation}
and
\begin{equation}\label{eq:stokes-vel-H1}
  \nu\,\tau\sum_{n=0}^{N-1}
  \mathbb{E}\bigl[\|\nabla({\bf u}(t_{n+\frac12})-{\bf u}^{n+\frac12})\|_{\mathbb{L}^2(\mt)^4}^2\bigr]
  \;\le\; C\,\tau^{3},
\end{equation}
where
\[
  {\bf u}\big(t_{n+\frac12}\big) := \tfrac12\bigl({\bf u}(t_n)+{\bf u}(t_{n+1})\bigr),
  \qquad
  {\bf u}^{n+\frac12} := \tfrac12\bigl({\bf u}_{n}+{\bf u}_{n+1}\bigr).
\]
\end{theorem}

For the pressure we obtain a matching bound, in the same spirit as
Theorem~\ref{lem:pressure-local}, but now without localization.

\begin{theorem}[Pressure error for the linear Stokes system]
\label{thm:stokes-press}
Assume in addition the continuous inf--sup condition
\eqref{eq:inf-sup-cont} for the pair $(V,Q)$. Define the pressure error
at level $t_{n+1}$ by
\[
  \pi_{n+1}
  := \frac{1}{\tau}\int_{t_n}^{t_{n+1}} p(t)\,\mathrm{d}t - p_{n+1}
  \in Q.
\]
Then there exists a constant $C>0$, independent of $\tau$ and $N$, such
that
\begin{equation}\label{eq:stokes-press-rate}
  \tau\sum_{n=0}^{N-1}
    \mathbb{E}\bigl[\|\pi_{n+1}\|_{\mathbb{L}^2(\mt)}^2\bigr]
  \;\le\; C\,\tau^{3}.
\end{equation}
\end{theorem}

\begin{remark}
The proofs of Theorems~\ref{thm:stokes-vel} and~\ref{thm:stokes-press}
follow the same structure as in Section~\ref{theo-2}: definition
of a consistency residual, derivation of a discrete energy identity, and
a Gronwall argument for the velocity, followed by an inf--sup argument
for the pressure. However, in the linear Stokes case all convective
terms are absent, no It\^o--type noise correction $I_n^{W^2}$ is needed, and no
localization in probability is required. As a consequence, the arguments
are shorter and all constants $C$ in
\eqref{eq:stokes-vel-L2}--\eqref{eq:stokes-press-rate} are deterministic.
We therefore omit the details.
\end{remark}

\section{Numerical experiments}\label{comstudies2}

We approximate the solution of the random PDE~\eqref{eq:random-NS} using a fully discrete variant of scheme \eqref{scheme:main0}.
For the spatial discretisation we employ the lowest order Taylor-Hood finite element approximation, {\em i.e.}, we use continuous piecewise quadratic approximation for the velocity (discrete space $\mathbb{V}_h$) and
continuous piecewise linear approximation for the pressure (discrete space $\mathbb{Q}_h$).
In all experiments below we consider the spatial domain to be a unit square $\mathcal{D} = (0,1)^2$.
The finite element approximation is considered over a uniform triangulation of $\mathcal{D}$ with mesh size $h=1/L$ which is constructed as follows:
the unit square is divided into $L\times L$ squares with side $h$ and each square is subdivided into four equally shaped triangles with barycenter of the square as a common vertex.

The fully-discrete finite element approximation of \eqref{eq:random-NS} takes the form:
for $n=0, \dots, N-1$, given ${\bf y}_n,{\bf y}_{n-1}\in\mathbb{V}_h$ (with ${\bf y}_{-1}={\bf y}_0$), find
$({\bf y}_h^{n+1},p_h^{n+1})\in\mathbb{V}_h\times\mathbb{Q}_h$ such that, for all
$(\boldsymbol{\varphi},q)\in\mathbb{V}_h\times\mathbb{Q}_h$,
\begin{align}\label{scheme:cn_fem}
\begin{cases}
\displaystyle
\Big(\frac{{\bf y}_h^{n+1}-{\bf y}_h^n}{\tau},\boldsymbol{\varphi}\Big)
+ \mathcal C\!\big({\bf y}_{h,\star}^{n+\frac12}+\Phi\mathcal I_n^W,\;
                   {\bf y}_h^{n+\frac12}+\Phi\mathcal I_n^W,\;
                   \boldsymbol{\varphi}\big)
 -\big(\mathcal{I}_n^{W^2},\nabla\boldsymbol{\varphi}\big)
\\[1ex]\qquad\qquad
+ \nu\,\big(\nabla({\bf y}_h^{n+\frac12}+\Phi\mathcal I_n^W),\nabla\boldsymbol{\varphi}\big)
 - (p_h^{n+1},\mathrm{div}\,\boldsymbol{\varphi})
 =\bigg(\displaystyle\frac{1}{\tau}\int_{t_n}^{t_{n+1}}f(s)\mathrm{d}s,\boldsymbol{\varphi}\bigg),\\[1ex]
(\mathrm{div}\, {\bf y}_h^{n+1},q)=0.
\end{cases}
\end{align}
Below we refer to the above scheme as CN.

For comparison we also consider semi-implicit backward Euler scheme:
for $n=0, \dots, N-1$, given ${\bf y}_n,{\bf y}_{n-1}\in\mathbb{V}_h$, find
$({\bf y}_h^{n+1},p_h^{n+1})\in\mathbb{V}_h\times\mathbb{Q}_h$ such that, for all
$(\boldsymbol{\varphi},q)\in\mathbb{V}_h\times\mathbb{Q}_h$,
\begin{align}\label{scheme:euler_fem}
\begin{cases}
\displaystyle
\Big(\frac{{\bf y}_h^{n+1}-{\bf y}_h^n}{\tau},\boldsymbol{\varphi}\Big) +  \mathcal C^*\!\big({\bf U}_{h}^{n},\;  {\bf y}_h^{n+1}+\Phi W^{n+1},\; \boldsymbol{\varphi}\big)
\\[1ex]\qquad\qquad\qquad
+ \nu\,\big(\nabla({\bf y}_h^{n+1}+\Phi W^{n+1}),\nabla\boldsymbol{\varphi}\big)
 - (p_h^{n+1},\mathrm{div}\,\boldsymbol{\varphi})
 =(f(t_{n+1}),\boldsymbol{\varphi}),\\[1ex]
(\mathrm{div}\, {\bf y}_h^{n+1},q)=0,
\end{cases}
\end{align}
where $W^n := W(t_{n})$, and
$$
\mathcal{C}^*\!\big({\bf u},\;  {\bf v},\; \boldsymbol{\varphi}\big) = \frac{1}{2}\left( \mathcal C\!\big({\bf u},\;  {\bf v},\; \boldsymbol{\varphi}\big)
 - \mathcal C\!\big({\bf v},\;  {\bf u},\; \boldsymbol{\varphi}\big)    \right),
$$
is an approximation of the convective term $\mathcal C(\cdot,\cdot,\cdot)$ (which is equivalent for divergence-free functions) that ensures that the fully discrete scheme satisfies an energy law.

We consider three variants of the Euler scheme (\ref{scheme:euler_fem}) with ${\bf U}_{h}^{n} = {\bf y}_h^{n}+\Phi W^{n}$, ${\bf U}_{h}^{n} = {\bf y}_h^{n}+\Phi W^{n+1}$,
and ${\bf U}_{h}^{n} = {\bf y}_h^{n,*}+\Phi W^{n+1}$, respectively,
where ${\bf y}_h^{n,*}\approx {\bf y}_h^{n+1}$ is obtained by one fixed-point iteration.
Below we refer to these schemes as SI, SIS and IE1, respectively.
We remark that the above SI scheme for the RPDE \eqref{eq:random-NS} is equivalent to the `standard' semi-implicit Euler scheme for the SPDE (\ref{eq:stochastic-NS}); see for instance \cite{BP1}.

\begin{remark}\label{remark_aufwand}
The variant SI of the semi-implicit Euler scheme (\ref{scheme:euler_fem}) has been considered in \cite{BP1} where it is shown that the scheme satisfies a discrete energy law
and exhibits first order convergence with respect to the time step $\tau$.
The semi-implicit Euler scheme requires the solution of one linear system of equations per time-level.

Compared to the semi-implicit Euler-type schemes the CN scheme (\ref{scheme:cn_fem}) has an improved order of convergence $3/2$ with respect to $\tau$.
In addition to the solution of a linear system of equations (with similar computational cost as in the case of the Euler scheme),
at each time-level the CN scheme requires the computation of the stochastic integrals $\mathcal{I}_n^W$, $\mathcal{I}_n^{W^2}$ over a micro-grid with step-size $\tau^2$.
At each time level the computational cost associated with the computation of these integrals is proportional to $M\,K^2$,
with $M=1/\tau$, and $K$ the number of Brownian motions $\{W_k\}_{k=1}^{K}$ (where the quadratic cost $K^2$ is associated with the computation of the integral $\mathcal{I}_n^{W^2}$).
 In the lid-driven cavity flow example where we use $K=4$ and $M=100$
the differences in the computational times were negligible.
In computations which require noise with a large number of modes along with very small time steps,
the additional computations cost of the CN scheme may become more obvious, nevertheless it can be still compensated by the improved accuracy of the CN scheme.

Finally, we also mention that the CN scheme only satisfies the discrete energy law approximately, nevertheless the computational results below indicate
that it enjoys similar stability properties as the semi-implicit Euler scheme, even in the presence of strong noise.
\end{remark}


\subsection{Academic example}\label{academic_example}

We demonstrate the convergence order for the Crank-Nicolson scheme \eqref{scheme:cn_fem} using an academic example with exact solution ${\bf y}_{ex}(t,x) = 2\cos(6t){\bf g}(x)$ with ${\bf g}(x) = (x_1^3, -3x_1^2x_2)^T$,
$\Phi W(t, x) = 4{\bf g}(x) W_1(t)$ and $p(t,x) = t(x_1^2 + x_2^2 -2/3)$. Hence, we approximate \eqref{eq:random-NS} on $(0,T)\times \mathcal{D}$ for $T=1$ with right-hand side
$${\bf f} = \partial_t {\bf y}_{ex} + \big(({\bf y}_{ex}+\Phi W)\cdot\nabla \big)\big({\bf y}_{ex}+\Phi W)\big) - \nu \Delta {\bf y}_{ex} + \nabla p_{ex}$$ and a Dirichlet boundary condition ${\bf y}|_{\partial \mathcal{D}} = {\bf y}_{ex}$.
The Wiener process $W_1(t)$ is generated on a grid with step size $10^{-8}$, and the time integral $\int_{t_n}^{t_{n+1}}{\bf f}(t)\mathrm{d}t$
which needs to be included in scheme \eqref{scheme:main0} is approximated by a Riemann sum over this fine grid.

In Figure~\ref{fig_conv} (left) we display the approximation errors
$$ \mathbb E\big[\max_{0\le n\le N} \|{\bf y}(t_n)-{\bf y}_n\|_{\mathbb L^2(\mt)^2}^2\big],\quad \tau \sum_{n=1}^N\mathbb{E}\big[ \|\pi_n\|_{\mathbb L^2(\mt)}^2\big],$$
for the velocity and the pressure, respectively for $\tau=0.1, 0.05,0.025,0.0125,0.005, 0.00025$ computed with fixed mesh size $h=1/16, 1/32$, respectively.
We observe convergence rate $3/2$ for the velocity and the pressure approximation even for this problem with Dirichlet boundary datum ${\bf y}_{ex}$ until the spatial errors start to dominate.
For comparison we also display results computed with
the Crank-Nicolson scheme \eqref{scheme:main0} without the correction term $\mathcal{I}^{W^2}_n$ (denoted `no IW2'), where we observe that the convergence rate without the correction term deteriorates.
In Figure~\ref{fig_conv} (right) we display the approximation errors of the Stokes equations computed with the CN scheme with $h=1/16, 1/32$, respectively
(for comparison we also include the error plot for the Navier-Stokes equations with $h=1/16$, denoted `NS' in the figure). We observe the predicted convergence rate $3/2$
until the spatial approximation error becomes dominant.
\begin{figure}
\includegraphics[width=0.45\textwidth]{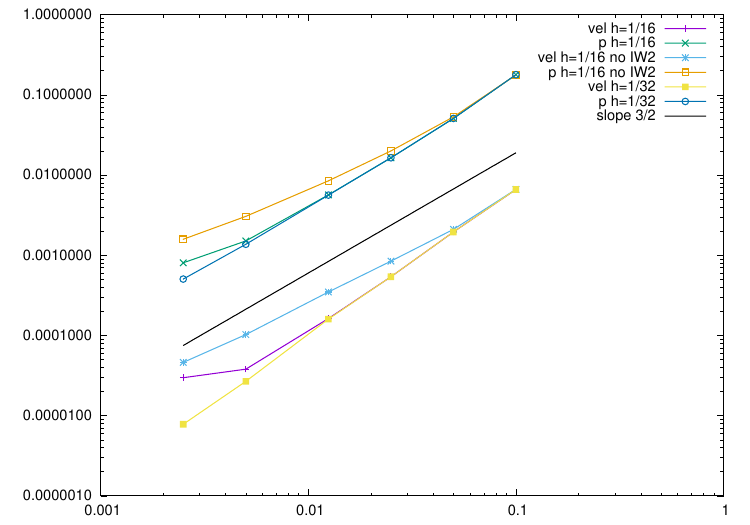}
\includegraphics[width=0.45\textwidth]{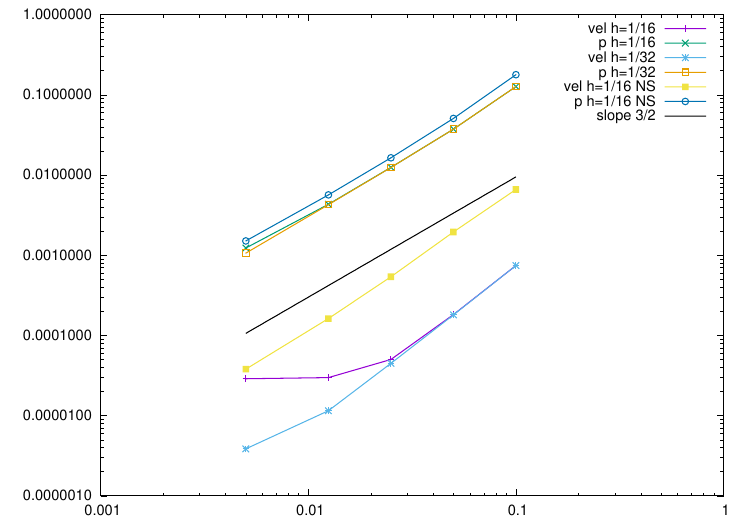}
\caption{Academic Example in Section \ref{academic_example}: approximation error of the CN scheme for the velocity and the pressure for Navier--Stokes equations (left) and Stokes equations (right).}
\label{fig_conv}
\end{figure}
We also compare the CN scheme \eqref{scheme:cn_fem} to the SI, SIS, IE1 variants of the semi-implicit Euler scheme  \eqref{scheme:euler_fem}.
The convergence of the error of the velocity and of the pressure for $\tau=0.1, 0.05,0.025,0.0125,0.005, 0.00025$ with fixed mesh size $h=1/16$ is displayed
in Figure~\ref{fig_conv_euler}. Here we observe a linear convergence of the error for the velocity and the pressure for schemes SIS, IE1.
The rate of the convergence of the SI scheme seems slightly less than linear. The rate for the velocity error of the SI improves for smaller values of $\tau$,
the rate for the pressure error remains suboptimal, which may be attributed to the effect of the boundary condition.
The error of the CN scheme is significantly lower than the error of the semi-implicit Euler schemes; for the largest time step $\tau=0.01$ the velocity error is roughly $2$-times smaller and the pressure error is $2.5$-times
smaller than the respective errors of IE1, which is the best performing variant of the Euler scheme. To achieve similar error as the CN scheme with $\tau=0.005$ (for $h=1/16$)
the IE1 scheme would require the use of roughly $10$-times smaller timestep.
\begin{figure}
\includegraphics[width=0.45\textwidth]{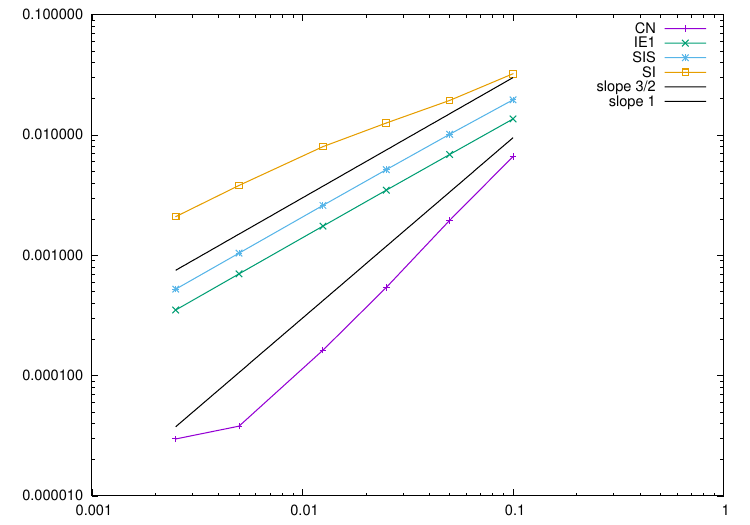}
\includegraphics[width=0.45\textwidth]{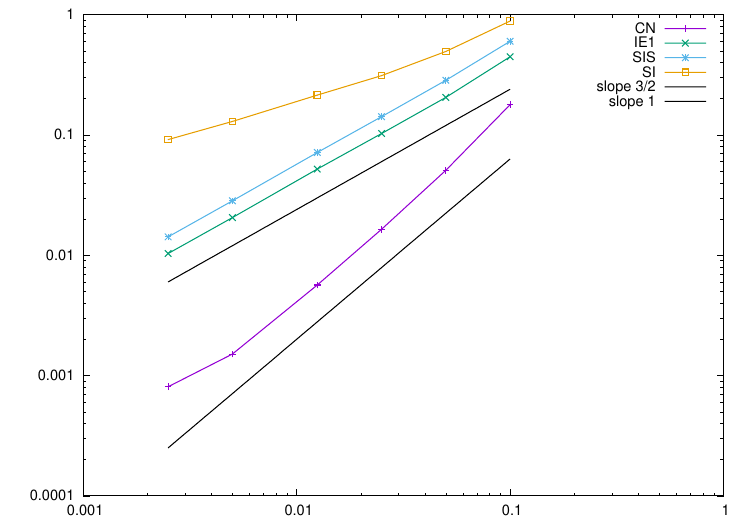}
\caption{Error of the approximation of the velocity (left) and the pressure (right).}
\label{fig_conv_euler}
\end{figure}

\subsection{Lid-driven cavity flow}\label{example_lid_driven}

We test the long-time behavior of scheme \eqref{scheme:cn_fem} using the lid-driven cavity benchmark problem from \cite[Example 5.9]{BPW}.
We take ${\bf v}_0 = {\bf u}_0 = {\bf 0}$, ${\bf f} = {\bf 0}$;
and the noise is taken  to be $\Phi W(t, x) = \mu \sum_{i=1}^4 W_i(t){\bf g}_i(x)$, $\mu>0$
where ${\bf g}_i$ are translated and scaled versions of the vector field $${\bf g}(x) = \mathbf{1}_{\mathcal{D}}(x)\big(\tilde{g}(x_1, x_2), \tilde{g}(x_2, x_1)\big)^T, \quad\tilde{g}(x_1,x_2)= 2x_1^2(1-x_2)^2 x_2(x_2-1)(2x_2-1).$$
We impose a Dirichlet boundary condition ${\bf u}(x) = \mathbf{1}_{\{x_2=1\}}(x)(1,0)^T$ and solve the problem on the time-interval $(0,T)$ with $T=100$.
All presented results are computed with time-step $\tau=0.01$ and mesh size $h=1/16$.
For the given setting the deterministic solution is close to the steady state approximately at the time $t=50$.
In all experiments below the solution computed with the semi-implicit Euler scheme is very similar to the solution of the CN scheme and the corresponding results are therefore not displayed, except for Figure~\ref{fig_euler_path}.
We also note that the differences in computational times for the SI and the CN schemes were negligible in this example.

In Figure~\ref{fig_lid_aver} we display the streamlines of the stationary solution of the deterministic problem (i.e., the solution at the final time $T=100$ for $\mu=0$) computed with the CN scheme
and a time-averaged solution computed with the CN scheme for $\mu=10$ and $\mu=40$ (the time-average is computed over the interval $[50,T]$, the streamlines are colored according to the amplitude of the velocity field).
We only observe minor differences (located close to the bottom of the domain) between the deterministic solution and the solution for $\mu=10$. The solution with the stronger noise $\mu=40$ is significantly different
which highlights the amplified effect of the nonlinear convective term in the stochastic Navier--Stokes equations for stronger noise.
\begin{figure}
\includegraphics[width=0.32\textwidth]{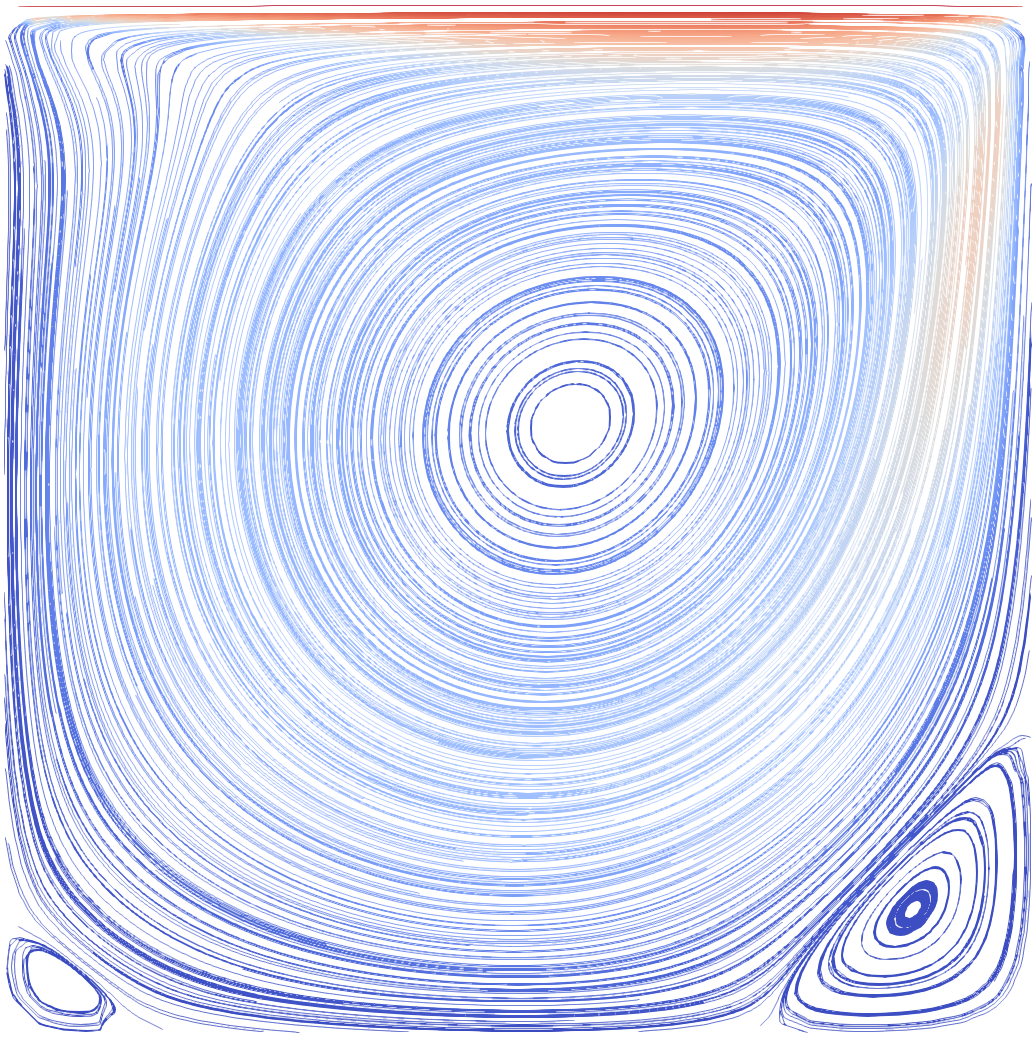}
\includegraphics[width=0.32\textwidth]{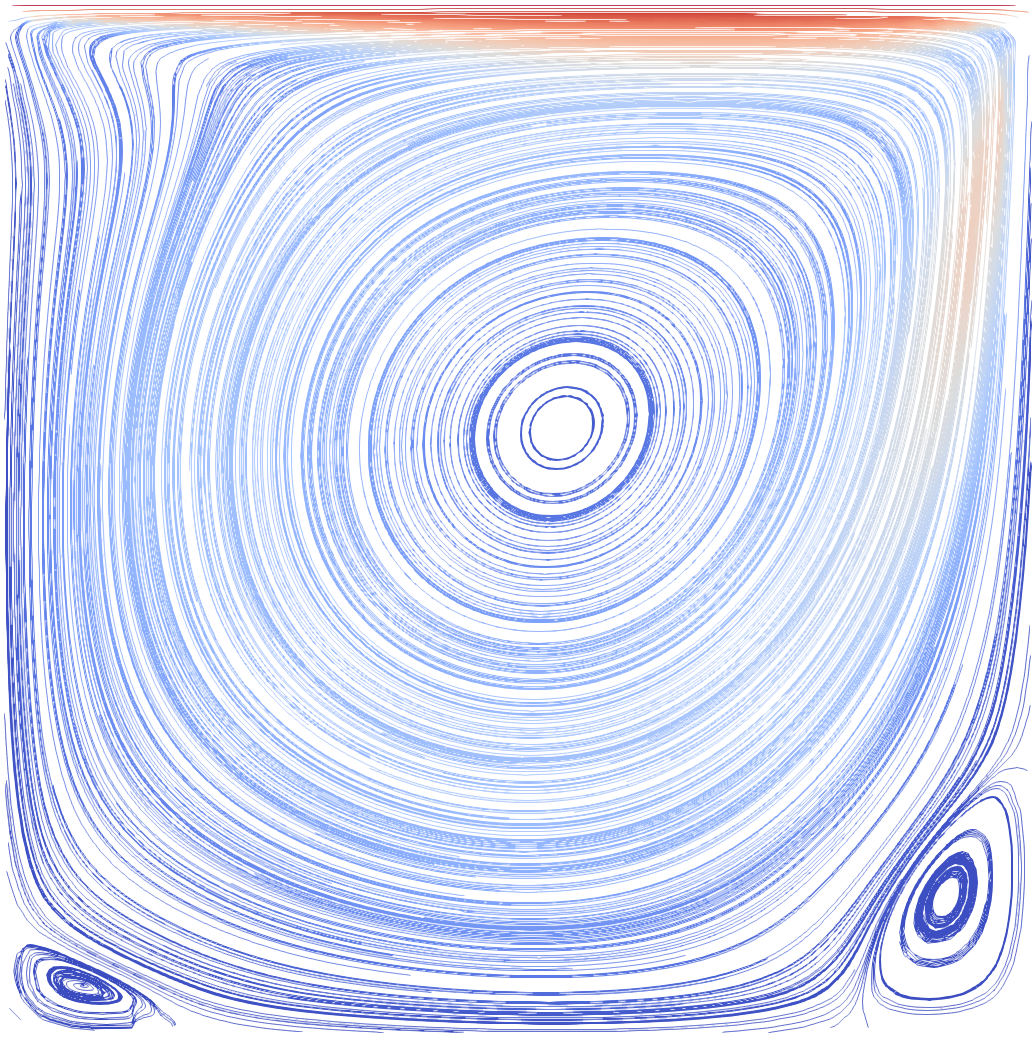}
\includegraphics[width=0.32\textwidth]{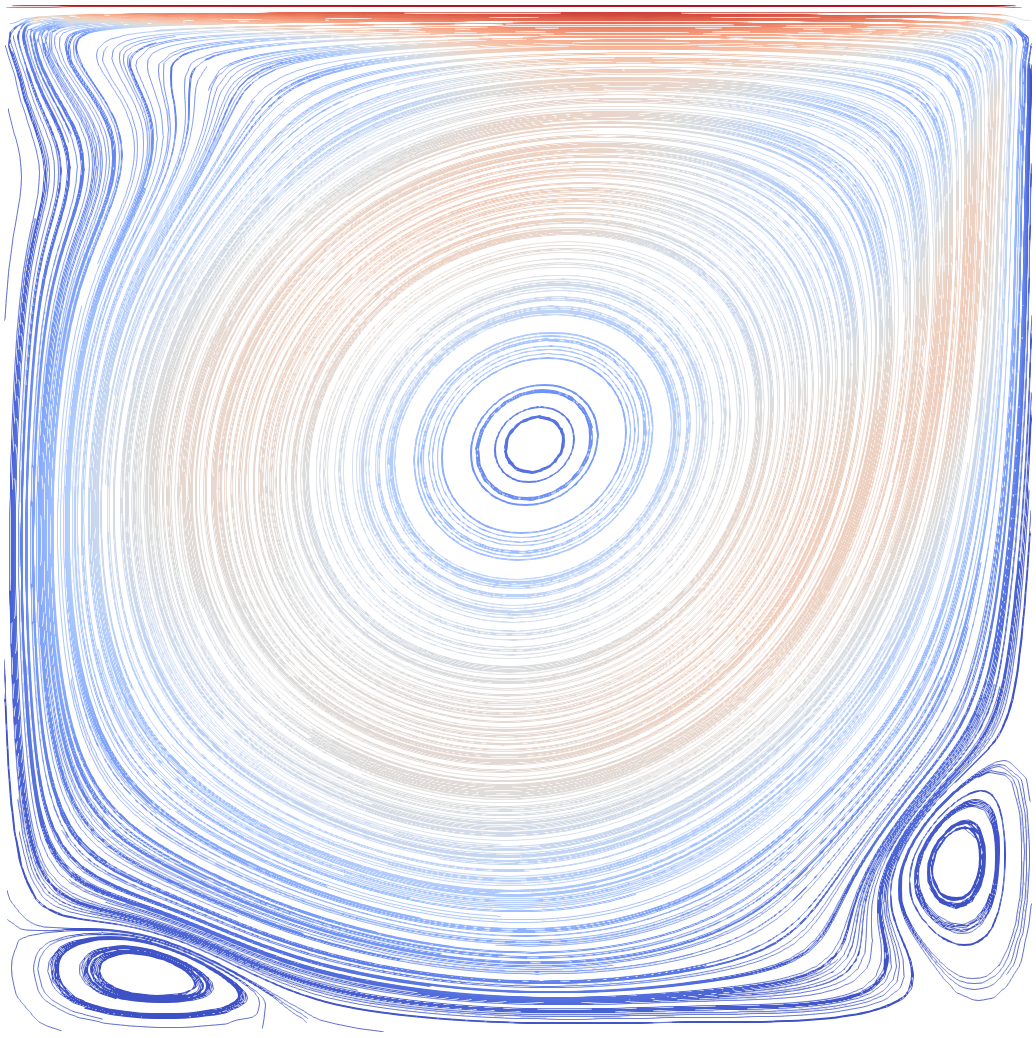}
\caption{Streamlines of the deterministic solution of the CN scheme ($\mu=0$) at $T=100$ (left), streamlines of the time-averaged solution of the CN scheme with $\mu=10$ (middle) and $\mu=40$ (right).}
\label{fig_lid_aver}
\end{figure}


It was Figure~\ref{fig_lid_expval_nu40} in the introduction where the expected value of the numerical solution at $T=20$ was computed over $1000$ realisations of the noise with $\mu=10$ and $\mu=40$, respectively,
along with the deterministic solution ($\mu=0$).
The deterministic solution and the solution for $\mu=10$ are similar, but the solution computed with $\mu=40$ differs significantly.

Figure~\ref{fig_cn_path_nu10} now displays the evolution of one realisation of the noise with $\mu=10$ for the CN scheme, which again differs from the deterministic dynamics here.
\begin{figure}
\includegraphics[width=0.32\textwidth]{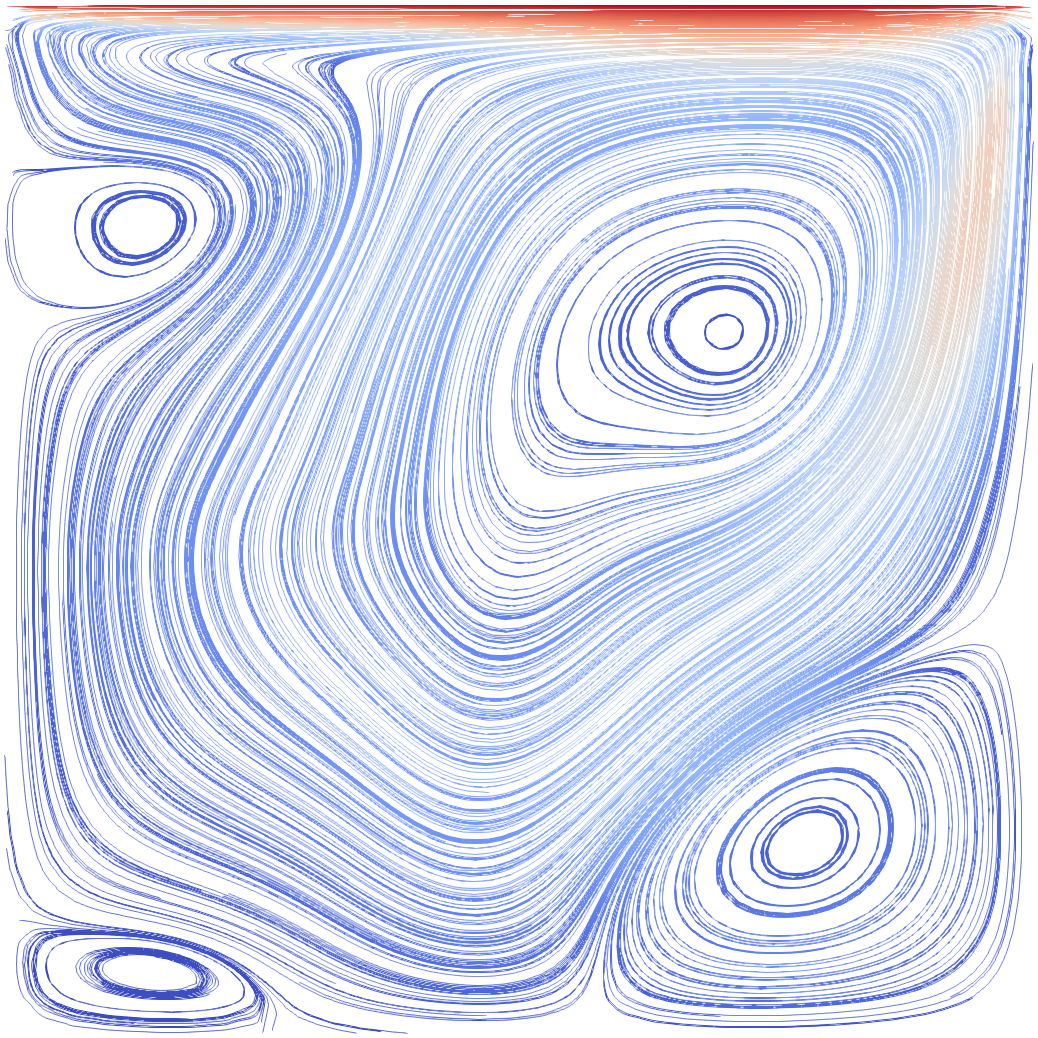}
\includegraphics[width=0.32\textwidth]{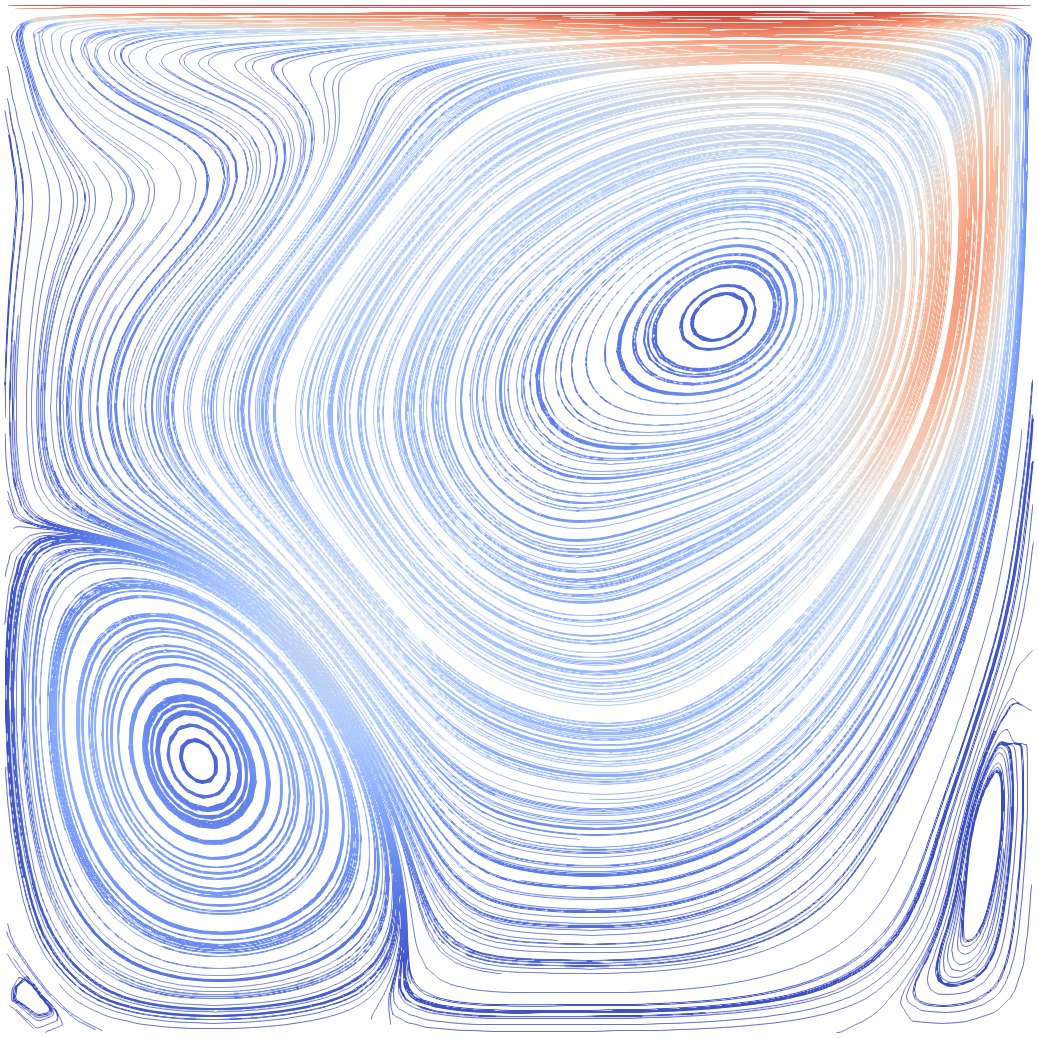}
\includegraphics[width=0.32\textwidth]{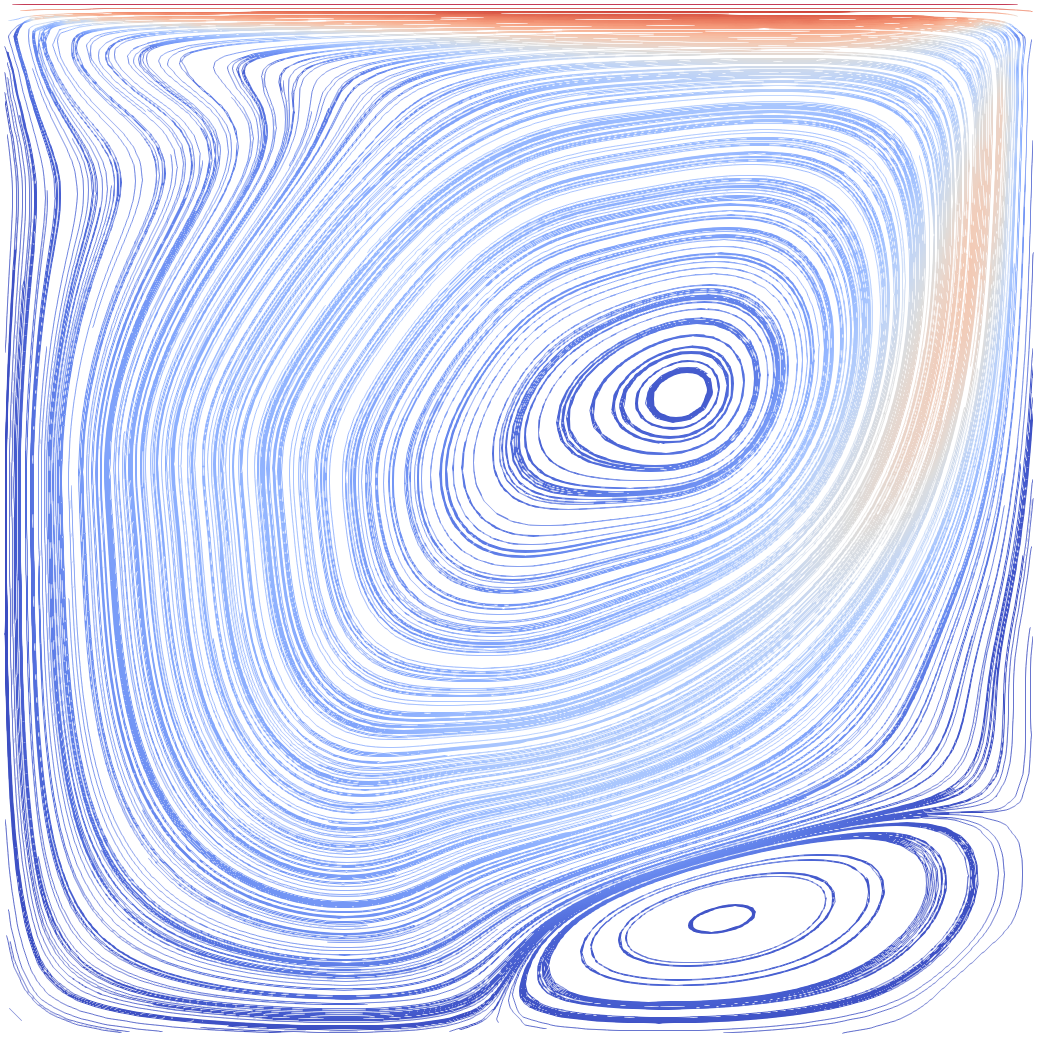}
\caption{Solution of the CN scheme for one realisation of the noise with $\mu=10$ at time $t=10,20,50$.}
\label{fig_cn_path_nu10}
\end{figure}
Even more different is the evolution of this realisation for the stronger noise with $\mu=40$ for the CN scheme and the SI scheme
in Figure~\ref{fig_cn_path} and Figure~\ref{fig_euler_path}, respectively. Both schemes yield qualitatively very similar results.
\begin{figure}
\includegraphics[width=0.32\textwidth]{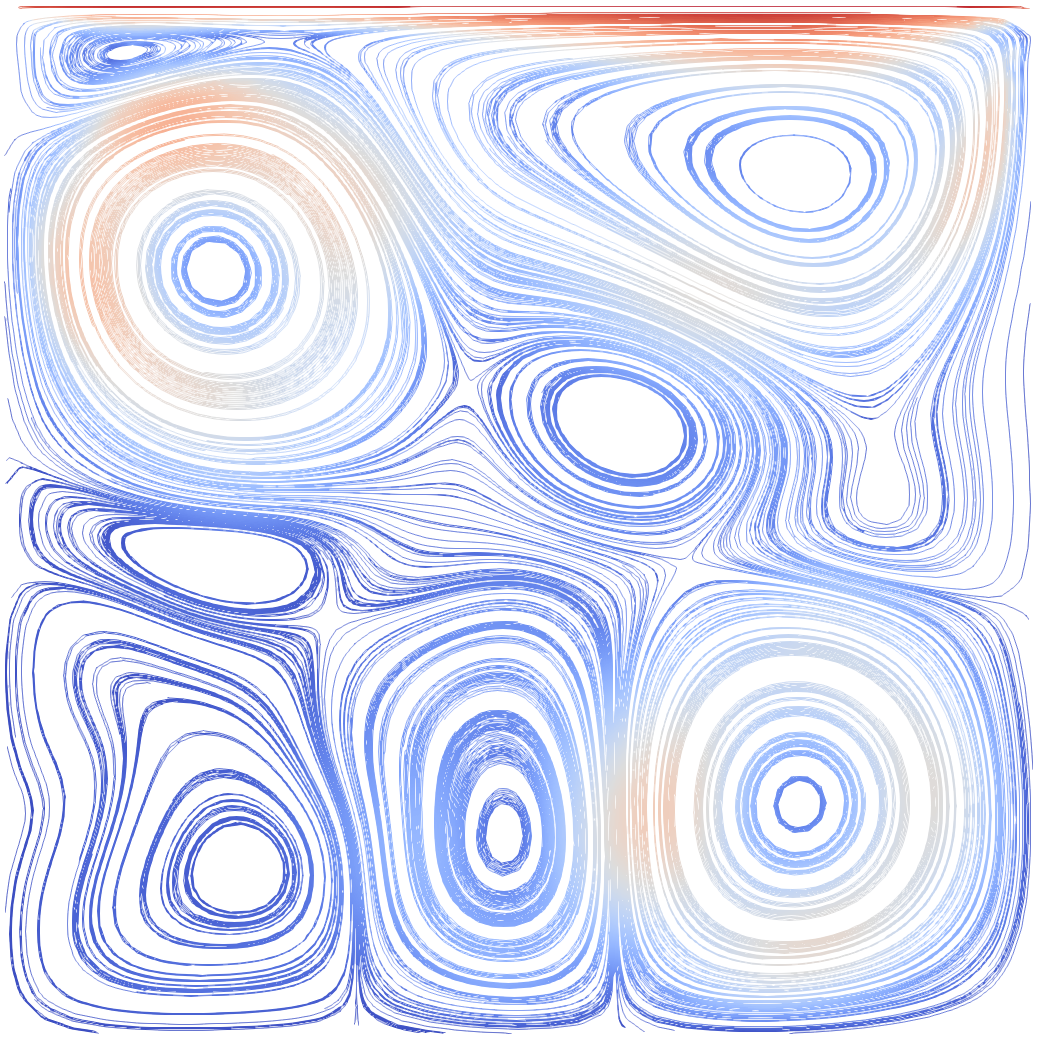}
\includegraphics[width=0.32\textwidth]{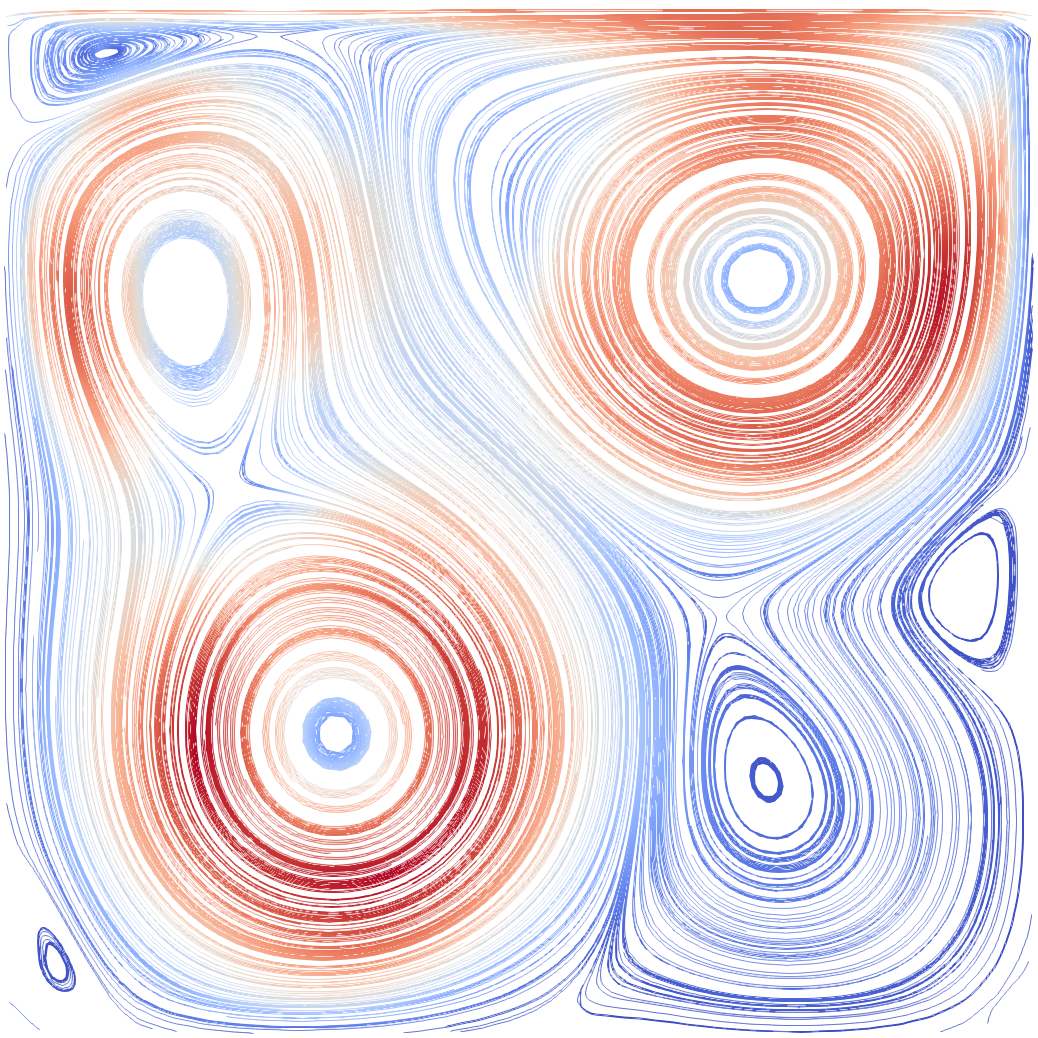}
\includegraphics[width=0.32\textwidth]{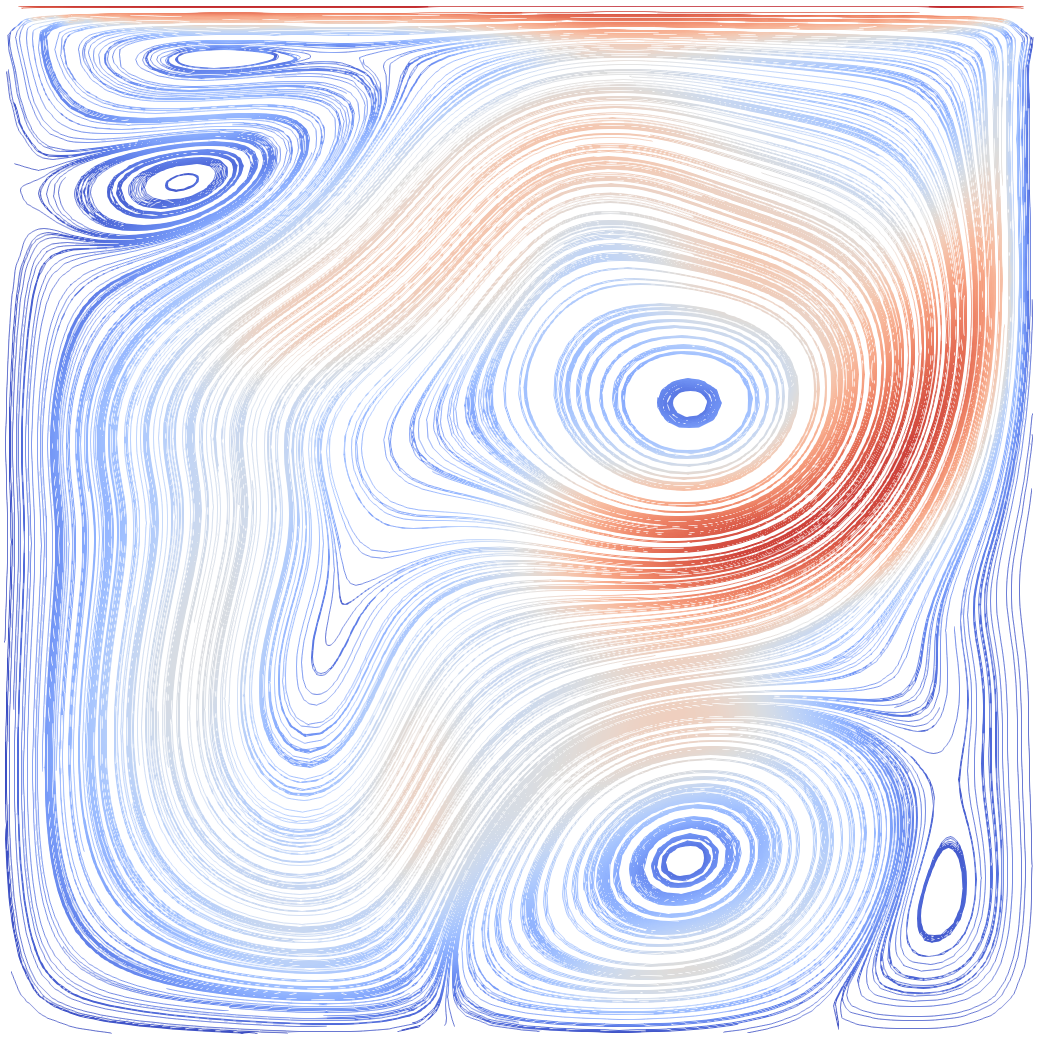}
\caption{Solution of the CN scheme for one realisation of the noise with $\mu=40$ at time $t=10,20,50$.}
\label{fig_cn_path}
\end{figure}

\begin{figure}
\includegraphics[width=0.32\textwidth]{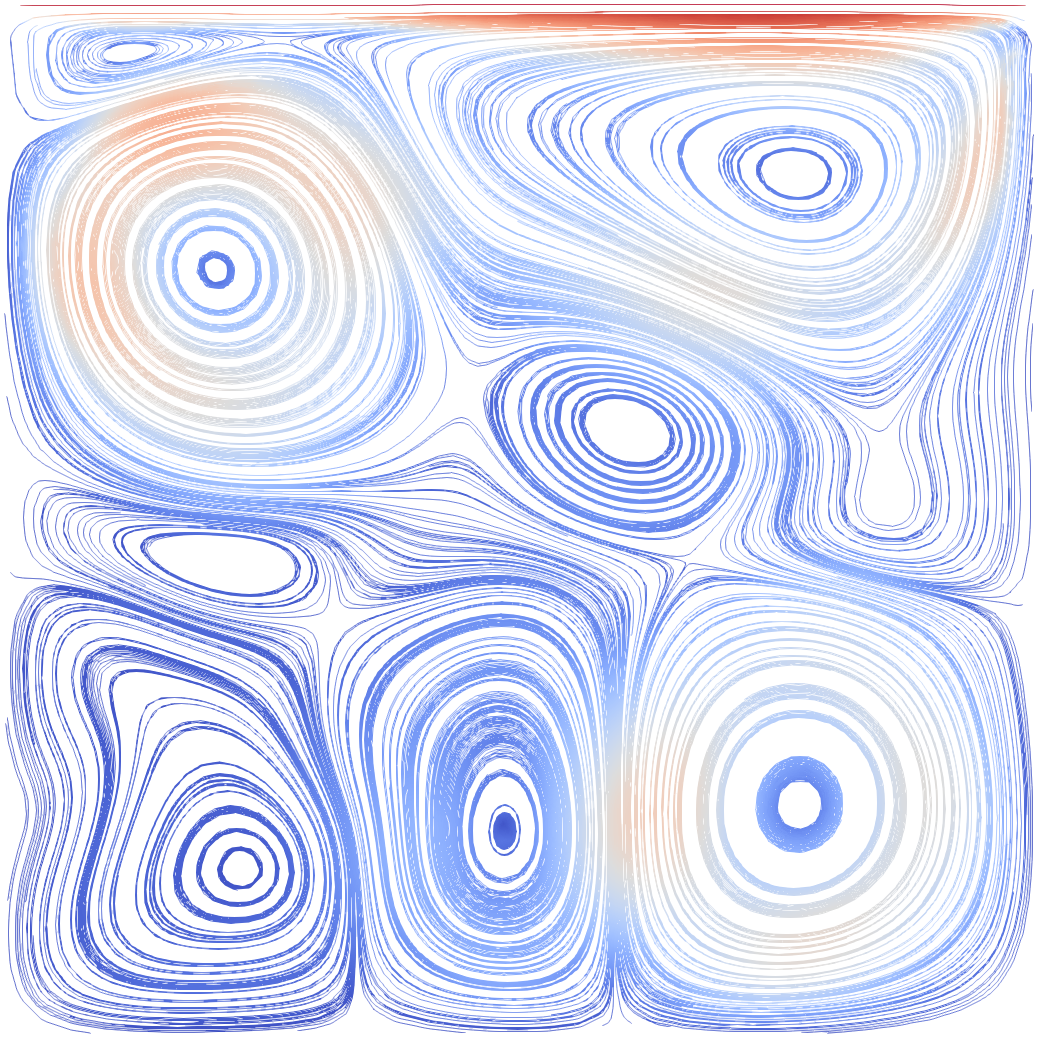}
\includegraphics[width=0.32\textwidth]{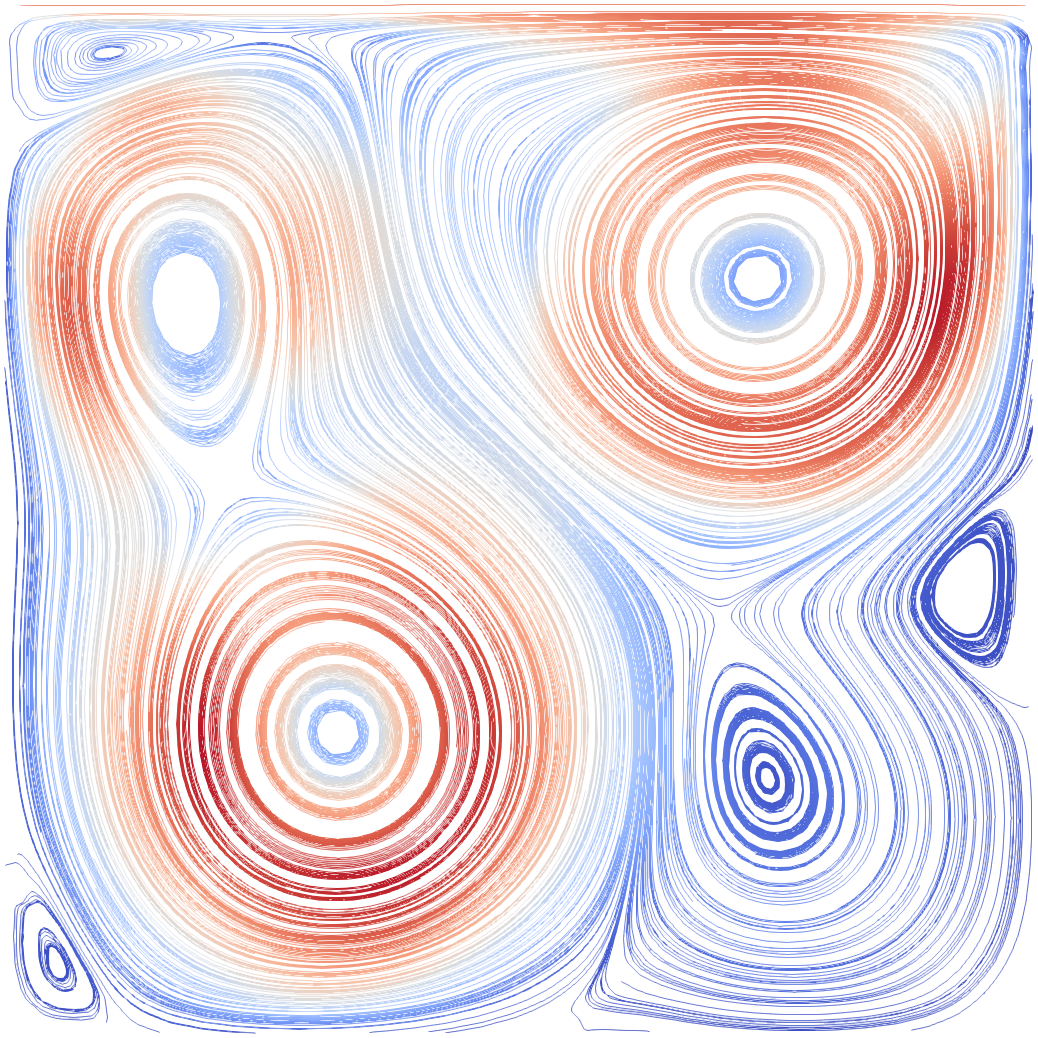}
\includegraphics[width=0.32\textwidth]{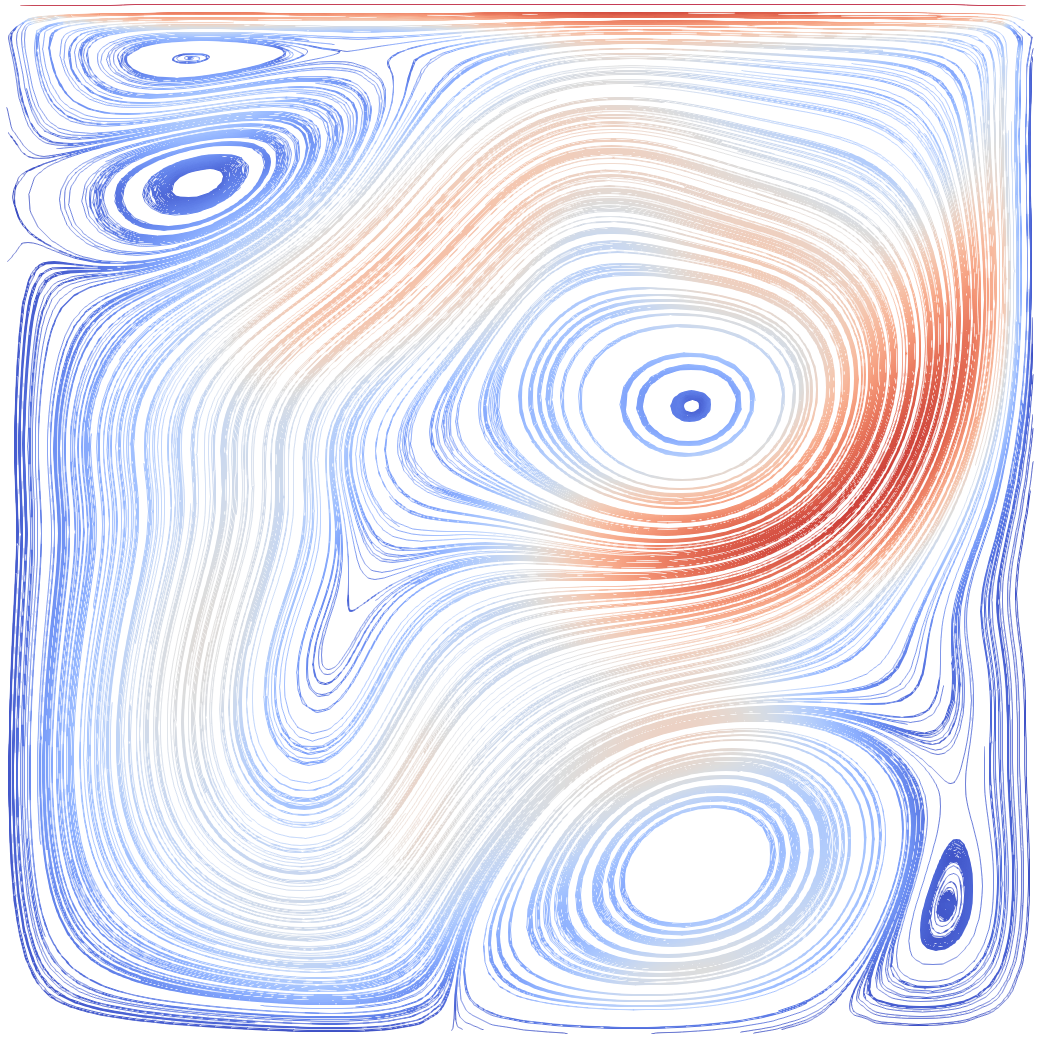}
\caption{Solution of the Euler scheme (SI) for one realisation of the noise with $\mu=40$ at time $t=10,20,50$.}
\label{fig_euler_path}
\end{figure}

\section*{Compliance with Ethical Standards}\label{conflicts}
\smallskip
\par\noindent
{\bf Conflict of Interest}. The author declares that he has no conflict of interest.

\smallskip
\par\noindent
{\bf Data Availability}. Data sharing is not applicable to this article as no datasets were generated or analysed during the current study.

\appendix
\section{Technical results}
\subsection{Peano-kernel proof of the midpoint extrapolation error}\label{app:peano}

\begin{lemma}\label{lem:deltastar-average}
Let $\mathrm{K}$ be a Banach space and let
\[
\delta_\star^{\,n+\frac12}
:= \frac{1}{\tau}\int_{t_{n}}^{t_{n+1}} {\bf y}(s)\,\mathrm{d}s
 \;-\;\Big(\tfrac32\,{\bf y}(t_{n})-\tfrac12\,{\bf y}(t_{n-1})\Big).
\]
Assume ${\bf y}_t\in C^{1/2}\!\big([0,T];\mathrm{K}\big)$, i.e.
\(
\|{\bf y}_t(t)-{\bf y}_t(s)\|_{\mathrm{K}}\le [{\bf y}_t]_{C^{1/2}([0,T]; \mathrm{K})}\,|t-s|^{1/2}
\) for all $s,t\in[0,T]$.
Then there exists a constant $C>0$, independent of $n$ and $\tau$, such that
\[
\boxed{\quad
\|\delta_\star^{\,n+\frac12}\|_{\mathrm{K}}\ \le\ C\,[{\bf y}_t]_{C^{1/2}([0,T]; \mathrm{K})}\,\tau^{3/2}.
\quad}
\]
\end{lemma}

\begin{proof}
We set $t_0:=t^{n+\frac12}=t_n+\tfrac{\tau}{2}$ and $h:=\tfrac{\tau}{2}$, so that $[t_{n},t_{n+1}]=[t_0-h,t_0+h]$.
We split
\[
\delta_\star^{\,n+\frac12}
= \underbrace{\Big(\frac{1}{\tau}\int_{t_0-h}^{t_0+h} {\bf y}(t)\,\mathrm{d}t - {\bf y}(t_0)\Big)}_{=:A}
\;+\;
\underbrace{\Big({\bf y}(t_0)-\tfrac32\,{\bf y}(t_0-h)+\tfrac12\,{\bf y}(t_0-3h)\Big)}_{=:B}.
\]
We will show $\|A\|_{\mathrm{K}}\le C\, [{\bf y}_t]_{C^{1/2}([0,T]; \mathrm{K})}\,\tau^{3/2}$ and
$\|B\|_{\mathrm{K}}\le C\, [{\bf y}_t]_{C^{1/2}([0,T]; \mathrm{K})}\,\tau^{3/2}$.

\smallskip\noindent\emph{Step 1: the mid--cell averaging defect $A$.}
By using the fundamental theorem of calculus,
\[
{\bf y}(t_0+\theta)-{\bf y}(t_0)=\int_{0}^{\theta} {\bf y}_t(t_0+s)\,\mathrm{d}s
= \theta\,{\bf y}_t(t_0)+\int_{0}^{\theta}\!\big({\bf y}_t(t_0+s)-{\bf y}_t(t_0)\big)\,\mathrm{d}s.
\]
By averaging over the symmetric interval and dividing by $\tau=2h$,
\[
A=\frac{1}{2h}\int_{-h}^{h}\!\big({\bf y}(t_0+\theta)-{\bf y}(t_0)\big)\,\mathrm{d}\theta
= \underbrace{\frac{1}{2h}\int_{-h}^{h}\!\theta\,{\bf y}_t(t_0)\,\mathrm{d}\theta}_{=\,0}
\;+\; \frac{1}{2h}\int_{-h}^{h}\int_{0}^{\theta}\!\big({\bf y}_t(t_0+s)-{\bf y}_t(t_0)\big)\,\mathrm{d}s\,\mathrm{d}\theta.
\]
Hence, by using the $C^{1/2}$ modulus of ${\bf y}_t$ in $\mathrm{K}$ and that
$\int_{0}^{\theta}\!=\ \mathrm{sgn}(\theta)\int_{0}^{|\theta|}$,
\[
\|A\|_{\mathrm{K}}
\;\le\; \frac{1}{2h}\int_{-h}^{h}\int_{0}^{|\theta|}
 [{\bf y}_t]_{C^{1/2}([0,T]; \mathrm{K})}\, s^{1/2}\,\mathrm{d}s\,\mathrm{d}\theta
= \frac{[{\bf y}_t]_{C^{1/2}([0,T]; \mathrm{K})}}{2h}\int_{-h}^{h}\!\frac{2}{3}\,|\theta|^{3/2}\,\mathrm{d}\theta.
\]
We compute the elementary integral:
\[
\int_{-h}^{h}\!|\theta|^{3/2}\,\mathrm{d}\theta
= 2\int_{0}^{h}\theta^{3/2}\,\mathrm{d}\theta
= \frac{4}{5}h^{5/2}.
\]
Therefore,
\[
\|A\|_{\mathrm{K}}
\;\le\; \frac{[{\bf y}_t]_{C^{1/2}([0,T]; \mathrm{K})}}{2h}\cdot \frac{2}{3}\cdot \frac{4}{5}h^{5/2}
= \frac{4}{15}\,[{\bf y}_t]_{C^{1/2}([0,T]; \mathrm{K})}\,h^{3/2}
= C\,[{\bf y}_t]_{C^{1/2}([0,T]; \mathrm{K})}\,\tau^{3/2}.
\]

\smallskip\noindent\emph{Step 2: the BDF2--midpoint defect $B$.}
By using backward integral representations from $t_0$,
\[
{\bf y}(t_0-h)={\bf y}(t_0)-\int_{0}^{h}{\bf y}_t(t_0-s)\,\mathrm{d}s,\qquad
{\bf y}(t_0-3h)={\bf y}(t_0)-\int_{0}^{3h}{\bf y}_t(t_0-s)\,\mathrm{d}s,
\]
we obtain
\[
\begin{aligned}
B
&= {\bf y}(t_0)-\tfrac32\Big({\bf y}(t_0)-\!\int_{0}^{h}{\bf y}_t(t_0-s)\,\mathrm{d}s\Big)
+\tfrac12\Big({\bf y}(t_0)-\!\int_{0}^{3h}{\bf y}_t(t_0-s)\,\mathrm{d}s\Big)\\
&= \int_{0}^{h}{\bf y}_t(t_0-s)\,\mathrm{d}s - \frac12\int_{h}^{3h}{\bf y}_t(t_0-s)\,\mathrm{d}s
= \int_{0}^{3h} w(s)\,{\bf y}_t(t_0-s)\,\mathrm{d}s,
\end{aligned}
\]
where the \emph{Peano kernel}
\[
w(s)=
\begin{cases}
1, & s\in[0,h],\\[2pt]
-\tfrac12, & s\in(h,3h],\\[2pt]
0, & \text{otherwise},
\end{cases}
\qquad\text{satisfies}\quad \int_{0}^{3h}w(s)\,\mathrm{d}s = h-\tfrac12(2h)=0.
\]
By subtract and add ${\bf y}_t(t_0)$ and use the zero moment of $w$:
\[
B=\int_{0}^{3h}\!w(s)\,\big({\bf y}_t(t_0-s)-{\bf y}_t(t_0)\big)\,\mathrm{d}s.
\]
Hence, by the $C^{1/2}$ modulus of ${\bf y}_t$ in $\mathrm{K}$,
\[
\|B\|_{\mathrm{K}}
\le [{\bf y}_t]_{C^{1/2}([0,T]; \mathrm{K})}\int_{0}^{3h}\!|w(s)|\, s^{1/2}\,\mathrm{d}s
= [{\bf y}_t]_{C^{1/2}([0,T]; \mathrm{K})}\!\left(\int_{0}^{h}\! s^{1/2}\,\mathrm{d}s
+\frac12\int_{h}^{3h}\! s^{1/2}\,\mathrm{d}s\right).
\]
We  compute the integrals:
\[
\int_{0}^{h}\! s^{1/2}\,\mathrm{d}s=\frac{2}{3}h^{3/2},\qquad
\frac12\int_{h}^{3h}\! s^{1/2}\,\mathrm{d}s=\frac{1}{3}\big((3h)^{3/2}-h^{3/2}\big).
\]
Therefore,
\[
\|B\|_{\mathrm{K}}\ \le\ C\,[{\bf y}_t]_{C^{1/2}([0,T]; \mathrm{K})}\,h^{3/2}
= C\,[{\bf y}_t]_{C^{1/2}([0,T]; \mathrm{K})}\,\tau^{3/2}.
\]

\smallskip\noindent\emph{Step 3: conclusion.}
By the triangle inequality,
\[
\|\delta_\star^{\,n+\frac12}\|_{\mathrm{K}}\ \le\ \|A\|_{\mathrm{K}}+\|B\|_{\mathrm{K}}
\ \le\ C\,[{\bf y}_t]_{C^{1/2}([0,T]; \mathrm{K})}\,\tau^{3/2}.
\]
This completes the proof.
\end{proof}

\subsection{Approximation of a matrix--valued Brownian triple integral}

Let $W$ be a cylindrical Wiener process let
\[
  \Phi \in L_2(\mathfrak U;L^2(\mt))
\]
be a given coefficient field. 
We then define the $\mathbb{L}^2(\mt)^2$--valued process
\[
  Z(t) := \Phi W(t)\in\mathbb{L}^2(\mt)^2,\qquad t\ge 0.
\]

On a macro--interval $[t_n,t_{n+1}]$ of length $\tau := t_{n+1}-t_n$ we
consider the (space--dependent) matrix--valued triple integral
\[
  B_n
  :=\tau^3\mathcal{Q}_n^{W^2}
  = \int_{t_n}^{t_{n+1}} \int_{t_n}^{t_{n+1}} \int_{t_n}^{t_{n+1}}
      \bigl(Z(t)-Z(s)\bigr)\otimes\bigl(Z(t)-Z(r)\bigr)\,
      \mathrm{d}s\, \mathrm{d}t\, \mathrm{d}r,
\]
which is an element of $\mathbb{L}^2(\mt)^2$. On $[t_n,t_{n+1}]$ we also introduce the fine (Brownian) time mesh
\[
  h := \tau^2,\qquad M := \tau^{-1},\qquad
  t_\ell := t_n+\ell h,\quad \ell=0,\dots,M,
\]
and define the corresponding triple Riemann sum
\[
  B_n^{\mathrm{disc}}
  :=\tau^3\mathcal{I}_n^{W^2}
  = h^3 \sum_{\ell=1}^{M}\sum_{k=1}^{M}\sum_{j=1}^{M}
      \bigl(Z(t_\ell)-Z(t_k)\bigr)\otimes
      \bigl(Z(t_\ell)-Z(t_j)\bigr)
  \;\in\;\mathbb{L}^2(\mt)^4.
\]

\begin{lemma}\label{lem:triple-BM}
There exists a constant $C_\Phi>0$, depending only on
$\|\Phi\|_{L_2(\mathfrak U;L^2(\mt))}$, such that
\[
  \tau^6\,
  \mathbb{E}\big[\|\mathcal{Q}_n^{W^2}-\mathcal{I}_n^{W^2}\|_{{\mathbb{L}^2({\mathbb{T}^2})^4}}^2\big]
  \;=\;
  \mathbb{E}\bigl[\|B_n-B_n^{\mathrm{disc}}\|_{{\mathbb{L}^2({\mathbb{T}^2})^4}}^2\bigr]
  \;\le\; C_\Phi\,\tau^9.
\]
Equivalently,
\[
  \mathbb{E}\big[\|\mathcal{Q}_n^{W^2}-\mathcal{I}_n^{W^2}\|_{{\mathbb{L}^2({\mathbb{T}^2})^4}}^2\big]
  \;\le\; C_\Phi\,\tau^3.
\]
\end{lemma}

\begin{proof}
By stationarity of Brownian motion we may, without loss of generality,
assume $t_n=0$ and $t_{n+1}=\tau$, and drop the index $n$. Thus we work
on $[0,\tau]$ and define
\[
  B := \int_0^\tau\int_0^\tau\int_0^\tau
        F(t,s,r)\, \mathrm{d}s\, \mathrm{d}t\, \mathrm{d}r,
  \qquad
  F(t,s,r) := (Z(t)-Z(s))\otimes(Z(t)-Z(r)).
\]
On $[0,\tau]$ we use the fine mesh
\[
  h := \tau^2,\qquad M := \tau^{-1},\qquad
  t_\ell := \ell h,\quad \ell=0,\dots,M,
\]
so that $B^{\mathrm{disc}}$ is defined in the same way as above.

\medskip\noindent
\textbf{Step 1: decomposition into micro-boxes.}
We partition $[0,\tau]^3$ into micro-boxes
\[
  R_{\ell k j}
  := (t_{\ell-1},t_\ell]\times(t_{k-1},t_k]\times(t_{j-1},t_j],
  \qquad \ell,k,j=1,\dots,M.
\]
Then
\[
  B = \sum_{\ell,k,j}
        \int_{R_{\ell k j}} F(t,s,r)\, \mathrm{d}s\, \mathrm{d}t\, \mathrm{d}r.
\]
We define the local quadrature error on each micro--box by
\[
  {\bf e }_{\ell k j}
  := \int_{R_{\ell k j}} \bigl(F(t,s,r)-F(t_\ell,t_k,t_j)\bigr)\,
     \mathrm{d}s\, \mathrm{d}t\, \mathrm{d}r
  \;\in\;\mathbb{L}^2(\mt)^4.
\]
Since
\[
  B^{\mathrm{disc}}
  = \sum_{\ell,k,j}\int_{R_{\ell k j}} F(t_\ell,t_k,t_j)\,
      \mathrm{d}s\, \mathrm{d}t\, \mathrm{d}r,
\]
we obtain the exact decomposition
\[
  B-B^{\mathrm{disc}} = \sum_{\ell,k,j} {\bf e }_{\ell k j}.
\]

\medskip\noindent
\textbf{Step 2: Hilbert--space estimate for $F(t,s,r)--F(t',s',r')$.} We now work entirely at the level of $\mathbb{L}^2(\mt)^2$, so that we do not mix space and time variables. We fix $(t,s,r),(t',s',r')\in[0,\tau]^3$ and set
\[
  U := Z(t)-Z(s)\in\mathbb{L}^2(\mt)^2,\qquad
  V := Z(t)-Z(r)\in\mathbb{L}^2(\mt)^2,
\]
\[
  U' := Z(t')-Z(s')\in\mathbb{L}^2(\mr)^2,\qquad
  V' := Z(t')-Z(r')\in\mathbb{L}^2(\mt)^2.
\]
Then
\[
  F(t,s,r) = U\otimes V,\qquad
  F(t',s',r') = U'\otimes V'.
\]
We write
\begin{align*}
  F(t,s,r)-F(t',s',r')
  &= U\otimes V - U'\otimes V'\\
  &= (U-U')\otimes V + U'\otimes(V-V') + (U-U')\otimes(V-V').
\end{align*}
By using the triangle inequality and H\"older's inequality,
\begin{align*}
  \|F(t,s,r)-F(t',s',r')\|_{{\mathbb{L}^2({\mathbb{T}^2})^4}}
  &\le \|U-U'\|_{\mathbb{L}^2(\mt)^2}\,\|V\|_{\mathbb{L}^2(\mt)^2} + \|U'\|_{\mathbb{L}^2(\mt)^2}\,\|V-V'\|_{\mathbb{L}^2(\mt)^2}\\&\qquad
       + \|U-U'\|_{\mathbb{L}^2}\,\|V-V'\|_{\mathbb{L}^2(\mt)^2}.
\end{align*}
We now bound the second moment of this quantity. By
$(x+y+z)^2\le 3(x^2+y^2+z^2)$ and Cauchy--Schwarz,
\begin{align*}
  \mathbb{E}\big[\|F(t,s,r)-F(t',s',r')\|_{{\mathbb{L}^2({\mathbb{T}^2})^4}}^2\big]
  &\le 3\Big(
      \mathbb{E}\big[\|U-U'\|_{\mathbb{L}^2(\mt)^2}^2\|V\|_{\mathbb{L}^2(\mt)^2}^2\big]
     +\mathbb{E}\big[\|U'\|_{\mathbb{L}^2(\mt)^2}^2\|V-V'\|_{\mathbb{L}^2(\mt)^2}^2\big]\\&\qquad
     +\mathbb{E}\big[\|U-U'\|_{\mathbb{L}^2(\mt)^2}^2\|V-V'\|_{\mathbb{L}^2(\mt)^2}^2\big]\Big)\\
  &\le 3\Big(
      (\mathbb{E}\|U-U'\|_{\mathbb{L}^2(\mt)^2}^4)^{1/2}(\mathbb{E}\|V\|_{\mathbb{L}^2(\mt)^2}^4)^{1/2}\\
  &\qquad\quad
     +(\mathbb{E}\|U'\|_{\mathbb{L}^2(\mt)^2}^4)^{1/2}(\mathbb{E}\|V-V'\|_{\mathbb{L}^2(\mt)^2}^4)^{1/2}\\
  &\qquad\quad
     +(\mathbb{E}\|U-U'\|_{\mathbb{L}^2(\mt)^2}^4)^{1/2}(\mathbb{E}\|V-V'\|_{\mathbb{L}^2(\mt)^2}^4)^{1/2}
     \Big).
\end{align*}
Since $Z(t)=\Phi W(t)$ and $\Phi\in L_2(\mathfrak U;L^2(\mt)^2)$,
$Z$ is a Gaussian process in the Hilbert space
$\mathbb{L}^2(\mt)^2$. Standard Gaussian moment
estimates yield
\begin{align*}
  \mathbb{E}\|Z(t)-Z(s)\|^4 &\le C_\Phi\,|t-s|^2,
\end{align*}
with $C_\Phi$ depending only on $\|\Phi\|_{L_2(\mathfrak U;L^2(\mt))}$.
By applying this to the various increments, we obtain
\begin{align*}
  \mathbb{E}\|U\|_{\mathbb{L}^2(\mt)^2}^4 &\le C_\Phi\,|t-s|^2, &
  \mathbb{E}\|U'\|_{\mathbb{L}^2(\mt)^2}^4 &\le C_\Phi\,|t'-s'|^2,\\
  \mathbb{E}\|V\|_{\mathbb{L}^2(\mt)^2}^4 &\le C_\Phi\,|t-r|^2, &
  \mathbb{E}\|V'\|_{\mathbb{L}^2(\mt)^2}^4 &\le C_\Phi\,|t'-r'|^2,
\end{align*}
and
\begin{align*}
  \mathbb{E}\|U-U'\|_{\mathbb{L}^2(\mt)^2}^4
  &= \mathbb{E}\big\|\bigl(Z(t)-Z(s)\bigr)-\bigl(Z(t')-Z(s')\bigr)\big\|_{\mathbb{L}^2(\mt)^2}^4
   \le C_\Phi\,\bigl(|t-t'|+|s-s'|\bigr)^2,\\[0.5ex]
  \mathbb{E}\|V-V'\|_{\mathbb{L}^2(\mt)^2}^4
  &= \mathbb{E}\big\|\bigl(Z(t)-Z(r)\bigr)-\bigl(Z(t')-Z(r')\bigr)\big\|_{\mathbb{L}^2(\mt)^2}^4
   \le C_\Phi\,\bigl(|t-t'|+|r-r'|\bigr)^2.
\end{align*}

Since $t,s,r,t',s',r'\in[0,\tau]$, we have $|t-s|,|t-r|,|t'-s'|,|t'-r'|\le\tau$,
so
\[
  (\mathbb{E}\|U\|_{\mathbb{L}^2(\mt)^2}^4)^{1/2}
  +(\mathbb{E}\|U'\|_{\mathbb{L}^2(\mt)^2}^4)^{1/2}
  +(\mathbb{E}\|V\|_{\mathbb{L}^2(\mt)^2}^4)^{1/2}
  +(\mathbb{E}\|V'\|_{\mathbb{L}^2(\mt)^2}^4)^{1/2}
  \le C_\Phi\,\tau,
\]
and
\[
  (\mathbb{E}\|U-U'\|_{\mathbb L^2(\mt)^2}^4)^{1/2}
  \le C_\Phi\bigl(|t-t'|+|s-s'|\bigr),\qquad
  (\mathbb{E}\|V-V'\|_{\mathbb{L}^2(\mt)^2}^4)^{1/2}
  \le C_\Phi\bigl(|t-t'|+|r-r'|\bigr).
\]
Plugging these bounds back into the previous estimate yields
\begin{align}\label{eq:F-diff-bound}
  \mathbb{E}\|F(t,s,r)-F(t',s',r')\|_{{\mathbb{L}^2({\mathbb{T}^2})^2}}^2
  &\le C_\Phi\,\tau\bigl(|t-t'|+|s-s'|+|r-r'|\bigr),
\end{align}
for all $(t,s,r),(t',s',r')\in[0,\tau]^3$.

\medskip\noindent
\textbf{Step 3: application on each micro--box.}
Fix $\ell,k,j$ and take $(t',s',r')=(t_\ell,t_k,t_j)$. For
$(t,s,r)\in R_{\ell k j}$ we have $|t-t_\ell|\le h$, $|s-t_k|\le h$,
$|r-t_j|\le h$, hence
\[
  |t-t'|+|s-s'|+|r-r'|
  \le 3h.
\]
Therefore, by \eqref{eq:F-diff-bound},
\[
  \mathbb{E}\|F(t,s,r)-F(t_\ell,t_k,t_j)\|_{{\mathbb{L}^2({\mathbb{T}^2})^4}}^2
  \le C_\Phi\,\tau h
  = C_\Phi\,\tau^3,
  \qquad (t,s,r)\in R_{\ell k j},
\]
since $h=\tau^2$.

\medskip\noindent
\textbf{Step 4: $L^2$--bound for the local error ${\bf e }_{\ell k j}$.}
By Jensen's inequality on the box $R_{\ell k j}$ (which has volume
$|R_{\ell k j}|=h^3$),
\begin{align*}
  \|{\bf e }_{\ell k j}\|_{{\mathbb{L}^2({\mathbb{T}^2})^4}}^2
  &= \Bigl\|\int_{R_{\ell k j}}
         \bigl(F(t,s,r)-F(t_\ell,t_k,t_j)\bigr)\,
         \mathrm{d}s\, \mathrm{d}t\, \mathrm{d}r\Bigr\|_{{\mathbb{L}^2({\mathbb{T}^2})^4}}^2
 \\& \le |R_{\ell k j}|\int_{R_{\ell k j}}
         \|F(t,s,r)-F(t_\ell,t_k,t_j)\|_{{\mathbb{L}^2({\mathbb{T}^4})^2}}^2\,
         \mathrm{d}s\, \mathrm{d}t\, \mathrm{d}r.
\end{align*}
Taking expectations and using the bound above gives
\[
  \mathbb{E}\|{\bf e }_{\ell k j}\|_{{\mathbb{L}^2({\mathbb{T}^2})^4}}^2
  \le h^3\int_{R_{\ell k j}} C_\Phi\,\tau^3\,
        \mathrm{d}s\, \mathrm{d}t\, \mathrm{d}r
  = C_\Phi\,h^6\tau^3.
\]
Recalling $h=\tau^2$, we obtain
\begin{equation}\label{98976}
  \mathbb{E}\|{\bf e }_{\ell k j}\|_{{\mathbb{L}^2({\mathbb{T}^2})^4}}^2
  \le C_\Phi\,\tau^{15}.
\end{equation}

\medskip\noindent
\textbf{Step 5: summation over all micro--boxes.}
There are $M^3 = \tau^{-3}$ boxes. By using the triangle inequality in
$\mathbb{L}^2(\Omega;\mathbb{L}^2(\mt)^{4})$,
\begin{align*}
  \|B-B^{\mathrm{disc}}\|_{\mathbb{L}^2(\Omega;{\mathbb{L}^2({\mathbb{T}^2})^4})}
 & = \Big\|\sum_{\ell,k,j}{\bf e }_{\ell k j}\Big\|_{\mathbb{L}^2(\Omega;{\mathbb{L}^2({\mathbb{T}^2})^4})}
  \le \sum_{\ell,k,j}\|{\bf e }_{\ell k j}\|_{\mathbb{L}^2(\Omega;{\mathbb{L}^2({\mathbb{T}^2})^4})}
 \\& = \sum_{\ell,k,j}\sqrt{\mathbb{E}\|{\bf e }_{\ell k j}\|_{{\mathbb{L}^2({\mathbb{T}^2})^4}}^2}.
\end{align*}
From \eqref{98976},
\[
  \sqrt{\mathbb{E}\|{\bf e }_{\ell k j}\|_{{\mathbb{L}^2({\mathbb{T}^2})^4}}^2}
  \le C_\Phi^{1/2}\,\tau^{15/2},
\]
hence
\[
  \sum_{\ell,k,j}\sqrt{\mathbb{E}\|{\bf e }_{\ell k j}\|_{{\mathbb{L}^2({\mathbb{T}^2})^4}}^2}
  \le M^3 C_\Phi^{1/2}\,\tau^{15/2}
  = \tau^{-3} C_\Phi^{1/2}\,\tau^{15/2}
  = C_\Phi^{1/2}\,\tau^{9/2}.
\]
By squaring both sides, we obtain
\[
  \mathbb{E}\|B-B^{\mathrm{disc}}\|_{{\mathbb{L}^2({\mathbb{T}^2})^2}}^2
  \le C_\Phi\,\tau^9.
\]
Finally, since $B=\tau^3\mathcal{Q}_n^{W^2}$ and
$B^{\mathrm{disc}}=\tau^3\mathcal{I}_n^{W^2}$, we have
\[
  B-B^{\mathrm{disc}}
  = \tau^3\big(\mathcal{Q}_n^{W^2}-\mathcal{I}_n^{W^2}\big),
\]
and therefore
\[
  \mathbb{E}\big[\|B-B^{\mathrm{disc}}\|_{{\mathbb{L}^2({\mathbb{T}^2})^4}}^2\big]
  = \tau^6\,\mathbb{E}\big[\|\mathcal{Q}_n^{W^2}-\mathcal{I}_n^{W^2}\|_{{\mathbb{L}^2({\mathbb{T}^2})^4}}^2\big]
  \le C_\Phi\,\tau^9.
\]
Dividing by $\tau^6$ gives
\[
  \mathbb{E}\big[\|\mathcal{Q}_n^{W^2}-\mathcal{I}_n^{W^2}\|_{{\mathbb{L}^2({\mathbb{T}^2})^4}}^2\big]
  \le C_\Phi\,\tau^3,
\]
which proves the lemma.
\end{proof}

\end{document}